\documentclass[
11pt,  reqno]{amsart}
\usepackage[margin=1.1in,marginparwidth=1.5cm, marginparsep=0.5cm]{geometry}

\usepackage{booktabs} 
\usepackage{microtype}
\usepackage{amssymb}
\usepackage{mathrsfs} 
\usepackage{upgreek} 
\usepackage{stix}
\usepackage{color}
\usepackage[implicit=true]{hyperref}
\usepackage{xsavebox}
\usepackage{cases}
\allowdisplaybreaks[2]
\usepackage{comment}

\makeatletter
\def\@tocline#1#2#3#4#5#6#7{\relax
	\ifnum #1>\c@tocdepth 
	\else
	\par \addpenalty\@secpenalty\addvspace{#2}%
	\begingroup \hyphenpenalty\@M
	\@ifempty{#4}{%
		\@tempdima\csname r@tocindent\number#1\endcsname\relax
	}{%
		\@tempdima#4\relax
	}%
	\parindent\z@ \leftskip#3\relax \advance\leftskip\@tempdima\relax
	\rightskip\@pnumwidth plus4em \parfillskip-\@pnumwidth
	#5\leavevmode\hskip-\@tempdima
	\ifcase #1
	\or\or \hskip 1em \or \hskip 2em \else \hskip 3em \fi%
	#6\nobreak\relax
	\dotfill
	\hbox to\@pnumwidth{\@tocpagenum{#7}}
	\par
	\nobreak
	\endgroup
	\fi}
\makeatother

\definecolor{gr}{rgb}   {0.,   0.69,   0.23 }
\definecolor{bl}{rgb}   {0.,   0.5,   1. }
\definecolor{mg}{rgb}   {0.85,  0.,    0.85}
\definecolor{yl}{rgb}   {0.8,  0.7,   0.}
\definecolor{or}{rgb}  {0.7,0.2,0.2}

\usepackage{tikz} 

\usepackage{marginnote}
\usepackage{scalerel} 

\usetikzlibrary{shapes.misc}
\usetikzlibrary{shapes.symbols}
\usetikzlibrary{decorations}
\usetikzlibrary{decorations.markings}

\newtheorem{theorem}{Theorem} [section]

\newtheorem{lemma}[theorem]{Lemma}
\newtheorem{proposition}[theorem]{Proposition}

\newtheorem{definition}[theorem]{Definition}
\newtheorem{corollary}[theorem]{Corollary}

\newtheoremstyle{remarkstyle}
{}{}{
}{ }{\bfseries}{.}{ }{\thmname{#1}\thmnumber{ #2}\thmnote{ (#3)}}
\theoremstyle{remarkstyle}
\newtheorem{remark}{Remark}[section]

\DeclareMathOperator{\GNS}{GNS}
\newcommand{\noi}{\noindent}

\newcommand{\R}{\mathbb{R}}
\newcommand{\T}{\mathbb{T}}

\let\Re=\undefined\DeclareMathOperator*{\Re}{Re}

\let\P= \undefined
\newcommand{\P}{\mathbf{P}}

\newcommand{\E}{\mathbb{E}}
\renewcommand{\L}{\mathcal{L}}

\newcommand{\F}{\mathcal{F}}

\newcommand{\al}{\alpha}
\newcommand{\be}{\beta}
\newcommand{\dl}{\delta}

\newcommand{\eps}{\varepsilon}

\newcommand{\g}{\gamma}
\newcommand{\G}{\Gamma}
\newcommand{\ld}{\lambda}
\newcommand{\Ld}{\Lambda}
\newcommand{\s}{\sigma}

\newcommand{\ft}{\widehat}

\newcommand{\wt}{\widetilde}
\newcommand{\cj}{\overline}

\DeclareMathSymbol{\wcol}{\mathord}{operators}{"3A}
\newcommand{\wick}[1]{\wcol#1\wcol}

\newcommand{\ta}{\theta}

\renewcommand{\o}{\omega}
\renewcommand{\O}{\Omega}
\newcommand{\les}{\lesssim}
\newcommand{\ges}{\gtrsim}

\newcommand{\jb}[1]
{\langle #1 \rangle}

\newcommand{\ind}{\mathbf 1}

\renewcommand{\S}{\mathcal{S}}

\newcommand{\N}{\mathbb{N}}

\renewcommand{\H}{\mathcal{H}}

\DeclareMathOperator*{\qua}{qu}
\DeclareMathOperator*{\cl}{cl}
\DeclareMathOperator*{\test}{test}

\newtheorem*{ackno}{Acknowledgements}

\newcommand{\PP}{\mathbb{P}}

\DeclareMathOperator{\Law}{Law}

\numberwithin{equation}{section}
\numberwithin{theorem}{section}

\newcommand{\cZ}{\mathcal{Z}}

\newcommand{\W}{\mathcal{W}}

\newcommand{\dr}{\theta}
\newcommand{\Dr}{\Theta}

\makeatletter
\def\subsubsection{\@startsection{subsubsection}{3}%
	\z@{.5\linespacing\@plus.7\linespacing}{-.5em}%
	{\normalfont\bfseries}}
\makeatother

\begin{document}
	\baselineskip = 14pt

	\title[NLS in general traps]{Statistical mechanics of the radial focusing nonlinear Schr\"odinger equation in general traps}

	\author[V.~D.~Dinh, N.~Rougerie, L.~Tolomeo, Y.~Wang]
	{Van Duong Dinh, Nicolas Rougerie, Leonardo Tolomeo, Yuzhao Wang}
	
	\address[V. D. Dinh]{Ecole Normale Supérieure de Lyon \& CNRS, UMPA (UMR 5669), Lyon, France}
	\email{contact@duongdinh.com}
	
	\address[N. Rougerie]{Ecole Normale Supérieure de Lyon \& CNRS, UMPA (UMR 5669), Lyon, France}
	\email{nicolas.rougerie@ens-lyon.fr}

	\address[L. Tolomeo]{ 
		School of Mathematics,
		The University of Edinburgh,
		James Clerk Maxwell Building, Room 5601,
		The King's Buildings, Peter Guthrie Tait Road,
		Edinburgh, EH9 3FD, United Kingdom}

	\email{l.tolomeo@ed.ac.uk}
	
	
	\address[Y. Wang]{
		School of Mathematics\\
		Watson Building\\
		University of Birmingham\\
		Edgbaston\\
		Birmingham\\
		B15 2TT\\ United Kingdom}
	
	\email{y.wang.14@bham.ac.uk}

	
	
	\begin{abstract}
		In this paper, we investigate the Gibbs measures associated with the focusing nonlinear Schr\"odinger equation with an anharmonic potential. We establish a dichotomy for normalizability and non-normalizability of the Gibbs measures in one dimension and higher dimensions with radial data. This extends a recent result of the third and fourth authors with Robert and Seong (2022), where the focusing Gibbs measures with a harmonic potential were addressed. Notably, in the case of a subharmonic potential, we identify a novel critical nonlinearity (below the usual mass-critical exponent) for which the Gibbs measures exhibit a phase transition. The primary challenge emerges from the limited understanding of eigenvalues and eigenfunctions of the Schr\"odinger operator with an anharmonic potential. We overcome the difficulty by employing techniques related to a recent work of the first two authors (2022).
	\end{abstract}
	
	\date{December 2023}

	\maketitle

	\vspace{-5mm}
	
	\tableofcontents

	\section{Introduction}
	
	This paper concerns the statistical mechanics of the focusing nonlinear Schr\"odinger equation with anharmonic potential
	\[
	{\rm i} \partial_t u + (\Delta -|x|^s) u = -  \alpha |u|^{p-2} u
	\]
	on the Euclidean space $\mathbb R^d, d\geq 1$, restricted to radial functions for $d\geq 2$. We will assume that $s>1$, $p>2$ and $\alpha>0$. More precisely, we study the integrability and non-integrability of the associated, formally time-invariant, Gibbs measures given by
	\begin{align}
		\label{gibbs}
		d\rho_K (u) = \mathcal Z_K^{-1} \ind_{\{|M(u) | \le K\}} e^{ \al R_p(u)} d\mu(u),
	\end{align}
	where $\mu$ is the Gaussian measure associated with the anharmonic operator
	\begin{align}
		\label{anharmonic}
		\mathcal L = -\Delta + |x|^s
	\end{align}
	given formally by 
	\begin{align}
		\label{gaussian}
		d \mu (u)  = \mathcal Z^{-1} e^{- \frac12 \jb{\L u, u} } du =  \mathcal Z^{-1} e^{- \frac12 \int_{\R^d} |\nabla u(x)|^2 + |x|^s|u(x)|^2 dx } du
	\end{align}
	with a normalization constant $\mathcal Z$. The potential energy $R_p (u)$ is defined by
	\begin{align}
		\label{Rp}
		R_p (u) : = {\frac{1}{p} \int_{\R^d} |u(x)|^p dx}.
	\end{align}
	Since we consider the focusing case $\alpha >0$, we must impose a mass cut-off $|M(u) | \le K$ for some parameter $K>0$ in~\eqref{gibbs}, where 
	\[
	M(u) = \int_{\R^d} |u(x)|^2 dx
	\]
	is the (conserved) $L^2$ mass if $s>2$. If $s\leq 2$, this is infinite $\mu$-almost surely, and hence we consider the Wick-ordered version
	\[
	M(u) = \int_{\R^d} \wick{|u(x)|^2} dx
	\]
	i.e. formally the $L^2$-mass minus its' expectation value in the Gaussian measure.
	
	Our main purpose is to identify the range of values for the parameters $(s,p,\al)$ that delineates whether the Gibbs measure \eqref{gibbs} is well-defined or not.
	
	This paper is a continuation of a recent work of the third and fourth authors with Robert and Seong \cite{RSTW22} where the normalizability/non-normalizability of the focusing Gibbs measure with harmonic potential ($s=2$) were investigated. More precisely, the following result was proved in \cite{RSTW22}.
	
	\begin{theorem}[\textbf{The harmonic oscillator case, \cite{RSTW22}}] \label{THM:RSTW}\mbox{}\\
		Let $d\geq 1$, $s=2$ and restrict the measures to radial functions when $d\geq 2$. Then, the following statements hold\textup{:}
		
		\smallskip
		
		\noi
		\begin{itemize}
			\item
			[\textup{(i)}] \textup{(subcritical case)} 
			If $2< p< 2+\frac{4}{d}$, then for any $K > 0$, the focusing Gibbs measure
			$$
			d\rho_K(u) = \mathcal Z_K^{-1} \ind_{\{|\int_{\R^d} \wick{|u(x)|^2} dx| \le K\}} e^{ \frac{1}{p}\|u\|^p_{L^p(\R^d)}} d\mu(u)  
			$$
			is well-defined as a probability measure and it is absolutely continuous with respect to the base Gaussian free field $\mu$.
			
			\item
			[\textup{(ii)}] \textup{(critical/supercritical cases)} 
			If $p\ge 2+\frac{4}{d}$ and $p<\frac{2d}{d-2}$ when $d\ge 3$, then for any $K > 0$, the focusing Gibbs measure
			$$
			d\rho_K(u) = \mathcal Z_K^{-1} \ind_{\{|\int_{\R^d} \wick{|u(x)|^2} dx| \le K\}} e^{ \frac{1}{p}\|u\|^p_{L^p(\R^d)}} d\mu(u)  
			$$
			is not well-defined as a probability measure.
		\end{itemize}
	\end{theorem}
	
	In the present paper we focus on the case where the parameter $s$ exceeds 1 but is not equal to 2 (although our approach can lead to simplifications in the latter case). Our main findings can be divided into two distinct cases: the superharmonic case where $s > 2$ and the subharmonic case where $s < 2$. 
	
	In the case of superharmonic potentials, we identify the critical nonlinearity as 
	$$p_{s > 2}:=2+\frac{4}{d}.$$
	The Gibbs measure \eqref{gibbs} is normalizable for any $\al, K>0$ provided that the power nonlinearity is subcritical, i.e., $p<p_{s > 2}$. On the other hand, it is non-normalizable for any $\alpha, K>0$ as long as the power nonlinearity is supercritical and below the energy-critical exponent, i.e., $p > p_{ s> 2}$ and $p<\frac{2d}{d-2}$ if $d\geq 3$. For the critical nonlinearity $p = p_{s> 2}$ we observe a phase transition: the Gibbs measure is well-defined for $\al^{\frac{d}{2}} K <\|Q\|^2_{L^2(\mathbb R^d)}$, and it is not well-defined for $\al^{\frac{d}{2}} K > \|Q\|^2_{L^2(\mathbb R^d)}$. Here $Q$ is an optimizer of the Gagliardo-Nirenberg-Sobolev inequality on $\mathbb R^d$. 
	
	In the case of subharmonic potentials, we identify a new critical nonlinearity 
	$$p_{s < 2}:=2+\frac{4s}{(d-1)s+2}$$
	which is below the usual mass-critical one above. The Gibbs measure is again normalizable when the power nonlinearity is subcritical and is non-normalizable for supercritical power nonlinearity regardless of the values of $\al, K>0$. For the critical nonlinearity, the phase transition is characterized in terms of the nonlinear strength $\al$: for each $K>0$, there exists $\al_0=\al_0(K)>0$ such that the Gibbs measure is well-defined for all $\alpha <\al_0$, and it is not well-defined for all $\alpha>\al_0$.

	\subsection{Known results and motivation}
	
	The construction of Gibbs measures associated with focusing nonlinear Schr\"odinger equations was initiated by Lebowitz, Rose and Speer \cite{LRS88}. They considered the focusing NLS on the one-dimensional torus and proposed to study the Gibbs measure with a mass cutoff, namely
	$$
	d\rho_K(u) = \mathcal Z^{-1}_K \ind_{\{|\int_{\T} |u|^2 dx| \leq K\}} e^{\frac{1}{p} \int_{\T} |u|^p dx} d\mu(u),
	$$ 
	where $\mathcal Z_K$ is the normalization constant (often referred to as the partition function). The imposition of a mass cutoff is reasonable since the mass is a conserved quantity under the NLS flow. 
	
	It was asserted in \cite{LRS88} that the aforementioned measure is normalizable under two distinct conditions: when $2<p<6$ for any possitive mass cutoff $K$, and when $p=6$ for $K$ smaller than the mass of $Q$ -- the unique (up to symmetries) optimizer of the Gagliardo-Nirenberg-Sobolev inequality
	$$
	\int_{\mathbb R} |u|^6 dx \leq C_{\GNS} \left(\int_{\mathbb R} |u'|^2 dx\right) \left(\int_{\mathbb R} |u|^2 dx\right)^2
	$$
	such that $C_{\GNS}=2\|Q\|^2_{L^2(\R)}$. They also proved that the Gibbs measure is not well-defined when $p=6$ and $K$ exceeds the mass of $Q$, and when $p>6$ regardless of the mass cutoff size $K$. 
	
	While the proof of normalizability presented in \cite{LRS88} uses an elegant probabilistic argument, it however contains a gap as pointed out and rectified for $p<6$ in \cite{CFL16}. Later, Bourgain \cite{BO94} gave an analytic proof for the normalizability for $2<p<6$ with any positive mass cutoff size $K$, and for $p=6$ with sufficiently small $K$. He also proved the invariance of these measures under the NLS dynamics. Recently, the third author together with Oh and Sosoe in \cite{OST22} proved the normalizability when $p=6$ and $K$ is smaller than the mass of $Q$, which resolves the issue in \cite{LRS88}. Remarkably, they were also able to prove the normalizability with the mass cutoff $K$ exactly equal to the mass of $Q$. This result is rather surprising since the focusing quintic NLS on the 1D torus admits blow-up solutions with this minimal mass, as found by Ogawa and Tsutsumi \cite{OT90}. Essentially, this implies that the Gibbs measure lives on Sobolev spaces of low regularity on which there are no blow-up solutions at the critical mass threshold. The result of \cite{OST22} shows a {\it phase transition} for the focusing Gibbs measure on the one-dimensional torus at the critical nonlinearity. Specifically, the partition function $\mathcal Z_K$ is not analytic with respect to $K$ when $p=6$.
	
	Similar results hold for the Gibbs measure
	$$
	d\rho_K(u) = \mathcal Z_K^{-1} \ind_{\{|\int_{\mathbb D} |u|^2 dx| \leq K\}} e^{\frac{1}{p}\int_{\mathbb D} |u|^p dx} d\mu(u).
	$$
	associated with the focusing NLS on the two-dimensional unit disc with Dirichlet boundary condition, restricted to radial functions. 
	
	In \cite{TZ06}, Tzvetkov constructed and proved the invariance of the focusing Gibbs measure with the subcritical nonlinearity $p<4$ and any positive $K$. Later, Bourgain and Bulut \cite{BB14} extended Tzvetkov's result to the critical nonlinearity $p=4$ and sufficiently small mass cutoff $K$. In a subsequent work \cite{OST22}, the third author, in collaboration with Oh and Sosoe, demonstrated the normalizability of the Gibbs measure when $p=4$ and $K$ is less than the mass of $Q$ -- the (positive and radial) ground state solution of 
	$$
	-\Delta Q + Q - Q^3=0 \quad \text{in } \mathbb R^2.
	$$
	They also established the non-normalizability of the Gibbs measure when $p=4$ and $K$ exceeds the mass of $Q$, and when $p>4$ no matter the size of the mass cutoff $K$. Recently, Xian \cite{XI22} proved the optimal mass normalizability in the critical case, that is, the above Gibbs measure is normalizable when $p=4$ and $K$ equals the mass of $Q$.
	
	Concerning the Gibbs measures associated with focusing NLS with potential on the real line, the pioneering work was undertaken by Burq, Thomann and Tzvetkov \cite{BTTz13}, where they studied the focusing Gibbs measure for the cubic nonlinearity with harmonic potential $V(x)=|x|^2$. In this setting, the Gibbs measure is slightly different to the cases discussed earlier on the 1D torus and 2D unit disc:
	$$
	d\rho_K(u) = \mathcal Z^{-1}_K \ind_{\{| \int_{\mathbb R^d} :|u|^2: dx | \leq K \}} e^{\frac{1}{p} \int_{\mathbb R^d} |u|^p dx} d\mu(u).
	$$
	Here a Wick-ordered renormalized mass is employed instead of the usual mass since the latter is infinite on the support of the Gaussian measure. 
	
	In one dimension, Burq, Thomann and Tzvetkov \cite{BTTz13} demonstrated the normalizability for $p=4$ and any positive $K$. They also proved the invariance of this measure under the dynamics of the nonlinear Schr\"odinger equation. In the two-dimensional case, the focusing Gibbs measure with harmonic potential was investigated by Deng~\cite{Deng12} under a radial assumption. More precisely, he constructed the focusing Gibbs measure for any $2<p<4$ and any positive $K$ and proved its invariance under the NLS flow. More recently, the last two authors, in collaboration with Robert and Seong, revisited the construction of focusing Gibbs measure with harmonic potential. By exploiting the so-called Bou\'e-Dupuis variational formula and refined stochastic analysis, they proved the focusing measure is normalizable when $2<p<2+\frac{4}{d}$, but non-normalizable when $p\geq 2+\frac{4}{d}$ and $p<\frac{2d}{d-2}$ if $d\geq 3$ regardless of the mass cutoff size $K$ (see Theorem \ref{THM:RSTW}). Their result completes the picture of focusing Gibbs measures with harmonic potential in one dimension and higher dimensions for radial data.  Notably, a distinguishing feature from the cases of the 1D torus or 2D unit disc is the absence of a critical nonlinearity, thus precluding a phase transition. 
	
	One advantage of considering the harmonic potential is that eigenfunctions and eigenvalues of the linear operator $-\Delta +|x|^2$ are explicit, which allows one to obtain precise $L^p$-estimates of the eigenfunctions (see e.g., \cite{KT05} and \cite{IRT16}). However, this advantageous feature is no longer accessible when dealing with a general anharmonic potential $|x|^s$ with $s \ne 2$.
	
	In \cite{DR23}, the first two authors investigated the one-dimensional focusing Gibbs measure with a potential exhibiting a growth of $|x|^s$ at infinity. They successfully overcame the aforementioned difficulty by employing techniques from many-body quantum mechanics, specifically the application of a Lieb-Thirring type inequality and operator-inequalities in Schatten ideals, which originated from \cite{LNR18}. Instead of relying on
	$L^p$-estimates for eigenfunctions, which are unavailable in this context, they observed that the construction of the focusing Gibbs measure hinged on $L^p$-estimates for the (diagonal part of the) Green function of $-\Delta + V(x)$:
	$$
	G(x,x) = \sum_{n\geq 0} \lambda_n^{-2} e^2_n(x),
	$$
	where $(e_n, \lambda_n^2)_{n\geq 0}$ represent the (normalized) eigenfunctions and eigenvalues of $-\Delta + V(x)$. Acquiring this information proved to be considerably more tractable, thanks to the application of standard inequalities like H\"older and Kato-Seiler-Simon within Schatten spaces. With these methods, they could define the focusing Gibbs measure with a cubic nonlinearity $p=4$ for any positive $K$, provided that $s>\frac{8}{5}$. Note that the potential energy $\frac{1}{4}\int_{\mathbb R} |u|^4 dx$ is finite almost surely with respect to the Gaussian measure as long as $s>1$. Thus there is a gap between 1 and $\frac{8}{5}$ for the measure construction. This gap is technical due to the use of a fractional Gagliardo-Nirenberg inequality (see \cite[Lemma 3.8]{DR23}). In this paper, we will fill this gap and prove the normalizability of focusing Gibbs measure for all $s>1$. We also aim at extending these result to higher dimensional radial NLS equations. Upon a suitable change of variables, the radial Schr\"odinger operator with an anharmonic potential on $\mathbb{R}^d$ transforms into a one-dimensional Schr\"odinger operator with an inverse square potential plus the anharmonic potential on $(0,+\infty)$. The appearance of the inverse square potential necessitates new approaches, as standard inequalities in Schatten spaces, as used in \cite{LNR18, DR23}, are no longer applicable. Hence, novel arguments are required to address this potential term.
	
	In addition to the previously discussed works on nonlinear Gibbs measures, we would like to mention related research on the probabilistic theory of Schrödinger operators with trapping potential (\cite{PoiRobTho-14} and \cite{RobTho-15}). Additionally, there are notable works on the deterministic theory for the nonlinear Schrödinger equation with harmonic potential (see e.g., \cite{KilVisZha-09}, \cite{Jao-16}, and a series of studies by Carles (\cite{Car-02, Car-03, Car-05, Car-11}).
	
	Another motivation for this study comes from the mean-field approximation of Bose gases and Bose-Einstein condensates. More specifically, the Gibbs measure linked with the nonlinear Schr\"odinger equation was rigorously derived from many-body quantum mechanics (see \cite{RS22, RS23} for the 1D focusing NLS,~\cite{LNR15, LNR18, FKSS19,Sohinger-22} for the 1D defocusing NLS and~\cite{LewNamRou-20,FKSS17,FroKnoSchSoh-20b,FroKnoSchSoh-22} for higher dimensions). In~\cite{RS22}, Rout and Sohinger rigorously derived the Gibbs measure associated with the focusing cubic NLS on the 1D torus, considering any mass cutoff size $K$. More recently, in their subsequent work \cite{RS23}, they successfully extended this result to the focusing quintic NLS provided that the mass cutoff size is small. It is expected that some form of phase transition, as proved in \cite{OST22}, should happen at the many-body level provided that the number of particles is close to a critical value. However, this remains an open question.

	\subsection{Measure construction and main results}
	
	In this subsection, we go over the construction of the Gibbs measures \eqref{gibbs}, and state our main results precisely.  
	
	Let us start by recalling some basic properties of the operator \eqref{anharmonic}. It is known (under the radial assumption when $d\geq 2$) that $\L$ has a sequence of eigenvalues $\ld_n^2$ with 
	\[
	0 < \ld_0 \le \ld_1 \le \cdots \le \ld_n \to \infty 
	\]
	and that the corresponding normalized eigenfunctions $e_n$, i.e.
	\[
	\L e_n = \ld_n^2 e_n,
	\]
	form an orthonormal basis of $L^2(\R^d)$. The radial assumption is invoked hereafter whenever the spatial dimension is two or larger. We also use the liberty to choose a real-valued eigenbasis. More details on the radial Schr\"odinger operator are given in Section \ref{sec:Schro} below.
	
	We define Sobolev spaces associated with the operator $\L$ as follows.
	
	\begin{definition}[\textbf{Sobolev spaces}]\mbox{}\\
		\label{DEF:sob}
		For $1\le q \le \infty$
		and $\s \in \R$,
		the Sobolev space $\W^{\s,q} (\R^d)$ is defined by the norm
		\[
		\| u\|_{\W^{\s,q} (\R^d)} = \| \L^{\frac{\s}2} u\|_{L^q (\R^d)}.
		\]
		
		\noi
		When $p=2$,
		we write $\W^{\s,2} (\R^d) = \H^{\s} (\R^d)$
		and for $u = \sum_{n=0}^\infty u_n e_n$ we have
		$$\|u \|_{\H^{\s} (\R^d)}^2 = \sum_{n=0}^\infty \ld^{2\s}_n |u_n|^2.$$
	\end{definition}

	With the above notation, we define the Hamiltonian as
	\begin{align} \label{Hamil}
		H(u)  = \frac12 \int_{\R^d} |\L^{\frac12} u(x) |^2 dx - \frac{\al}p \int_{\R^d} |u(x)|^p dx .
	\end{align}
	
	\noi
	In particular, using the eigenbasis $\{e_n\}_{n\ge 0}$, we can decompose any $u\in\S'(\R^d)$ as 
	$$u = \sum_{n= 0}^\infty u_n e_n, \qquad u_n=\langle u, e_n\rangle = \int_{\R^d} u(x)e_n(x)dx.$$
	
	\noi
	Then, in the coordinates $u = (u_n)_{n\geq 0}$, the Hamiltonian \eqref{Hamil} has the form
	\[
	H(u) = H \Big(\sum_{n=0}^\infty u_n e_n \Big) = \frac12 \sum_{n=0}^\infty \lambda_n^2 |u_n|^2
	- \frac{\al}p \int_{\R^d} \Big| \sum_{n=0}^\infty u_n e_n(x) \Big|^p dx.
	\]
	From the above computation,
	we may define the Gaussian measure with the Cameron-Martin space $\mathcal H^1(\R^d)$ formally given by
	\begin{align}
		d\mu = \mathcal Z^{-1} e^{-\frac12 \|u\|_{\mathcal H^1}^2} du = \mathcal Z^{-1} \prod_{n=0}^\infty e^{-\frac12 \ld^2_n |u_n|^2} d u_n d \cj u_n,
		\label{Gaussian}
	\end{align}
	
	\noi
	where $du_n d \cj u_n$ is the Lebesgue measure on $\mathbb C$.
	We note that this Gaussian measure $\mu$ is the induced probability measure under the map
	\begin{align}
		\label{maps}
		\o \in \O \longmapsto u^\o = \sum_{n= 0}^\infty \frac{g_n (\o)}{\ld_n} e_n,
	\end{align}
	
	\noi
	where $\{g_n\}_{n \ge 0}$ is a sequence of independent standard complex-valued Gaussian random variables on a probability space $(\O, \F, \PP)$.
	From \eqref{maps} and \eqref{traceN}, we see that
	\begin{align} \label{L2}
		\E \big[\| u^\o\|^2_{L^2 (\R^d)} \big] = \sum_{n=0}^\infty \ld_n^{-2}  \begin{cases}
			<\infty &\text{if } s>2, \\
			= \infty &\text{if } s < 2.
		\end{cases} 
	\end{align}
	
	\noi
	This implies that a typical function $u$ in the support of $\mu$ is not square integrable when $s < 2$.
	On the other hand, when $s < 2$, Corollary \ref{COR:intp} (i)
	implies that
	$\E [\|u^\o\|_{L^p (\R^d)}] < \infty$ when $p > \frac{4}s$ and $p<\frac{2d}{d-2}$ if $d\ge 3$. 
	Therefore, the potential energy $\frac{1}{p} \int_{\R^d} |u(x)|^p dx$ does not require a renormalization if we confine ourselves to the range $p > \frac4s$ and $p<\frac{2d}{d-2}$ if $d\ge 3$.

	To define the Gaussian measure $\mu$ in \eqref{Gaussian} rigorously, 
	we start with a finite dimensional version.
	First, define the spectral projection $\P_N$ by
	\begin{align}
		\label{projN}
		\P_N u =  \sum_{n=0}^N u_n  e_n.
	\end{align} 
	
	
	\noi 
	The image of $\P_N$ is the finite-dimensional space
	\[
	E_N = \textup{span} \{ e_0, \cdots, e_N\}.
	\]
	
	\noi
	Through the isometric map 
	\begin{align}
		(u_n)_{n=0}^N \mapsto \sum_{n=0}^N u_n e_n,
		\label{iso}
	\end{align}
	
	\noi
	from $\mathbb C^{N+1}$ to $E_N$, 
	we may identify $E_N$ with $\mathbb C^{N+1}$.
	Consider a Gaussian measure on 
	$\mathbb C^{N+1}$ (or on $ \R^{2N+2}$) given by
	\[
	d  \mu_N = \prod_{n=0}^N \frac{\ld_n^2}{2\pi} e^{-\frac{\ld_n^2}2 |u_n|^2} d u_n d \overline{u}_n.
	\]
	
	\noi
	This Gaussian measure defines a probability measure on the finite dimensional space $E_N$ 
	via the map \eqref{iso},
	which will be also denoted by $\mu_N$.
	The measure $\mu_N$ can also be viewed as the induced
	probability measure under the map 
	\begin{align}
		\o \mapsto u_N^\o  :  = \sum_{n=0}^N \frac{g_n (\o)}{\ld_n}  e_n  .
		\label{RVN}
	\end{align}
	
	\noi
	Given any $\s > \frac1s - \frac12$, 
	the sequence $(u_N^\o)$
	is a Cauchy sequence in $L^2 (\O; \H^{- \s} (\R^d))$ converging to $u^\o$ given in \eqref{maps}.
	See Corollary \ref{COR:intp} (ii).
	In particular,
	the distribution of the random variable $u^\o \in \H^{- \s}(\R^d)$ is the
	Gaussian measure $\mu$.
	The measure $\mu$ can be decomposed as
	\begin{align}
		\label{mu}
		\mu = \mu_N \otimes \mu_N^\perp,
	\end{align}
	
	\noi
	where the measure $\mu_N^\perp$ is the distribution of the 
	random variable given by 
	\[
	u_N^{\o,\perp} (x):  = \sum_{n = N+1}^\infty \frac{g_n (\o)}{\ld_n} e_n (x).
	\]
	
	
	
	
	Recall from the discussion in the introduction that, to define the focusing Gibbs measure \eqref{gibbs}, a mass cut-off is necessary.
	As previously mentioned, note that $u^\o \notin L^2(\mathbb{R}^d)$ $\mu$-almost surely when $s < 2$. This motivates the introduction of a Wick-ordered renormalized $L^2$-mass, similar to the approach used in \cite{BO99,BTTz13,OST22,RSTW22}.
	Given $x \in \R^d$, $u_N^\o (x)$ in \eqref{RVN} 
	is a mean-zero complex-valued Gaussian random variable with variance
	\begin{align}
		\label{variance}
		\s_N (x) = \E \big[ |u_N^\o (x)|^2 \big] = 2\sum_{n=0}^N \frac{e^2_n(x)}{\ld_n^2} ,
	\end{align}
	
	\noi
	from which and Corollary \ref{COR:CLR} we have 
	\[
	\E\left[\|u_N^\o\|_{L^2(\R^d)}^2\right]=\int_{\R^d} \s_N (x) dx = \sum_{n =0}^N \frac{2}{\ld_n^2}  \sim \ld_N^{-1 + \frac2s} \to \infty 
	\]
	
	\noi
	as $N \to \infty$ provided $s\in (1,2)$. Here $\s_N$ depends on $x\in \R^d$ as the random process $u^\o$ given by \eqref{maps} is not stationary.
	We can then define the Wick power $\wick{|u_N|^2}$ via
	\begin{align}
		\wick{|u_N|^2} =  |u_N|^2 - \s_N.
		\label{Wick}
	\end{align}
	
	\noi
	It is known (See Corollary \ref{COR:WCE} below) that 
	\begin{equation}\label{Wick bis}
		\int_{\R^d} \wick{|u_N (x)|^2} dx \to \int_{\R^d} \wick{|u (x)|^2} dx 
	\end{equation}
	$\mu$-almost surely, which defines the renormalized (Wick-ordered) $L^2$ mass in the right-hand side.

	The main purpose of this paper is to define the focusing
	Gibbs measure \eqref{gibbs}  with Wick-ordered $L^2$-cutoff.
	We start with a finite dimensional approximation.
	\begin{align}
		d \rho_{K,N} (u) =  \mathcal  Z_{K,N}^{-1} \ind_{ \{ | \int_{\R^d} \wick{| u_N (x)|^2} dx | \le K\}}  e^{\frac{\al}p{\| u_N \|_{L^p (\R^d)}^p}} d\mu_N (u_N)  \otimes d \mu_N^\perp (u_N^\perp),
		\label{tru_rho}
	\end{align} 
	
	\noi
	where $u_N = \P_{N} u$, $u_N^\perp = \P^\perp_N u := u - \P_{N} u$, and the partition function $\mathcal Z_{K,N}$ is given by
	\begin{align}
		\label{partition}
		\mathcal Z_{K,N} = \int \ind_{ \{ | \int_{\R^d} \wick{| u_N (x)|^2} dx | \le K\}}  e^{\frac{\al}p{\| u_N \|_{L^p (\R^d)}^p}} d\mu (u)  .
	\end{align}
	
	\noi
	Our main findings are sharp criteria under which the above \eqref{tru_rho} converges to a probability measure as $N \to \infty$ (i.e. $0 < Z_{K,N} < \infty$ uniformly in $N$). They are stated as follows:

	\begin{theorem}[\textbf{Gibbs measure construction, subharmonic case}]\label{THM:main}\mbox{}\\
		Let $d\ge 1$, $1<s<2$ and assume the radial condition when $d\ge 2$. Then the following statements hold\textup{:}
		
		\smallskip
		
		\noi
		\begin{itemize}
			\item
			[\textup{(i)}] \textup{(subcritical case)} 
			If 
			$$\frac4s < p< 2+\frac{4s}{(d-1)s+2},$$
			then for any $\al, K > 0$, we have uniform exponential integrability of the density:
			given any finite $r \ge 1$, 
			\begin{align}
				\label{uniint_p}
				\sup_{N \in \mathbb N} \Big\|  \ind_{ \{ | \int_{\R^d} \wick{ | u_N (x)|^2 } dx | \le K\}}  e^{\frac{\al} p{\| u_N \|_{L^p (\R^d)}^p}} \Big\|_{L^r (\mu)}   < \infty .
			\end{align}
			
			\noi
			Moreover, we have
			\begin{align}
				\lim_{N \to \infty}  \ind_{ \{ | \int_{\R^d} \wick{ | u_N (x)|^2 } dx | \le K\}}  e^{\frac{\al} p{\| u_N \|_{L^p (\R^d)}^p}} 
				=  \ind_{ \{ | \int_{\R^d} \wick{ | u (x)|^2 } dx | \le K\}}  e^{\frac{\al} p{\| u \|_{L^p (\R^d)}^p}} 
				\label{cov-lp}
			\end{align}
			
			\noi
			in $ L^r(\mu)$.
			As a consequence, the Gibbs measure $\rho_{N,K}$ in \eqref{tru_rho} converges, in total variation,
			to the focusing Gibbs measure $\rho_K$ defined by
			\begin{align}
				\label{rho}
				d \rho_K (u) = \mathcal  Z_{K}^{-1} \ind_{ \{ | \int_{\R^d} \wick{ | u (x)|^2 } dx | \le K\}}  e^{\frac{\al} p{\| u \|_{L^p (\R^d)}^p}} d\mu (u)  .
			\end{align}
			
			\noi
			Furthermore, the resulting measure $\rho_K$ is absolutely continuous with respect to the base Gaussian free field $\mu$ in \eqref{mu}.
			
			\smallskip
			
			\noi
			\item[\textup{(ii)}] \textup{(critical case)}
			If 
			$$p = 2+\frac{4s}{(d-1)s +2},$$
			we have the following phase transition for the Gibbs measure in \eqref{gibbs}. Then, for every $K > 0$, there exists $\al_0 = \al_0(K) \in (0, \infty)$ such that
			\begin{itemize}
				\item[\textup{(a)}]  \textup{(weakly nonlinear regime)}.
				Let $ \al < \al_0$.
				Then the Gibbs measure $\rho_{K,N}$ in \eqref{tru_rho} converges, in total variation, to the focusing Gibbs measure $\rho_K$ defined by
				\begin{align}
					\label{rho1}
					d \rho_K (u) = \mathcal  Z_{K}^{-1} \ind_{ \{ | \int_{\R^d} \wick{ | u (x)|^2 } dx | \le K\}}  e^{\frac{\al} p{\| u \|_{L^p (\R^d)}^p}} d\mu (u)  .
				\end{align}
				
				\item[\textup{(b)}]  
				\textup{(strongly nonlinear regime)}.
				When $\al > \al_0$, the focusing Gibbs measure \eqref{gibbs},
				even with a Wick-ordered $L^2$-cutoff,
				cannot be defined as a probability measure. 
				
			\end{itemize}
			Moreover, we have that 
			$$ 0 < \inf_{K>0} \alpha_0(K) \le \sup_{K > 0} \alpha_0(K) < \infty, $$
			so \textrm{(a)} and \textrm{(b)} hold for $\al \ll 1$ and $\al \gg 1$ (respectively), independently of the particular value of $K>0$.
			
			\smallskip
			
			\noi
			\item[\textup{(iii)}] \textup{(supercritical case)}
			Let 
			$$p > 2+\frac{4s}{(d-1)s+2}$$ 
			and further assume that 
			$$p<\frac{2d}{d-2} \mbox{ if } d\ge 3.$$
			Then, for any $\al, K >0$, we have
			\begin{align}
				\label{non_int}
				\sup_{N \in \mathbb N} \mathcal Z_{K,N} = \sup_{N \in \mathbb N} \Big\|  \ind_{ \{ | \int_{\R^d} \wick{| u_N (x)|^2} dx | \le K\}}  e^{\frac{\al} p{\| u_N \|_{L^p (\R^d)}^p}} \Big\|_{L^1 (\mu)}   = \infty,
			\end{align}
			
			\noi
			where $\mathcal Z_{K,N}$ is the partition function given in \eqref{partition}.
			The same divergence holds for $\mathcal Z_K$, i.e.\ 
			\begin{align}
				\label{non_int2}
				\mathcal Z_{K} =  \int \ind_{ \{ | \int_{\R^d} \wick{| u (x)|^2} dx | \le K\}}  e^{\frac\al p{\| u \|_{L^p (\R^d)}^p}}  d\mu   = \infty.
			\end{align}
			As a consequence, 
			the focusing Gibbs measure \eqref{gibbs},
			even with a Wick-ordered $L^2$-cutoff,
			cannot be defined as a probability measure. 
		\end{itemize}
		
	\end{theorem}
	
	\medskip 
	
	\begin{remark}
		For the critical case, Theorem \ref{THM:main} (ii) claims the normalizability/non-normalizability of the Gibbs measure for $\al \ll 1$ and $\al \gg 1$ respectively, uniformly in the cut-off size $K$. Furthermore, when $\al \sim 1$, there exists a critical coupling constant $\al_0(K)$ for given $K>0$. However,  whether the critical coupling constant $\al_0$ is independent of $K$ is not clear. 
	\end{remark}
	
	Similar results also hold for superharmonic potentials. In this setting, we have a more precise description of the phase transition.

	\begin{theorem}[\textbf{Gibbs measure construction, superharmonic case}]\label{THM:main1}\mbox{}\\
		Let $d\ge 1$, $s>2$ and assume the radial condition when $d\ge 2$.
		Given  $\al, K > 0$, define
		the partition function $\mathcal Z_{K}$ by 
		\begin{align}
			\label{Z1}
			\mathcal Z_{K} = \E_{\mu}
			\Big[e^{\frac{\al}{p} \int_{\R^d} |u|^p d x}\ind_{\{\|u\|^2_{L^2(\R^d)}\le K\}}\Big], 
		\end{align}
		
		\noi
		where $\E_{\mu}$ denotes an expectation with respect
		to the Gaussian measure $\mu$.
		Then, the following statements hold\textup{:}
		
		\smallskip
		
		\noi
		\begin{itemize}
			\item
			[\textup{(i)}] \textup{(subcritical case)}
			If $2< p< 2+\frac4d$, then $\mathcal Z_{K}<\infty$ for any $\al, K>0$.

			\smallskip

			\noi
			\item[\textup{(ii)}] \textup{(critical case)}
			Let $p= 2+\frac4d$.
			Then, $\mathcal Z_{K}<\infty$ if $\al^{\frac{d}2} K<\|Q\|^2_{L^2(\R^d)}$, and $\mathcal Z_{K}=\infty$ if $\al^{\frac{d}2}K>\|Q\|^2_{L^2(\R^d)}$. Here, $Q$ is an optimizer of the Gagliardo-Nirenberg-Sobolev inequality on $\R^d$ 
			\begin{align}\label{GNS}
				\|u\|^p_{L^p(\R^d)} \leq C_{\textup{GNS}} \|\nabla u\|_{L^2(\R^d)}^{\frac{d(p-2)}{2}} \|u\|_{L^2(\R^d)}^{\frac{4-(d-2)(p-2)}{2}}, \quad u \in H^1(\R^d).
			\end{align}
			
			\item[\textup{(iii)}]
			\textup{(supercritical case)}
			If $ p> 2+\frac4d$ and $p<\frac{2d}{d-2}$ if $d\ge 3$,
			then $\mathcal Z_{K} = \infty$ for any $\al, K >0$.
			
		\end{itemize}
	\end{theorem}

	As previously mentioned, when $p=4$, Theorems \ref{THM:main} and \ref{THM:main1} represent an improvement over a recent work \cite{DR23}, where the normalizability of the focusing Gibbs measure was established only for $s>\frac{8}{5}$. Here we not only demonstrate its normalizability for all $s>1$, but we also extend the result to other nonlinearities.
	
	Furthermore, our main results also extend a recent work \cite{RSTW22}, where the normalizability and non-normalizability were established specifically for the case of the harmonic potential $s=2$. The extension to anharmonic potentials with $s>1$ is not a direct adaptation of the arguments presented in \cite{RSTW22}, primarily due to a lack of explicit knowledge concerning the eigenfunctions and eigenvalues of the Schr\"odinger operator with anharmonic potential.
	
	An intriguing feature of our main results is the identification of a new critical nonlinearity (below the standard mass-critical nonlinearity $p=2+\frac{4}{d}$) that exhibits a phase transition in the subharmonic case $1<s<2$. It is noteworthy that, at this level of nonlinearity, the associated nonlinear Schr\"odinger equation always exhibits global dynamics given sufficiently high regularity data (so that the energy is finite). This leads to the suggestion that a new blow-up phenomenon may emerge for solutions of the NLS with low regularity (within the support of the Gibbs measure) due to the weak growth of the trapping potential. To the best of our knowledge, no such result is available in the existing literature.
	
	As previously discussed, we address the challenge arising from the absence of an explicit formula for the eigenvalues and $L^p$-estimates of eigenfunctions by examining the Green function of the Schr\"odinger operator $\mathcal{L}$ defined in \eqref{anharmonic}. 
	
	First, in order to determine the regularity of the Gaussian measure, we need to establish an upper bound on the trace of $\mathcal{L}$ raised to certain negative powers (see Lemma \ref{LEM:main2}). In the one-dimensional case, this was proved in \cite[Example 3.2]{LNR18} using a version of the Lieb-Thirring inequality originating in~\cite{DFLP06}. This does not apply to the radial case in higher dimensions, which leads to a Schr\"odinger operator with an inverse square potential. We circumvent this issue by employing the fundamental solution of the heat equation with an inverse-square potential, as provided by Ortner and Wagner~\cite{OW18}. This fundamental solution is expressed in terms of the modified Bessel function of the first kind. Taking advantage of the asymptotic behavior of this special function near the origin and at infinity, we establish a variant of the Lieb-Thirring inequality adapted to our context.
	
	Next, we need an $L^p$-bound for the diagonal of the Green function (see Lemma \ref{LEM:main3}). To achieve this, we first employ the asymptotics of the modified Bessel function to deduce a decay property near the origin. Subsequently, to capture decay behavior at infinity, we utilize the odd extension technique to extend the underlying operator to the entire real line and then use previously known results for the one-dimensional anharmonic oscillator by~\cite{LNR18}. In the case of two dimensions, careful considerations are required due to the negative sign in front of the inverse-square potential, which is addressed through a refined Hardy inequality recently established by Frank and Merz~\cite{FM23}.
	
	Finally, we establish a Weyl-type asymptotic for the number of eigenvalues below a large threshold for the radial Schr\"odinger operator with an anharmonic potential. In one dimension, this type of estimate is known as the Cwikel-Lieb-Rozenbljum bound (see, for example, \cite[Lemma D.1]{DR23}). The proof is based on coherent states and semi-classical analysis on the phase space. Specifically, we define a quantum energy whose minimizer is attained by a fermionic density matrix. The problem of counting the number of eigenvalues below a large threshold is then reduced to computing the trace of this fermionic density matrix, which is determined by comparing lower and upper bounds.
	
	The results of Theorem \ref{THM:main} and Theorem \ref{THM:main1}, together with the previous work on $s = 2$ in \cite{RSTW22}, give a complete characterization of the construction of regular\footnote{We mean that the potential $|u|^p$ is well-defined.} focusing Gibbs measures with trapping potentials. We observe different critical phenomena for different values of $s$: when $s > 2$, the phase transition at the critical exponent depends on the cutoff parameter $K$; when $s = 2$, there is no phase transition at the critical case \cite{RSTW22}; when $s \in (1,2)$, the critical exponent depends on $s$, and moreover, the phase transition depends on the coupling strength $\al$. These differences reflect the distinct spectral properties of the Schr\"odinger operators for different $s$, and require different techniques for each case. 
	
	We apply the Barashkov-Gubinelli variational method \cite{BG20} to prove Theorem \ref{THM:main} and Theorem \ref{THM:main1}. Specifically, we use the variational formula of Bou\'e-Dupuis \cite{BD98,Zhang09}, Lemma \ref{LEM:var}, to reformulate the Gibbs measure construction as a stochastic optimization problem. Then, to show the normalizability part of Theorems \ref{THM:main} and \ref{THM:main1}, we need to control the stochastic optimization problems uniformly; see Subsections \ref{SEC:nor} and \ref{SEC:nor1}. To show the non-normalizability parts, we need to find suitable sequences of appropriate drift terms, which drive the stochastic optimization problem to diverge; see Subsections \ref{SEC:non} and \ref{SEC:non2}. We remark that the asymptotic behaviour of the variational formulae, which depends on the behaviour of Schatten norms of trapped Laplacian in Corollary \ref{COR:CLR}, determines the different criticality of different $s$.
	
	For the subharmonic trapping case $s \in (1,2)$, the proof of Theorem \ref{THM:main} is inspired by recent works \cite{OOT,OOT1,TW23}. In particular, in \cite{OOT}, the authors studied a focusing $\Phi_4^3$-model with a Hartree-type nonlinearity, where the potential for the Hartree nonlinearity is given by the Bessel potential of order $\beta$. They show that the case $\beta = 2$ is critical, leading to phase transition regarding the coupling strength. In \cite{OOT1}, similar critical behaviour was shown for the $\Phi_3^3$ measure. However, both the above-mentioned works only consider  cases where the coupling strength is either very weak, i.e. weakly nonlinear regime $\al \ll 1$, or very strong, i.e. strongly nonlinear regime $\al \gg 1$; and leave the case $\al \sim 1$ open. Theorem \ref{THM:main} shows the existence of a critical coupling strength $\al_0$ for normalizability. However, if $\al = \al_0$, then the normalizability is undetermined by Theorem \ref{THM:main}.
	
	For the superharmonic trapping case $s > 2$, the trap is strong enough to make the problem almost a local one, similar to the case on $\T$. Heuristically, this is because the trapping potential penalises functions which have nontrivial mass outside of a large ball. Indeed, when the trapping is strong enough, i.e. in the superharmonic case $s > 2$, we have $\|u\|_{L^2} < \infty$ $\mu$-a.s., at which point the analysis becomes similar to the case of a bounded domain. Therefore, the Gibbs measure is less singular and the proof of Theorem \ref{THM:main1} exploits ideas from the torus setting \cite{LW22,OST22}. In particular, the ground state, which is the minimizer of the Gagliardo-Nirenberg-Sobolev inequality, is essential to characterize the critical behaviour of the Gibbs measure, as in the torus cases in \cite{LW22,OST22}. Our proof also relies on the exponential decay of the ground state at infinity to eliminate the unbounded trapping by using a suitable scaling argument.
	
	\begin{remark}\mbox{}
		
		\begin{itemize}
			\item[(i)] The $L^p$-bound for the diagonal of the Green function, as established in Lemma \ref{LEM:main3}, allows us to define the defocusing measure for all $p>\max\left\{\frac{4}{s},2\right\}$ and $p<\frac{2d}{d-2}$ if $d\geq 3$. However, when $p\geq \frac{2d}{d-2}$, the potential energy $R_p(u)$ (see \eqref{Rp}) becomes infinite almost surely on the support of the Gaussian measure $\mu$, necessitating a renormalization. The question of constructing the defocusing Gibbs measure in this case remains an interesting open problem.
			
			\item[(ii)] In the absence of the radial assumption, a renormalization would be necessary for any (even) $p>2$ as soon as $d\geq 2$. We are aware of the sole work by de Bouard and Debussche \cite{deBDF22}, where they established the construction of the Gibbs measure associated with the 2D defocusing cubic NLS with a harmonic potential. Extending this construction to other (even) power nonlinearities $p>4$ and potentially exploring it with other (anharmonic) potentials would be a highly interesting problem. On the other hand, in the focusing case, even with a renormalization of the potential energy, we expect that the Gibbs measure would remain non-normalizable, similar to the work of Brydges and Slade \cite{BS96}.
		\end{itemize}
	\end{remark}

    \begin{remark}
        The construction of Gibbs measures for wave equations in focusing settings has been explored in recent literature, notably \cite{OOT} for Hartree-type interactions and \cite{OOT1} for cubic interactions, both by the third author and their collaborators; both works identify phase transitions in a broadly similar spirit. However, there are several fundamental differences between the Schr\"dinger and wave settings.
        
        A first difference is that the invariant measure for the wave equation is concentrated on real-valued distributions. While this does affect some very fine properties of these measures, it is not a substantial difference, and indeed, we expect that the results of Theorems \ref{THM:main} and \ref{THM:main1} to hold without any modifications in the statements (including the optimal constants) when we restrict the measure $\mu$ to the real-value setting. 
        
        For the purposes of this paper, the main difference is that, unlike the Schr\"odinger equation, the wave equation does not conserve mass, so a mass-cutoff argument does not provide the same information on the PDE flow (in particular, such cutoffs are not invariant under the wave flow). This problem has been addressed in \cite{OOT, OOT1} by considering a \emph{tamed} version of the measure, or more precisely, by replacing the mass cutoff with an exponential factor of the form 
        $$ \exp\Big( - \big|\int\wick{u^2}\big|^\gamma\Big), $$
        for some $\gamma \gg 1$, and then considering the wave equation associated with the modified Hamiltonian. We believe that proceeding as in \cite{OOT1}, Theorem \ref{THM:main} holds for the tamed measures as well, as long as $\gamma$ is big enough in cases (i) and (ii) (a). On the other hand, we do not expect Theorem \ref{THM:main1} to hold as stated for the tamed measures. The reason for this is that the critical case (ii) formally requires $\gamma = \infty$. Therefore, we expect the measure to exist (for $\gamma > \gamma(p)$) in the subcritical case, and not to exist in both the critical and supercritical cases.
    \end{remark}

    {\bf Notation.} Throughout the paper, we use standard asymptotic notation. We write $A \lesssim B$ if there exists a constant $C>0$, independent of the relevant parameters, such that $A \leq C B$. Similarly, $A \gtrsim B$ means $A \geq C^{-1} B$, and $A \sim B$ means that both $A \lesssim B$ and $A \gtrsim B$ hold. We write $A \ll B$ (resp. $A \gg B$) to indicate that $A$ is sufficiently small compared to $B$ (resp. $A$ is sufficiently large compared to $B$), in a way depending only on fixed parameters of the problem. Expressions such as $p=1+$ mean that $p=1+\varepsilon$ for some arbitrarily small but fixed $\varepsilon>0$; similarly, $p=1-$ means $p=1-\varepsilon$. 
    
    The structure of this paper is as follows. In Section~\ref{sec:Schro}, we study the radial Schr\"odinger operator with anharmonic potential including some properties of the resolvent and a Weyl-type asymptotic. In Section~\ref{sec:variational}, we recall the Bou\'e-Dupuis variational formula, which plays a vital role in our proof. We also give some applications of the resolvent estimates derived in the preceding section. Section~\ref{SEC:subharmonic} is devoted to demonstrating the normalizability and non-normalizability for the subharmonic potential, while Section~\ref{SEC:superharmonic} addresses the case of the superharmonic potential. Finally, we recall in the appendices some essential tools used in proving Weyl-type asymptotics, and extend our result to the case of fractional Schr\"odinger operators with anharmonic potentials.
	
	\begin{ackno}
		The authors would like to thank Tadahiro Oh for helpful discussions and for his suggestion of possible collaboration between us. Y.W. would like to thank Rowan Killip for inspiring discussions and valuable suggestions at the early stage of this project.
		The first two authors were supported by the European Union's Horizon 2020 Research and Innovation Programme (Grant agreement CORFRONMAT No. 758620).
		Y.W. was supported by the EPSRC New Investigator Award (grant no. EP/V003178/1).
	\end{ackno}

	\section{Radial Schr\"odinger operators with anharmonic potential}\label{sec:Schro}  
	
	In this section, we collect some properties of the Schr\"odinger operator with anharmonic potential. In dimensions larger than 1, we restrict its' domain to radial functions.
	
	Let $d\ge 1$, $s>0$ and consider the Schr\"odinger operator
	$$
	\L:=-\Delta + |x|^s \quad \text{on } \R^d.
	$$
	When $d\ge 2$, we restrict our consideration to radial functions, that is we look at 
	\begin{equation}\label{eq:RadLap} 
		\L := -\partial_r ^2 -\frac{d-1}{r} \partial_r +r^s \text{ on } (0,+\infty).
	\end{equation}
	We may define the above as the Friedrichs extension of the associated quadratic form, starting from the domain of $C^\infty$ functions vanishing at infinity. 
	
	We will rely on the change of variable
	\begin{equation}\label{eq:changerad}
		U: f(r) \mapsto r^{\frac{d-1}{2}} f(r)
	\end{equation}
	to reduce our study to the operator
	\begin{align} \label{L1}
		\L_1  = -\partial^2_r + \frac{(d-1)(d-3)}{4r^2} + r^s
	\end{align}
	acting on $L^2((0,+\infty),dr)$ with a Dirichlet boundary condition at $0$. More precisely, the radial Laplacian corresponds to the Friedrichs extension of~\eqref{L1} (see below).
	
	\subsection{Properties of the resolvent}

	\begin{lemma}[\textbf{Definition of the radial Laplacian with trap}]\label{LEM:main1}\mbox{}\\
		The follow properties hold:
		\begin{itemize}
			\item
			[\textup{(i)}] When $d\geq 2$, the self-adjoint extension of $\L$ (still denoted by $\L$) has domain $\mathcal{D} (\L)$ such that $U\mathcal{D} (\L)$ is a subset of continuous functions vanishing at the origin.
			\item
			[\textup{(ii)}] $(\L+1)^{-1}$ is compact.
			\item
			[\textup{(iii)}] There exists a sequence of eigenvalues $(\ld_n^2)_{n\ge 0}$ of $\L$ satisfying
			$$
			0<\ld_0 \le \ld_1 \le \cdots \le \ld_n \to +\infty
			$$ 
			and the corresponding eigenfunctions $(e_n)_{n\ge 0}$ form an orthonormal basis of $L^2(\R^d)$ (radial functions of $L^2(\R^d)$ when $d\ge 2$). In addition, eigenfunctions can be chosen to be real-valued.
		\end{itemize}
	\end{lemma}
	
	The Dirichlet condition at the origin inherited by the domain of $\L_1$ will be crucial to employ results from~\cite{OW18} in the sequel (see~\eqref{heat-equa} and~\eqref{heat-kern} below).
	
	\begin{proof}
		The 1D case is standard (see e.g., \cite[Theorem 3.1]{BS12}). Let us consider the case $d\geq 2$ where the radial assumption is imposed.
		Consider the quadratic form associated to $\L_1$ in~\eqref{L1}, namely
		\begin{equation}\label{eq:RadForm}
			\mathcal Q(f,f) := \int_0^{+\infty} |\partial_r f(r)|^2 + \frac{(d-1)(d-3)}{4r^2} |f(r)|^2 + r^s |f(r)|^2 dr, \quad f \in C^\infty_0((0,+\infty)).
		\end{equation}
		When $d\geq 3$, it is clear that $\mathcal Q$ is non-negative. When $d=2$, we use the following Hardy inequality (see e.g., \cite{Davies99}):
		\begin{align} \label{Hardy-ineq}
			\int_0^{+\infty} \frac{1}{4r^2} |f(r)|^2 dr \leq \int_0^{+\infty} |\partial_r f(r)|^2 dr, \quad f\in C^\infty_0((0,+\infty))
		\end{align}
		to deduce that $\mathcal Q$ is also non-negative. The Friedrichs extension theorem (see e.g., \cite[Theorem X.23]{RS75}) ensures that $\L_1$ admits a self-adjoint realization~\cite{OW18}, still denoted by $\L_1$, whose core is $C^\infty_0((0,+\infty))$. Observe that the map $U$ defined in~\eqref{eq:changerad} maps the quadratic form domain of the radial Laplacian~\eqref{eq:RadLap} to the domain of~\eqref{eq:RadForm}. Hence the Friedrichs extension, whose domain includes the quadratic form domain, is the appropriate one. Other self-adjoint extensions exist (see \cite{BDG11}), but we shall prove below that they are not selected by conjugating~\eqref{eq:RadLap} with~\eqref{eq:changerad}.
		
		When $d\geq 3$, it is immediate that the quadratic form domain of~\eqref{eq:RadForm} is embedded in
		$$ H^1 ((0,+\infty)) \cap  L^2 ((0,+\infty),r^s dr),$$
		which is well-known to be compactly embedded in $L^2((0,+\infty))$. Hence the Friedrichs realization of $\L_1$ has compact resolvent. Since the above also continuously embeds into $C^0 ((0,+\infty))$ and  the second term in~\eqref{eq:RadForm} comes with a positive sign, it is also clear that functions of the domain must vanish at the origin.
		
		When $d=2$, we need the following refined Hardy inequality\footnote{The inverse inequality holds for all $0<\theta<\frac{3}{2}$. In particular, we have the equivalence of norms for all $0<\theta<1$.} proved recently by Frank and Merz \cite{FM23}: for $0<\theta<1$,
		\begin{align}\label{Hardy-ineq-2d}
			\int_0^{+\infty} |(-\partial^2_r)^{\theta/2} f(r)|^2 dr \leq C \int_0^{+\infty}\Big|\Big(-\partial^2_r -\frac{1}{4r^2}\Big)^{\theta/2}f(r)\Big|^2 dr, \quad \forall f \in C^\infty_0((0,+\infty)).
		\end{align}
		Since $x \mapsto x^\theta$ with $0<\theta<1$ is concave, Jensen's inequality for operators (see\footnote{In \cite[Theorem 8.9]{Simon15}, Jensen's inequality was proved in a finite dimensional setting. However, the same proof applies for the infinite dimensional case using the multiplication operator form of the spectral theorem for unbounded self-adjoint operators (see e.g. \cite[Theorem VII.4]{RS75}).} e.g. \cite[Theorem 8.9]{Simon15}) yields
		$$\mathcal Q(f,f) \geq C \left(\int_0^{+\infty} |(-\partial^2_r)^{\theta/2} f(r)|^2 dr\right)^{\frac1\theta} + \int_0^{+\infty} r^s  |f(r)|^2 dr, $$
		which implies that the bottom of the spectrum of $\L_1$ is positive. Also
		\[
		(-\partial^2_r)^\theta \leq C \left(-\partial^2_r -\frac{1}{4r^2} \right)^\theta 
		\]
		as operators. By the operator monotonicity of $x \mapsto x^\theta$ with $0<\theta<1$ (see \cite[Theorem 2.6]{Carlen10}), we infer that
		\begin{align}
			(-\partial^2_r)^\theta 	&\leq C \left(-\partial^2_r -\frac{1}{4r^2}  + \frac{1}{2} r^s\right)^\theta \nonumber\\
			&\leq C \left(-\partial^2_r -\frac{1}{4r^2}  + \frac{1}{2} r^s\right) \label{est-A-theta}
		\end{align}
		or 
		\[
		-\partial^2_r -\frac{1}{4r^2} + r^s \geq C^{-1} (-\partial^2_r)^\theta + \frac{1}{2} r^s
		\]
		for some constant $C>0$ depending only on the bottom of the spectrum of $\L_1$. Here the constant $C$ varies line-by-line. Hence, for all $0<\theta <1$ the quadratic form domain of $\L_1$ continuously embeds into 
		$$ H^\theta ((0,+\infty)) \cap  L^2 ((0,+\infty),r^s dr),$$
		which is compactly embedded in $L^2((0,+\infty))$ when $\theta > 1/2$. Hence $\L_1$ has compact resolvent. Since the above space also continuously embeds into continuous functions and~\eqref{eq:changerad} maps regular functions to functions vanishing at the origin, we also deduce that functions from the Friedrichs domain of $\L_1$ must vanish at the origin.
		
		Applying the spectral theorem for compact operators (see e.g., \cite[Theorem XIII.64]{RS78}), there exists a sequence of eigenfunctions $f_n$ of $(\L_1+1)^{-1}$ which forms an orthonormal basis of $L^2((0,+\infty),dr)$. Moreover, the corresponding eigenvalues satisfy $\mu_n \to 0^+$ as $n\to \infty$. Note that 
		\[
		(\L_1+1)^{-1} f_n = \mu_n f_n \Leftrightarrow \L_1 f_n = (\mu_n^{-1}-1) f_n.
		\]
		In particular, $f_n$ is also an eigenfunction of $\L_1$ with the corresponding eigenvalue $\mu_n^{-1}-1\to +\infty$ as $n\to \infty$.  The rest of the lemma follows by setting $e_n = r^{\frac{d-1}{2}} f_n$ and $\ld_n^2 = \mu_n^{-1}-1$.
	\end{proof}

	We will now specify which Schatten space the resolvent of the operator we just constructed belongs to:
	
	\begin{lemma}[\textbf{Schatten-norm bounds for the resolvent}]\label{LEM:main2}\mbox{}\\
		Let $s>0$ and $(\ld_n^2)_{n\ge 0}$ be the eigenvalues of $\L$ given in Lemma \ref{LEM:main1}. Then we have
		\begin{align}
			\label{traceEst}
			{\rm Tr} [\L^{-\al}] = \sum_{n\ge 0} \ld_n^{-2\al} < \infty
		\end{align}
		\noi 
		provided $\al > \frac12 + \frac1s$.
	\end{lemma}
	
	\begin{proof}
		The 1D case was proved in \cite[Example 3.2]{LNR15} and \cite[Lemma A.1]{DR23}. In the following, we only consider the higher dimensional cases. To this end, we shall show the following Lieb--Thirring type inequality for the radial Schr\"odinger operator $\L = -\Delta + V(r)$: for $\al>\frac12$,
		\begin{align} \label{Lieb-Thir-ineq}
			{\rm Tr}[\L^{-\al}] \leq C(\alpha) \int_0^{+\infty} (V(r))^{\frac{1}{2}-\al} dr.
		\end{align}
		Assume \eqref{Lieb-Thir-ineq} for the moment, let us prove \eqref{traceEst}. Since $\L \geq \ld_0$, we infer that
		\[
		{\rm Tr}[\L^{-\al}] \leq 2^\al {\rm Tr}[(\L+\ld_0)^{-\al}],
		\]
		where $\ld_0>0$ is the first eigenvalue of $\L$. Applying \eqref{Lieb-Thir-ineq} with $V(r)=r^s+\ld_0$, we have
		\[
		{\rm Tr}[\L^{-\al}]\leq C\int_0^{+\infty} (r^s+\ld_0)^{\frac12 - \al} dr <\infty
		\]
		provided $s\al-\frac{s}{2} >1$ or $\al>\frac12 + \frac1s$. This proves \eqref{traceEst}.	
		
		It remains to prove \eqref{Lieb-Thir-ineq}. Let us start with the following observation. Let $(\ld^2, f)$ be an eigenpair of $\L=-\Delta +V(r)$, namely
		$$
		\left(-\partial^2_r -\frac{d-1}{r}\partial_r + V(r)\right) f(r) = \ld^2 f(r). 
		$$
		By the change of variable 
		\begin{align}\label{chan-vari}
			g(r) = r^{\frac{d-1}{2}} f(r),
		\end{align}
		it becomes
		$$
		\left(-\partial^2_r +\frac{(d-1)(d-3)}{4r^2} + V(r) \right) g(r) = \ld^2 g(r)
		$$
		or $(\ld^2, g)$ is an eigenpair of 
		$$
		\L_1 := -\partial^2_r +\frac{(d-1)(d-3)}{4r^2} + V(r)
		$$
		acting on $L^2((0,+\infty),dr)$ with the Dirichlet boundary condition at $0$. In particular, we have
		\begin{align}\label{TraceL1}
			{\rm Tr}[\L^{-\al}] = \sum_{n\ge 0} \ld_n^{-2\al} = {\rm Tr}[\L_1^{-\al}]
		\end{align}
		hence the proof of \eqref{Lieb-Thir-ineq} is reduced to proving
		\begin{align} \label{LT-proof}
			{\rm Tr}[\L_1^{-\al}] \le C(\alpha) \int_0^{+\infty} (V(r))^{\frac{1}{2}-\al} dr.
		\end{align}
		To show this, we rely on an idea of Dolbeault et al. \cite{DFLP06} which is done in two steps. 
		
		\noindent {\bf Step 1. A Golden--Thompson type inequality.} We first prove that: for $t>0$, 
		\begin{align} \label{Gold-Thom-ineq}
			{\rm Tr} [e^{-t\L_1}] \leq Ct^{-1/2} \int_0^{+\infty} e^{-tV(r)}dr.
		\end{align}
		To prove \eqref{Gold-Thom-ineq}, we recall the following result of Ortner and Wagner \cite{OW18} concerning the fundamental solution to the heat equation with an inverse-square potential on the half line with Dirichlet condition at 0. More precisely, the unique solution to
		\begin{align}\label{heat-equa}
			\left\{
			\begin{array}{rcl}
				\left(\partial_t - \partial^2_r +\frac{(d-1)(d-3)}{4r^2}\right)u(t,r) &=& 0, \quad t>0, \quad r>0, \\
				u(t=0,r) &=& g(r), \quad r>0, \\
				u(t, r=0) &=& 0, \quad t>0.
			\end{array}
			\right.
		\end{align}
		is given by
		\[
		u(t,r) = \int_0^{+\infty} G(t,r,\tau) g(\tau)d\tau,
		\]
		where
		\begin{align}\label{heat-kern}
			G(t,r,\tau) = \frac{\sqrt{r\tau}}{2t} \exp \left(-\frac{r^2+\tau^2}{4t}\right) I_\nu\left(\frac{r\tau}{2t}\right)
		\end{align}
		with $\nu =\frac{d-2}{2}$ and $I_\nu$ the modified Bessel function of the first kind with index $\nu$, namely
		\[
		I_\nu(x) = \sum_{j\geq 0} \frac{1}{\Gamma(j+\nu+1)j!}\left(\frac{x}{2}\right)^{2j+\nu}.
		\]
		We have the following asymptotic behaviors of the modified Bessel function $I_\nu$ (see e.g., \cite[Section 10.30]{NIST}):
		\begin{align} \label{I-nu-zero}
			I_\nu(x) \sim \frac{1}{\Gamma(\nu+1)}\left(\frac{x}{2}\right)^\nu \quad \text{as } x\to 0
		\end{align}
		and 
		\begin{align}\label{I-nu-infi}
			I_\nu(x) \sim \frac{1}{\sqrt{2\pi x}}e^x \quad \text{as } x\to +\infty.
		\end{align} 
		Using Trotter's formula\footnote{For $A, B$ two positive operators on a Hilbert space $\mathcal{H}$, then for all $t>0$, we have
			\[
			e^{-t(A+B)} = s-\lim_{n\to \infty} (e^{-tA/n} e^{-tB/n})^n.
			\]} (see \cite{Kato78}), we have
		\[
		e^{-t\L_1} = s-\lim_{n\to \infty} \left(e^{-\frac{t}{n}H} e^{-\frac{t}{n}V(r)}\right)^n,
		\]
		where 
		\begin{align} \label{H}
			H:= -\partial^2_r +\frac{(d-1)(d-3)}{4r^2}.
		\end{align} 
		The integral kernel of $\left(e^{-\frac{t}{n}H} e^{-\frac{t}{n}V(r)}\right)^n$ is written as
		\begin{multline*}
			K(t,r,\tau) = \int_{(0,+\infty)^{n-1}} G\left(\frac{t}{n}, r, \tau_1\right) e^{-\frac{t}{n}V(\tau_1)} G\left(\frac{t}{n}, \tau_1, \tau_2\right)e^{-\frac{t}{n}V(\tau_2)}\\
			...G\left(\frac{t}{n}, \tau_{n-1}, \tau\right) e^{-\frac{t}{n}V(\tau)} d\tau_1 d\tau_2...d\tau_{n-1}.
		\end{multline*}
		Thus ${\rm Tr}\left[\left(e^{-\frac{t}{n}H} e^{-\frac{t}{n}V(r)}\right)^n\right]$ is
		\begin{multline*}
			\int_0^{+\infty} K(t, \tau, \tau) d\tau = \int_{(0,+\infty)^n} G\left(\frac{t}{n}, \tau, \tau_1\right) e^{-\frac{t}{n}V(\tau_1)} G\left(\frac{t}{n}, \tau_1, \tau_2\right)e^{-\frac{t}{n}V(\tau_2)}\\
			...G\left(\frac{t}{n}, \tau_{n-1}, \tau\right) e^{-\frac{t}{n}V(\tau)} d\tau_1 d\tau_2...d\tau_{n-1} d\tau.
		\end{multline*}
		Set $\tau_0=\tau$. We rewrite this as
		\begin{multline*}
			\int_{(0,+\infty)^n} G\left(\frac{t}{n}, \tau_0, \tau_1\right) G\left(\frac{t}{n}, \tau_1, \tau_2\right)...G\left(\frac{t}{n}, \tau_{n-1}, \tau_0\right) \\
			e^{-\frac{t}{n}\left(V(\tau_0)+V(\tau_1)+...+V(\tau_{n-1})\right)} d\tau_0d\tau_1 d\tau_2...d\tau_{n-1}.
		\end{multline*}
		By the convexity of $x \mapsto e^{-x}$, we have
		\[
		e^{-\frac{t}{n} \left(V(\tau_0)+V(\tau_1)+...+V(\tau_{n-1})\right)} \leq \frac{1}{n} \sum_{k=0}^{n-1} e^{-tV(\tau_k)}.
		\]
		Thus we get
		\begin{align*}
			{\rm Tr}[e^{-t\L_1}] \leq \frac{1}{n} \sum_{k=0}^{n-1} \int_{(0,+\infty)^n} G\left(\frac{t}{n}, \tau_0, \tau_1\right) G\left(\frac{t}{n}, \tau_1, \tau_2\right)...G\left(\frac{t}{n}, \tau_{n-1}, \tau_0\right) e^{-tV(\tau_k)} d\tau_0d\tau_1...d\tau_{n-1}.
		\end{align*}
		The right hand side can be rewritten as
		\[
		\frac{1}{n} \sum_{k=0}^{n-1} \int_0^{+\infty} F(\tau_k)e^{-tV(\tau_k)} d\tau_k,
		\]
		where
		\[
		F(\tau_k):=\int_{(0,+\infty)^{n-1}} G\left(\frac{t}{n}, \tau_0, \tau_1\right) G\left(\frac{t}{n}, \tau_1, \tau_2\right)...G\left(\frac{t}{n}, \tau_{n-1}, \tau_0\right) d\tau_0 d\tau_1...d\tau_{k-1} d\tau_{k+1}...d\tau_{n-1}.
		\]
		We recall that the heat kernel \eqref{heat-kern} is symmetric in $r, \tau$, and satisfies
		\[
		\int_0^{+\infty} G(t,r,\tau') G(s,\tau',\tau)d\tau' = G(t+s, r, \tau),
		\]
		where the latter comes from the fact that $e^{-tH} e^{-sH} = e^{-(t+s)H}$. We thus deduce
		\begin{align*}
			F(\tau_k)&= G(t,\tau_k, \tau_k) \\
			&= \frac{1}{\sqrt{2t}} \sqrt{\frac{\tau_k^2}{2t}} \exp\left(-\frac{\tau_k^2}{2t}\right) I_\nu\left(\frac{\tau_k^2}{2t}\right) \\
			&\leq Ct^{-1/2}, \quad k=0,..., n-1,
		\end{align*}
		where we have used \eqref{I-nu-zero} and \eqref{I-nu-infi} to get the last estimate. This shows that
		\begin{align*}
			{\rm Tr}[e^{-t\L_1}] &\leq \frac{1}{n} \sum_{k=0}^{n-1} \int_0^{+\infty} Ct^{-1/2} e^{-tV(\tau_k)} d\tau_k \\
			&\leq C t^{-1/2} \int_0^{+\infty} e^{-tV(r)}dr
		\end{align*}
		which is \eqref{Gold-Thom-ineq}.	
		
		\noindent {\bf Step 2. A Lieb--Thirring type inequality.} Using the Gamma function
		\begin{align}\label{Gamma}
			\lambda^{-\alpha} = \frac{1}{\Gamma(\alpha)} \int_0^{+\infty} e^{-t\lambda} t^{\alpha-1} dt, \quad \alpha>0, \quad \lambda>0,
		\end{align}
		the functional calculus gives
		\[
		{\rm Tr}[\L_1^{-\al}] =\frac{1}{\Gamma(\al)} \int_0^{+\infty} {\rm Tr}[e^{-t\L_1}] t^{\al-1} dt.
		\]
		Using \eqref{Gold-Thom-ineq}, we get
		\begin{align*}
			{\rm Tr}[\L_1^{-\al}] &\leq \frac{1}{\Gamma(\al)}  \int_0^{+\infty} \left( Ct^{-1/2} \int_0^{+\infty} e^{-tV(r)} dr\right) t^{\al-1} dt \\
			&= \frac{C}{\Gamma(\al)} \int_0^{+\infty} \left(\int_0^{+\infty} e^{-tV(r)} t^{\al-\frac{1}{2}-1} dt\right) dr \\
			&= \frac{C \Gamma\left(\al-\frac{1}{2}\right)}{\Gamma(\al)} \int_0^{+\infty} (V(r))^{\frac{1}{2}-\al} dr
		\end{align*}
		which proves \eqref{LT-proof}. 
	\end{proof}
	
	We next turn to integrability properties of the diagonal part of the resolvent's integral kernel:
	
	\begin{lemma}[\textbf{$L^p$ bounds for the resolvent's integral kernel}]\label{LEM:main3}\mbox{}\\
		Let $s>0$ and $(\ld_n^2,e_n)_{n\ge 0}$ be the eigenpairs of $\L$ given in Lemma \ref{LEM:main1}. Then, the Green function of $\L$ satisfies
		\begin{align}
			\label{GreenEst}
			\L^{-1} (x,x) = \sum_{n\ge 0} \frac{e^2_n(x)}{\ld_n^2} \in L^p(\R^d)
		\end{align}
		
		\noi 
		provided 
		$$
		\max\left\{1,\frac2s\right\} < p < 
		\begin{cases}
			\infty &\text{if } d=1,2, \\
			\frac{d}{d-2} &\text{if } d\geq 3.	
		\end{cases}
		$$	
	\end{lemma}

	\begin{proof}
		For $d=1$, this is contained in~\cite[Lemma~3.2]{LNR18} and~\cite[Lemma~3.3]{DR23}. By the change of variable \eqref{chan-vari}, we observe that
		$$
		\L^{-1}(x,x) = \sum_{n\ge 0} \frac{e^2_n(x)}{\ld^2_n} \in L^p(\R^d) 
		$$
		is equivalent to 
		\begin{align} \label{GreenEst-proof}
			r^{-(d-1)\left(1-\frac{1}{p}\right)} \L_1^{-1}(r,r) = r^{-(d-1)\left(1-\frac{1}{p}\right)} \sum_{n\ge 0} \frac{g^2_n(r)}{\ld^2_n} \in L^p((0,+\infty),dr),
		\end{align}
		where $g_n(r) = r^{\frac{d-1}{2}} e_n(r)$ are eigenfunctions of $\L_1$. 
		
		We proceed in several steps.
		
		{\bf Step 1. A decay property near the origin.} We first show the following estimate of the Green function: for any $0<\beta<1$, there exists $C_\beta>0$ such that for all $r>0$,
		\begin{align} \label{beta}
			r^{-\beta} \L_1^{-1}(r,r) \leq C_\beta. 
		\end{align}
		In fact, since $\L_1 \geq C(H+1)$ with $H$ as in \eqref{H}, the operator monotonicity of $x\mapsto x^{-1}$ gives $\L_1^{-1}\leq C^{-1} (H+1)^{-1}$, hence
		\begin{align}\label{boun-H1-L}
			\L_1^{-1}(r,r) \leq C^{-1} (H+1)^{-1}(r,r).
		\end{align}
		We will use the Gamma function \eqref{Gamma} to express the integral kernel of $(H+1)^{-1}$ in terms of the heat kernel of $e^{-tH}$. More precisely,  we have
		\[
		(H+1)^{-1} = \frac{1}{\Gamma(1)}\int_0^{+\infty} e^{-tH} e^{-t}dt
		\]
		which implies
		\begin{align*}
			(H+1)^{-1}(r,\tau) &= \frac{1}{\Gamma(1)}\int_0^{+\infty} G(t,r,\tau) e^{-t}dt \\
			&= \frac{1}{\Gamma(1)}\int_0^{+\infty} \frac{\sqrt{r\tau}}{2t} \exp\left(-\frac{r^2+\tau^2}{4t}\right) I_\nu\left(\frac{r\tau}{2t}\right) e^{-t}dt,
		\end{align*}
		where $G$ is as in \eqref{heat-kern}. We write
		\[
		(H+1)^{-1}(r,r) = \frac{1}{\Gamma(1)}\int_0^{+\infty} \left(\frac{r^2}{2t}\right)^{\frac{1}{2}} \exp\left(-\frac{r^2}{2t}\right) I_\nu\left(\frac{r^2}{2t}\right) e^{-t} (2t)^{-\frac{1}{2}}dt 
		\]
		and use \eqref{I-nu-zero}, \eqref{I-nu-infi} to deduce
		\[
		(H+1)^{-1}(r,r) \leq C \int_0^{+\infty} e^{-t} t^{-\frac{1}{2}}dt <\infty.
		\]
		In particular, we have
		\begin{align} \label{L-infty}
			\L_1^{-1}(r,r) \in L^\infty((0,+\infty),dr).
		\end{align}
		Now we write
		\begin{align*}
			r^{-\beta} (H+1)^{-1}(r,r) &= \frac{1}{\Gamma(1)} \int_0^{+\infty} r^{-\beta} 
			\left(\frac{r^2}{2t}\right)^{\frac{1}{2}} \exp\left(-\frac{r^2}{2t}\right) I_\nu\left(\frac{r^2}{2t}\right) e^{-t} (2t)^{-\frac{1}{2}}dt \\
			&=\frac{1}{\Gamma(1)} \int_0^{+\infty} \left(\frac{r^2}{2t}\right)^{\frac{1}{2}-\frac{\beta}{2}} \exp\left(-\frac{r^2}{2t}\right) I_\nu\left(\frac{r^2}{2t}\right) e^{-t} (2t)^{-\frac{1}{2}-\frac{\beta}{2}}dt.
		\end{align*}
		According to the asymptotic behaviors of the modified Bessel function (see \eqref{I-nu-zero} and \eqref{I-nu-infi}), we split the integral into three parts
		\begin{align*}
			\Omega_1 &= \left\{t \in (0,+\infty) : \frac{r^2}{2t} \leq C_1 \right\}, \\
			\Omega_2 &= \left\{t\in (0,+\infty) : C_1 \leq \frac{r^2}{2t} \leq C_2 \right\}, \\
			\Omega_3 &= \left\{t \in (0,+\infty) : \frac{r^2}{2t} \geq C_2\right\},
		\end{align*}
		for some $C_1 \ll 1$ and $C_2 \gg 1$. On $\Omega_2$, it is clear that
		\[
		\left(\frac{r^2}{2t}\right)^{\frac{1}{2}-\frac{\beta}{2}} \exp\left(-\frac{r^2}{2t}\right) I_\nu\left(\frac{r^2}{2t}\right) \lesssim 1.
		\]
		On $\Omega_1$, we use \eqref{I-nu-zero} to get
		\[
		\left(\frac{r^2}{2t}\right)^{\frac{1}{2}-\frac{\beta}{2}} \exp\left(-\frac{r^2}{2t}\right) I_\nu\left(\frac{r^2}{2t}\right) \lesssim \left(\frac{r^2}{2t}\right)^{\frac{1}{2}-\frac{\beta}{2}+\nu} \exp\left(-\frac{r^2}{2t}\right) \lesssim 1.
		\]
		On $\Omega_3$, by \eqref{I-nu-infi}, we have
		\[
		\left(\frac{r^2}{2t}\right)^{\frac{1}{2}-\frac{\beta}{2}} \exp\left(-\frac{r^2}{2t}\right) I_\nu\left(\frac{r^2}{2t}\right) \lesssim \left(\frac{r^2}{2t}\right)^{-\frac{\beta}{2}} \lesssim 1.  
		\]
		Thus we obtain
		\[
		r^{-\beta} (H+1)^{-1}(r,r)\lesssim \int_0^{+\infty} e^{-t} t^{-\frac{1}{2}-\frac{\beta}{2}}dt <\infty
		\]
		as $0<\beta<1$. This proves \eqref{beta}.
		
		{\bf Step 2. The superharmonic case $s>2$.} By \eqref{traceEst} and \eqref{TraceL1}, we see that ${\rm Tr}[\L_1^{-1}]<\infty$. Hence $\L_1^{-1}(r,r) \in L^1((0,+\infty),dr)$, which together with \eqref{L-infty} implies
		\begin{align} \label{Lp-L1}
			\L_1^{-1}(r,r) \in L^p((0,+\infty),dr), \quad \forall p \in [1,\infty].
		\end{align}
		We now estimate for $p>1$,
		\begin{align}
			\|(\cdot)^{-(d-1)\left(1-\frac{1}{p}\right)} \L_1^{-1}(\cdot,\cdot)\|^p_{L^p((0,+\infty),dr)} &= \int_0^{+\infty} \left(r^{-(d-1)\left(1-\frac{1}{p}\right)} \L_1^{-1}(r,r) \right)^p dr \nonumber\\
			&= \int_0^1 \left(r^{-(d-1)\left(1-\frac{1}{p}\right)} \L_1^{-1}(r,r) \right)^p dr \nonumber\\
			&\quad + \int_1^{+\infty} \left(r^{-(d-1)\left(1-\frac{1}{p}\right)} \L_1^{-1}(r,r) \right)^p dr. \label{Lp-est-proof}
		\end{align}
		The integration on $(1,+\infty)$ is finite using \eqref{Lp-L1}. On the other hand, using \eqref{beta}, we have
		\begin{align*}
			\int_0^1 \left(r^{-(d-1)\left(1-\frac{1}{p}\right)} \L_1^{-1}(r,r) \right)^p dr &=\int_0^1 \left(r^{-(d-1)\left(1-\frac{1}{p}\right)+\beta} r^{-\beta} \L_1^{-1}(r,r) \right)^p dr \\
			&\lesssim \int_0^1 r^{-\left((d-1)\left(1-\frac{1}{p}\right)-\beta\right)p} dr
		\end{align*}
		which is finite as long as
		\[
		\left((d-1)\left(1-\frac{1}{p}\right)-\beta\right)p <1
		\]
		which we guarantee by picking
		\[
		p<\frac{d}{d-1-\beta}.
		\]
		Since $\beta$ can be chosen arbitrarily in $(0,1)$, we infer that the integration on $(0,1)$ is ensured to be finite provided
		\begin{align}\label{Lp-est-01}
			1<p<\frac{d}{d-2}.
		\end{align}
		
		{\bf Step 3. The (sub)-harmonic case $1<s\leq 2$.} It remains to show the boundedness of the integration on $(1,+\infty)$ in \eqref{Lp-est-proof}. To do this, we consider separately $d\geq 3$ and $d=2$. 
		
		$\boxed{d\geq 3}$ We observe that
		$$
		\L_1 \geq -\partial^2_r + r^s
		$$
		so 
		\begin{align} \label{est-Green-d-geq3}
			\L_1^{-1}(r,r) \leq (-\partial^2_r +r^s)^{-1}(r,r).
		\end{align} 
		
		To compute the Green function of $-\partial^2_r +r^s$ on the half-line $(0,+\infty)$, we take $\varphi \in C^\infty_0((0,+\infty))$. The unique solution to 
		\begin{align} \label{equ-As}
			\left\{
			\begin{array}{rcl}
				(-\partial^2_r + r^s) u &=& \varphi, \quad r>0, \\
				u(0) &=& 0,
			\end{array}
			\right.
		\end{align}
		is given by
		\begin{align} \label{As}
			u(r) = (-\partial^2_r +r^s)^{-1}\varphi(r) = \int_0^{+\infty} (-\partial^2_r+r^s)^{-1}(r,\tau) \varphi(\tau) d\tau.
		\end{align}
		We use the odd extension technique to extend \eqref{equ-As} to the whole line. More precisely, we denote
		\begin{align} \label{odd-exte-u}
			u_{\text{odd}}(r) := \left\{
			\begin{array}{cl}
				u(r) &\text{if } r>0, \\
				-u(-r) &\text{if } r<0, \\
				0 &\text{if } r=0,
			\end{array}
			\right. \quad \varphi_{\text{odd}}(r) :=\left\{
			\begin{array}{cl}
				\varphi(r) &\text{if } r>0, \\
				-\varphi(-r) &\text{if } r<0, \\
				0 &\text{if } r=0,
			\end{array}
			\right. 
		\end{align}
		the odd extensions of $u$ and $\varphi$ respectively. Then $u_{\text{odd}}$ solves
		\[
		(-\partial^2_r +|r|^s) u_{\text{odd}} = \varphi_{\text{odd}}, \quad r \in (-\infty, +\infty)
		\]
		which admits a unique solution given by
		\[
		u_{\text{odd}}(r) = (-\partial^2_r + |r|^s)^{-1} \varphi_{\text{odd}}(r) = \int_{-\infty}^{+\infty} K(r,\tau) \varphi_{\text{odd}}(\tau) d\tau,
		\]
		where $K(r,\tau)$ is the Green function of $-\partial^2_r +|r|^s$ on the whole line $\R$. Restricting to $(0,+\infty)$, we get for $r>0$,
		\begin{align*}
			u(r) &= \int_{-\infty}^{+\infty} K(r,\tau) \varphi_{\text{odd}}(\tau) d\tau \\
			&= \int_{-\infty}^0 K(r,\tau) \left(-\varphi(-\tau)\right) d\tau + \int_0^{+\infty} K(r,\tau) \varphi(\tau) d\tau \\
			&= \int_0^{+\infty} \left(K(r,\tau) - K(r,-\tau) \right) \varphi(\tau) d\tau.
		\end{align*}
		Comparing to \eqref{As}, we obtain
		\[
		(-\partial^2_r +r^s)^{-1}(r,\tau) = K(r,\tau) - K(r,-\tau), \quad r, \tau \in (0,+\infty).
		\]
		Thanks to the fact that $K(r,r) \in L^p((-\infty,+\infty),dr)$ for all $\frac{2}{s}<p\leq \infty$ due to \cite[Lemma 3.2]{LNR18} and
		\[
		K(r,\tau) = \sum_{n\geq 1} \mu_n^{-2} l_n(r)l_n(\tau),
		\] 
		where $(\mu_n^2,l_n)$ is the eigenpair of $-\partial^2_r +|r|^s$ on $\R$, we infer that
		\[
		(-\partial^2_r +r^s)^{-1}(r,r) \in L^p((0,+\infty),dr), \quad \forall p \in \left(\frac{2}{s},\infty\right].
		\]
		This together with \eqref{est-Green-d-geq3} show that the integration on $(1,+\infty)$ is finite for all $\frac{2}{s}<p\leq \infty$. Combining with \eqref{Lp-est-01} yields $r^{-(d-1)\left(1-\frac{1}{p}\right)} H_1^{-1}(r,r)\in L^p((0,+\infty),dr)$ for all
		\[
		\frac{2}{s}<p<\frac{d}{d-2}.
		\]
		
		$\boxed{d=2}$ In this case, we recall from the proof of Lemma~\ref{LEM:main1} that for $0<\theta < 1$
		\[
		-\partial^2_r -\frac{1}{4r^2} + r^s \geq C^{-1} (-\partial^2_r)^\theta + \frac{1}{2} r^s
		\]
		for some constant $C>0$. In particular, we have 
		\begin{align} \label{est-Green-2d}
			\L_1^{-1}(r,r) \leq \left( C^{-1}(-\partial^2_r)^\theta +\frac{1}{2} r^s\right)^{-1}(r,r).
		\end{align}
		Using an odd extension as above, we have
		$$
		\left( C^{-1}(-\partial^2_r)^\theta +\frac{1}{2} r^s\right)^{-1}(r,\tau) = K(r,\tau) - K(r,-\tau), \quad \forall r,\tau\in (0,+\infty),
		$$
		where $K(r,\tau)$ is the integral kernel of $\left(C^{-1}(-\partial^2_r)^\theta +\frac{1}{2} |r|^s\right)^{-1}$ defined on $\R$. Thanks to an estimate on the Green function of the fractional Schr\"odinger operator with an anharmonic oscillator (see Proposition \ref{PROP:GreenLp}), we have for $\theta \in (1/2,1)$, $K(r,r) \in L^p(\R)$ for all $\frac{2\theta}{s(2\theta-1)}<p\leq \infty$. This implies that
		\[
		\left( C^{-1}(-\partial^2_r)^\theta +\frac{1}{2} r^s\right)^{-1}(r,r) \in L^p((0,+\infty),dr), \quad \forall p \in \left(\frac{2\theta}{s(2\theta-1)},\infty\right].
		\]
		Since $\theta$ can be taken arbitrarily in $(1/2,1)$, we deduce from \eqref{est-Green-2d} that
		$$
		\L_1^{-1}(r,r) \in L^p((0,+\infty),dr), \quad \forall p \in \left(\frac{2}{s},\infty\right].
		$$
		From this, we can conclude as in the case $d\geq 3$. The proof is complete.
	\end{proof}
	
	Lemmas \ref{LEM:main2} and \ref{LEM:main3} are crucial in the proof of Theorem \ref{THM:main} (i), i.e. the normalizability in the subcritical cases. To deal with the critical/supercritical cases, we  also need tighter estimates, in particular estimates on the number of eigenvalues below a given threshold. 
	
	\subsection{Weyl-type asymptotics for radial Schr\"odinger operators}
	\label{subsec:Weyl}
	Let 
	\[
	N(\L, \Ld) : = \# \{\ld_n^2: \ld_n^2 \le \Ld \},
	\]
	where $\ld_n^2$ are the eigenvalues of $\L$ given in Lemma \ref{LEM:main1}. We start with the following, whose equivalent in 1D is known as the ``Cwikel-Lieb-Rozenbljum" bound (see e.g. \cite[Lemma D.1]{DR23}). 
	
	\begin{theorem}[\textbf{Weyl's law for radial Schr\"odinger operators}]\label{theo-Weyl-rad}\mbox{}\\
		Let $d\geq 2$ and $s>0$. Then, for two constants $c,C >0$ 
		\begin{align} \label{Weyl-rad}
			c \Lambda^{\frac{1}{2}+\frac{1}{s}} \leq N(\L,\Ld) \leq C \Lambda^{\frac{1}{2}+\frac{1}{s}} \quad \text{ as } \Lambda \to \infty.
		\end{align}
	\end{theorem}
	
	We use the method of coherent states and semi-classical analysis on the phase space $\R\times\R$, whose basic ingredients we recall in Appendix~\ref{SEC:Weyl}. In the following, the single integral sign stands for the integral over the configuration space $\R$ and the double integral one is for the integration over the phase space $\R \times \R$. 
	
	\begin{proof}[Proof of Theorem~\ref{theo-Weyl-rad}]
		We have $N(\L,\Lambda) = N(\L_1, \Lambda)$, where $\L_1$ is given in \eqref{L1}. Thus, the analysis is reduced to study $\L_1$, which we do in several steps.
		
		{\bf Step 1. An odd extension.} It is convenient to extend the operator $\L_1$ to the whole line. More precisely, let $u$ be an eigenfunction of $\L_1$ with the eigenvalue $\ld^2$, i.e.,
		$$
		-\partial^2_r u + \frac{(d-1)(d-3)}{4r^2} u + r^s u = \ld^2 u
		$$
		with $u \in L^2((0,+\infty),dr)$ satisfying $u(0)=0$. We extend $u$ to the whole line using the odd extension, denoted $u_{\text{odd}}$ (see \eqref{odd-exte-u}). In particular, $u_{\text{odd}}$ is a solution to 
		$$
		-\partial^2_x u +\frac{(d-1)(d-3)}{4|x|^2}u + |x|^s u = \ld^2 u,
		$$
		where $u \in L^2(\mathbb{R})$ is an odd function. In particular, $(\ld^2, u_{\text{odd}})$ is an eigenpair of
		\[
		\bar{\L}_1 = -\partial^2_x + \frac{(d-1)(d-3)}{4|x|^2} + |x|^s
		\]
		acting on $L^2_{\text{odd}}(\mathbb{R})$ the space of odd $L^2$-functions. In addition, we have 
		\begin{align}\label{odd-exte}
			N(\L_1, \Ld) = N(\bar{\L}_1, \Ld).
		\end{align} 
		
		{\bf Step 2. A semiclassical reduction.} To estimate $N(\bar{\L}_1,\Lambda)$, we use a suitable scaling to reduce the problem to a semiclassical one that counts number of eigenvalues of a semiclassical operator below 1. More precisely, if $u$ is an eigenfunction of $\bar{\L}_1$ with the eigenvalue $\ld^2$, then $v(x):= u(\Lambda^{\frac{1}{s}} x)$ solves 
		$$
		\Lambda^{-1-\frac{2}{s}} \left(-\partial^2_x +\frac{(d-1)(d-3)}{4|x|^2}\right) v + |x|^s v = \lambda^2 \Lambda^{-1} v.
		$$
		Set 
		\begin{align} \label{hbar}
			\hbar := \Lambda^{-\frac{1}{2}-\frac{1}{s}}, \quad \mu^2 = \lambda^2 \Lambda^{-1}.
		\end{align} 
		Then $v$ solves the semiclassical equation
		$$
		\bar{\L}_{1,\hbar} v = \mu^2 v, \quad \bar{\L}_{1,\hbar} := \hbar^2 \left(-\partial^2_x +\frac{(d-1)(d-3)}{4|x|^2}\right) + |x|^s.
		$$
		In particular, 
		\begin{align} \label{scaling}
			N(\bar{\L}_1,\Lambda)= N(\bar{\L}_{1,\hbar},1).
		\end{align}
		The study of $N(\bar{\L}_{1,\hbar},1)$ is now a semiclassical problem. To proceed further, we consider the following quantum energy: 
		\begin{equation}\label{eq:quantum energy}
			E^{\qua}_{\hbar} := \inf \left\{\mathcal E^{\qua}_{\hbar}[\gamma] = {\rm Tr}[(\bar{\L}_{1,\hbar} -1)\gamma] : \gamma \in \mathfrak{S}^1(L^2_{\text{odd}}), 0 \leq \gamma \leq \ind\right\},
		\end{equation}
		where $\mathfrak{S}^1(L^2_{\text{odd}}(\R))$ is the space of trace-class operators on $L^2_{\text{odd}}(\R)$. Let $(u^{\hbar}_n)_{n\geq 0}$ be an orthonormal basis of $L^2_{\text{odd}}(\R)$ consisting of eigenfunctions of $\bar{\L}_{1,\hbar}$, with associated eigenvalues $(\lambda_n^\hbar)^2$. One readily sees that the quantum energy in \eqref{eq:quantum energy} is achieved by the following fermionic density matrix
		\begin{align} \label{gamma-h}
			\gamma_{\hbar} = \sum_{n\geq 0} \ind_{ (\lambda_n^\hbar)^2 \leq 1} |u^{\hbar}_n\rangle \langle u^{\hbar}_n|.
		\end{align} 
		In particular, we have
		\[
		{\rm Tr}[\gamma_{\hbar}] = \sum_{(\lambda_n^{\hbar})^2 \leq 1} 1 = N(\bar{\L}_{1,\hbar},1).
		\]
		Our goal is to show that
		\begin{align}\label{limi-trac-h}
			\sup_{\hbar \to 0} \hbar {\rm Tr}[\gamma_{\hbar}] \in (0,\infty)	
		\end{align}
		which together with the scaling \eqref{hbar} and \eqref{scaling} yield \eqref{Weyl-rad}.
		
		{\bf Step 3. A lower bound.} We next aim at proving the following lower bound
		\begin{align} \label{lowe-boun-trace}
			\liminf_{\hbar \to 0} \hbar {\rm Tr}[\gamma_{\hbar}] \geq C
		\end{align}
		for some constant $C>0$. 
		
		Fix a constant $K\in (0,1)$ and define the trial state
		\[
		\gamma^{\test} =\frac{1}{2\pi \hbar} \iint m_{K}(x,p) |f^\hbar_{x,p}\rangle \langle f^\hbar_{x,p}| dxdp,
		\]
		where
		\begin{align}\label{m0-test}
			m_{K}(x,p):=\ind_{\{|p|^2 + |x|^s-1\leq 0\} \cap \{|x|\geq K\}}
		\end{align}
		and $f^\hbar_{x,p}$ is the coherent state (see Appendix \ref{SEC:Weyl} for the definition). By \eqref{reso-iden}, we see that $0\leq \gamma^{\test} \leq \ind$ and 
		\[
		{\rm Tr}[\gamma^{\test}] = \frac{1}{2\pi\hbar} \iint m_{K}(x,p) \jb{f^{\hbar}_{x,p}, f^{\hbar}_{x,p}} dxdp = \frac{1}{2\pi \hbar} \iint m_{K}(x,p)dxdp \leq C(K,\hbar).
		\]
		Note that
		\begin{align}\label{inte-m0}
			\begin{aligned}
				\iint m_{K}(x,p) dxdp &= \int \left(\int \ind_{\{|p|^2 + |x|^s-1\leq 0\} \cap \{|x|\geq K\}}dp\right) dx \\
				&= \int_{|x|\geq K} \left(\int_{\{|p|\leq (|x|^s-1)_-^{1/2}\}} dp\right) dx\\
				&= 2\int_{|x|\geq K} (|x|^s-1)_-^{1/2} dx <\infty,
			\end{aligned}
		\end{align}
		where $V_-(x) = \max \left\{-V(x),0\right\}$ is the negative part of $V(x)$. 
		
		We will use the energy of $\gamma^{\test}$ as an upper bound to the quantum energy:
		$$E^{\qua}_{\hbar} \leq \mathcal E^{\qua}_{\hbar} \left[\gamma ^{\test}\right].$$
		By the Plancherel identity, we compute
		\begin{align*}
			\hbar {\rm Tr}[-\hbar^2 \partial^2_x \gamma^{\test}] &=\frac{1}{2\pi} \iint m_{K}(x,p) \jb{f^{\hbar}_{x,p}, -\hbar^2\partial^2_y f^{\hbar}_{x,p}} dxdp \\
			&= \frac{1}{2\pi}\iint m_{K}(x,p)\left(\int |q|^2 |\mathcal F_{\hbar}[f^\hbar_{x,p}](q)|^2 dq\right)dxdp \\
			&= \frac{1}{2\pi} \iint m_{K}(x,p)\left(\hbar^{-1/2}\int |q|^2 \left|\hat{f}\left(\frac{q-p}{\sqrt{\hbar}}\right)\right|^2 dq\right)dxdp \\
			&=\frac{1}{2\pi}\iint m_{K}(x,p)\left(\int |p+\sqrt{\hbar}q|^2 |\hat{f}(q)|^2 dq\right)dxdp \\
			&=\frac{1}{2\pi}\iint m_K(x,p) \left(\int (|p|^2 + 2\sqrt{\hbar}pq + \hbar |q|^2)|\hat{f}(q)|^2 dq\right)dxdp,
		\end{align*}
		where $\mathcal F_{\hbar}$ is the semi-classical Fourier transform (see Appendix \ref{SEC:Weyl} for the definition). We observe that
		$$
		\int |p|^2 |\hat{f}(q)|^2 dq = |p|^2 \|\hat{f}\|^2_{L^2}=|p|^2
		$$
		and
		$$
		\int |q|^2 |\hat{f}(q)|^2 dq = \|f'\|^2_{L^2}.
		$$
		Since $f$ is an odd function, so is $\hat{f}$, hence for $p$ fixed,
		\begin{align*}
			\int pq |\hat{f}(q)|^2 dq = 0.  
		\end{align*}
		In particular, we obtain
		\begin{align}\label{iden-1}
			\hbar {\rm Tr}[-\hbar^2 \partial^2_x \gamma^{\test}] = \frac{1}{2\pi} \iint m_K(x,p) (|p|^2 + \hbar \|f'\|^2_{L^2}) dxdp. 
		\end{align}
		We next have
		\begin{align*}
			\hbar {\rm Tr}\left[\hbar^2 \frac{(d-1)(d-3)}{4|x|^2} \gamma^{\test}\right] = \frac{(d-1)(d-3) \hbar^2}{8\pi} \iint m_K(x,p) \left \langle f^{\hbar}_{x,p}, \frac{1}{|y|^2}f^{\hbar}_{x,p}\right \rangle dxdp.
		\end{align*}
		It suffices to estimate the bracket on the region $\{|x|\geq K\}$ (see \eqref{m0-test}). We have
		\begin{align*}
			\left \langle f^{\hbar}_{x,p}, \frac{1}{|y|^2}f^{\hbar}_{x,p}\right \rangle &= \int \frac{1}{|y|^2}|f^{\hbar}_{x,p}(y)|^2 dy \\
			&= \hbar^{-1/2}\int \frac{1}{|y|^2}\Big|f\Big(\frac{y-x}{\sqrt{\hbar}}\Big)\Big|^2 dy \\
			&= \int \frac{1}{|x+\sqrt{\hbar}y|^2} |f(y)|^2 dy.
		\end{align*}
		Since $|x|\geq K$ and $f$ has compact support, there exists $\hbar_0\in (0,1]$ small such that
		\[
		|x+\sqrt{\hbar}y|\geq |x| -\sqrt{\hbar}|y| \geq \frac{K}{2}
		\]
		for all $\hbar \in (0,\hbar_0)$ and all $y \in \text{supp}(f)$. It follows that  
		\[
		\left \langle f^{\hbar}_{x,p}, \frac{1}{|y|^2}f^{\hbar}_{x,p} \right \rangle \leq \frac{4}{K^2} \|f\|^2_{L^2} =\frac{4}{K^2},
		\]
		hence
		\begin{align} \label{iden-2}
			\hbar {\rm Tr}\left[\hbar^2 \frac{(d-1)(d-3)}{4|x|^2} \gamma^{\test}\right] \leq C_K \hbar^2 \iint m_K(x,p)dxdp
		\end{align}
		for all $\hbar \in (0,\hbar_0)$. We also have
		\begin{align*}
			\hbar {\rm Tr}[|x|^s \gamma^{\test}] &=\frac{1}{2\pi} \iint m_K(x,p) \jb{f^{\hbar}_{x,p}, |y|^s f^{\hbar}_{x,p}} dxdp \\
			&= \frac{1}{2\pi} \iint m_K(x,p) \left( \int |y|^s |f^{\hbar}_{x,p}(y)|^2 dy\right) dxdp \\
			&= \frac{1}{2\pi} \iint m_K(x,p) \left( \hbar^{-1/2}\int |y|^s \Big|f\Big(\frac{y-x}{\sqrt{\hbar}}\Big)\Big|^2 dy\right) dxdp \\
			&=\frac{1}{2\pi} \iint m_K(x,p) \left( \int |x+\sqrt{\hbar}y|^s |f(y)|^2 dy\right) dxdp.
		\end{align*}
		Using the inequality
		\begin{align} \label{diff-est}
			|x+\sqrt{\hbar} y|^s \leq \left\{
			\begin{array}{ll}
				|x|^s + C |\sqrt{\hbar}y|^s + C \sqrt{\hbar}|x|^{s-1} |y| &\text{if } s\geq 1, \\
				|x|^s + C |\sqrt{\hbar} y|^s &\text{if } 0<s<1,
			\end{array}
			\right.
		\end{align}
		we get
		\begin{align*}
			\int |x+\sqrt{\hbar}y|^s |f(y)|^2 dy \leq |x|^s + C\hbar^{s/2} \int |y|^s |f(y)|^2 dy + C\hbar^{1/2} |x|^{s-1} \int |y||f(y)|^2dy
		\end{align*}
		for $s\geq 1$ and
		\begin{align*}
			\int |x+\sqrt{\hbar}y|^s |f(y)|^2 dy \leq |x|^s + C\hbar^{s/2} \int |y|^s |f(y)|^2 dy
		\end{align*}
		for $0<s<1$. In the following, we consider only the case $s\geq 1$ since the one for $0<s<1$ is similar. In particular, we obtain for $s\geq 1$,
		\begin{align}\label{iden-3}
			\begin{aligned}
				\hbar {\rm Tr}[|x|^s \gamma^{\test}] \leq \frac{1}{2\pi} \iint |x|^s m_K(x,p) dxdp &+ C \hbar^{1/2} \||y|^{1/2} f\|^2_{L^2} \iint |x|^{s-1} m_K(x,p)dxdp \\
				&+ C \hbar^{s/2} \||y|^{s/2} f\|^2_{L^2} \iint m_K(x,p) dxdp.
			\end{aligned}
		\end{align}
		Collecting \eqref{iden-1}, \eqref{iden-2}, \eqref{iden-3}, and using the resolution of the identity \eqref{reso-iden} yields
		\begin{align} \label{trial-boun}
			\begin{aligned}
				\hbar E^{\qua}_{\hbar} &\leq \hbar \mathcal E^{\qua}_{\hbar}[\gamma^{\test}] \\
				&= \hbar {\rm Tr}[(\bar{\L}_{1,\hbar}-1)\gamma^{\test}] \\
				&\leq \frac{1}{2\pi} \iint (|p|^2 + |x|^s -1) m_K(x,p) dxdp \\
				&\quad + C\left(\hbar^2 + \hbar^{s/2}\||y|^{s/2}f\|^2_{L^2} \right)\iint m_K(x,p)dxdp \\
				&\quad + C\hbar^{1/2} \||y|^{1/2}f\|^2_{L^2}\iint |x|^{s-1} m_K(x,p) dxdp
			\end{aligned}
		\end{align}
		for all $\hbar \in (0,\hbar_0)$ and some constant $C>0$ independent of $\hbar$. Since (see \eqref{inte-m0})
		\[
		\iint m_K(x,p)dxdp, \quad \iint |x|^{s-1} m_K(x,p) dxdp <\infty,
		\]
		letting $\hbar \to 0$, we obtain
		\begin{align} \label{uppe-boun-energy}
			\limsup_{\hbar \to 0} \hbar E^{\qua}_{\hbar} \leq \frac{1}{2\pi} \iint (|p|^2 + |x|^s -1) m_K(x,p) dxdp.
		\end{align}
		Now let $\gamma_{\hbar}$ be a minimizer for $E^{\qua}_{\hbar}$, i.e., $E^{\qua}_{\hbar} = \mathcal E^{\qua}_{\hbar}[\gamma_{\hbar}]$. Since $\bar{\L}_{1,\hbar} \geq 0$, we have
		\[
		\hbar E^{\qua}_{\hbar} = \hbar \mathcal E^{\qua}_{\hbar}[\gamma_{\hbar}] \geq -\hbar {\rm Tr}[\gamma_{\hbar}].
		\]
		It follows that
		\[
		\liminf_{\hbar \to 0} \hbar {\rm Tr}[\gamma_{\hbar}] \geq -\limsup_{\hbar \to 0} \hbar E^{\qua}_{\hbar} \geq -\frac{1}{2\pi} \iint (|p|^2 +|x|^s-1) m_K(x,p) dxdp.
		\]
		This proves \eqref{lowe-boun-trace} since
		\[	
		-\frac{1}{2\pi} \iint (|p|^2 +|x|^s-1) m_K(x,p) dxdp = \frac{2}{3\pi} \int_{|x|\geq K} (|x|^s-1)_-^{3/2} dx \in (0,\infty).
		\]
		
		{\bf Step 4. An upper bound.} This step is devoted to the following upper bound
		\begin{align} \label{uppe-boun-trace}
			\limsup_{\hbar \to 0} \hbar {\rm Tr}[\gamma_\hbar] \leq C
		\end{align}
		for some constant $C>0$.  Let 
		$$\gamma_\hbar = \sum_{n\geq 0} \mu_n^{\hbar} |u^{\hbar}_n\rangle \langle u^{\hbar}_n|$$
		be a minimizer for $E^{\qua}_{\hbar}$. We first denote 
		$$ \rho_{\hbar}(x) := \sum_n \mu_n^{\hbar} |u^{\hbar}_n(x)|^2$$
		and claim that, for $\hbar$ small enough, 
		\begin{align} \label{densi-gamma}
			\hbar \int |x|^s \rho_{\hbar}(x) dx \leq \hbar \int \rho_\hbar(x) dx.
		\end{align}		
		Indeed, thanks to the energy upper bound \eqref{uppe-boun-energy}, there exists $\hbar_0 \in (0,1]$ such that 
		\begin{align} \label{nega-energy}
			\hbar E^{\qua}_{\hbar} \leq 0, \quad \forall\hbar \in (0,\hbar_0).
		\end{align} 
		Since $-\partial^2_x +\frac{(d-1)(d-3)}{4|x|^2} \geq 0$, we have
		\begin{align*}
			0 \geq \hbar E^{\qua}_{\hbar} &= \hbar \mathcal E^{\qua}_{\hbar}[\gamma_{\hbar}] \\
			&= \hbar {\rm Tr}[(\bar{\L}_{1,\hbar}-1)\gamma_{\hbar}]\\
			&\geq \hbar {\rm Tr}[(|x|^s-1)\gamma_\hbar] \\
			&= \hbar \int (|x|^s-1)\rho_\hbar(x) dx
		\end{align*}
		which gives the desired estimate.

		We next define the Husimi function associated to $\gamma_{\hbar}$ by setting
		\[
		m_{\hbar}(x,p) := \jb{f^{\hbar}_{x,p}, \gamma_{\hbar}f^{\hbar}_{x,p}}
		\]
		and refer to Appendix~\ref{SEC:Weyl} for its' properties.
		
		Now let us consider separately two cases: $d\geq 3$ and $d=2$.
		
		\medskip
		
		$\boxed{d\geq 3}$ We have
		\begin{align}\label{boun-d-geq3-1}
			\hbar E^{\qua}_{\hbar} =\hbar \mathcal E^{\qua}_{\hbar}[\gamma_{\hbar}] = \hbar {\rm Tr}[(\bar{\L}_{1,\hbar}-1)\gamma_{\hbar}] \geq \hbar {\rm Tr}[(-\hbar^2\partial^2_x +|x|^s-1)\gamma_{\hbar}]. 
		\end{align}
		By Plancherel's identity, we compute
		\begin{align*}
			{\rm Tr}[-\hbar^2 \partial^2_x \gamma_{\hbar}] 
			= \sum_{n\geq 0} \mu_n^{\hbar} \jb{u^{\hbar}_n, -\hbar^2\partial^2_x u^{\hbar}_n}
			= \int |q|^2 t_\hbar(q)dq,
		\end{align*}
		where
		$$
		t_{\hbar}(q) =\sum_{n\geq 1} \mu_n^{\hbar} |\mathcal F_{\hbar}[u_n^{\hbar}] (q)|^2.
		$$
		We also have
		\begin{align*}
			{\rm Tr}[(|x|^s-1)\gamma_{\hbar}] 
			= \sum_{n\geq 0} \mu_n^{\hbar} \int (|x|^s-1)|u^{\hbar}_n(x)|^2 dx 
			= \int (|x|^s-1) \rho_{\hbar}(x)dx.  
		\end{align*}
		Thus we get
		\begin{align}\label{boun-d-geq3-2}
			\hbar {\rm Tr}[(-\hbar^2\partial^2_x +|x|^s-1)\gamma_{\hbar}] = \hbar \left(\int |q|^2 t_\hbar(q)dq + \int (|x|^s-1)\rho_{\hbar}(x) dx\right).
		\end{align}
		On the other hand, we have (see Lemma \ref{lem:Husimi})
		\begin{align*}
			\frac{1}{2\pi} \iint m_{\hbar}(x,p) |p|^2dxdp 
			&= \hbar \int |p|^2 t_{\hbar}\ast|g^{\hbar}|^2(p)dp \\
			&=\hbar \int t_{\hbar}(q) \left(\int |p|^2 |g^{\hbar}(p-q)|^2 dp\right) dq \\
			&= \hbar \int t_{\hbar}(q) \left(\hbar^{-1/2}\int |p|^2 \Big|\hat{f}\Big(\frac{p-q}{\sqrt{\hbar}}\Big) \Big|^2 dp \right) dq \\
			&= \hbar \int t_{\hbar}(q) \left(\int (|q|^2+2\sqrt{\hbar}pq + \hbar |p|^2) |\hat{f}(p)|^2dp \right) dq \\
			&= \hbar \int t_{\hbar}(q) \left(|q|^2+ \hbar \|f'\|^2_{L^2} \right) dq \\
			&= \hbar \int t_{\hbar}(q) |q|^2 dq + \hbar^2 \|f'\|^2_{L^2} \int t_\hbar(q) dq,
		\end{align*}
		where we have performed a similar calculation as in Step 3. We also have
		\begin{align*}
			\frac{1}{2\pi}\iint m_{\hbar}(x,p)(|x|^s-1) dxdp 
			&=\hbar \int (|x|^s-1) \rho_{\hbar}\ast |f^{\hbar}|^2(x) dx \\
			&= \hbar \int \rho_{\hbar}(y) \left(\int (|x|^s-1) |f^{\hbar}(x-y)|^2 dx\right) dy \\
			&= \hbar \int \rho_{\hbar}(y) \left(\int (|y+\sqrt{\hbar}x|^s-1) |f(x)|^2 dx\right) dy \\
			&= \hbar \int \rho_{\hbar}(y) (|y|^s-1) dy \\
			&\quad + \hbar \int \rho_{\hbar}(y) \left( \int |y+\sqrt{\hbar} x|^s|f(x)|^2 dx - |y|^s \right) dy,
		\end{align*}
		where we used $\|f\|_{L^2}=1$. To estimate the second term, we use \eqref{diff-est} to have
		\begin{align*}
			\Big|\int |y+\sqrt{\hbar}x|^s|f(x)|^2 dx - |y|^s\Big| &\leq C \hbar^{1/2} |y|^{s-1} \int |x| |f(x)|^2 dx + C \hbar^{s/2} \int |x|^s|f(x)|^2 dx \\
			&\leq C\hbar^{1/2} |y|^{s-1} \||x|^{1/2}f\|^2_{L^2} + C\hbar^{s/2} \||x|^{s/2} f\|^2_{L^2}.
		\end{align*}
		This shows that
		\begin{align*}
			\Big|\hbar \int \rho_{\hbar}(y) \Big( \int |y+\sqrt{\hbar} x|^s|f(x)|^2 dx - |y|^s \Big) dy\Big| &\leq C\hbar^{3/2} \||x|^{1/2}f\|^2_{L^2} \int |y|^{s-1}\rho_{\hbar}(y)dy \\
			&\quad + C\hbar^{s/2+1}\||x|^{s/2}f\|^2_{L^2} \int \rho_{\hbar}(y) dy.
		\end{align*}
		Using the inequality $|y|^{s-1} \leq C(|y|^s+1)$ and $s\geq 1$ (there is no such term if $0<s<1$), we infer from \eqref{densi-gamma} that
		\begin{align*}
			\Big|\hbar \int \rho_{\hbar}(y) \Big( \int |y+\sqrt{\hbar} x|^s|f(x)|^2 dx - |y|^s \Big) dy\Big| = O(\hbar^{3/2}) \int \rho_{\hbar}(y)dy
		\end{align*}
		for all $\hbar \in (0,\hbar_0)$, where $\hbar_0$ is as in \eqref{nega-energy} and $A_\hbar=O(\hbar^\alpha)$ means that $|A_\hbar|\leq C\hbar^\alpha$ for some constant $C>0$ independent of $\hbar$. Collecting the above estimates, we have
		\begin{align} 
			\frac{1}{2\pi}\iint m_{\hbar}(x,p) (|p|^2+|x|^s-1) dxdp &= \hbar \left(\int t_\hbar(q)|q|^2 dq + \int \rho_\hbar(x) (|x|^s-1) dx\right) \label{boun-d-geq3-3} \\
			&\quad + O(\hbar^2) \int t_\hbar(q) dq + O(\hbar^{3/2}) \int \rho_\hbar(x) dx. \nonumber
		\end{align}
		From \eqref{nega-energy}, \eqref{boun-d-geq3-1}, \eqref{boun-d-geq3-2}, and \eqref{boun-d-geq3-3}, we obtain
		\begin{align}
			0\geq \hbar E^{\qua}_{\hbar} &\geq \hbar \left(\int |q|^2 t_\hbar(q)dq + \int (|x|^s-1)\rho_{\hbar}(x) dx\right) \nonumber\\
			& \geq \frac{1}{2\pi}\iint m_{\hbar}(x,p)(|p|^2+|x|^s-1) dxdp  - C \hbar^2 \int t_{\hbar}(q) dq - C\hbar^{3/2} \int \rho_{\hbar}(x) dx \nonumber\\
			&\geq \frac{1}{2\pi}\iint m_{\hbar}(x,p)(|p|^2+|x|^s-1) dxdp  - C \hbar \int t_{\hbar}(q) dq - C\hbar \int \rho_{\hbar}(x) dx \nonumber\\
			&= \frac{1}{2\pi}\iint m_{\hbar}(x,p)(|p|^2+|x|^s-1-2C) dxdp \label{boun-d-geq3-4}
		\end{align}
		for all $\hbar \in (0,\hbar_0)$, where we used \eqref{momen-densi-relation} to get the last identity.
		
		Set $$W(x,p):= |p|^2+|x|^s-1-2C \geq -1-2C, \quad \forall x,p \in \R.$$ From \eqref{boun-d-geq3-4}, we have
		\begin{align*}
			0 &\geq \frac{1}{2\pi}\iint m_{\hbar}(x,p) W(x,p) dxdp \\
			&= \frac{1}{2\pi}\iint_{\{W\leq 1\}} m_{\hbar}(x,p) W(x,p) dxdp + \frac{1}{2\pi}\iint_{\{W\geq 1\}} m_{\hbar}(x,p) W(x,p) dxdp \\
			&\geq -\frac{1+2C}{2\pi} \iint_{\{W\leq 1\}} m_\hbar(x,p) dx dp + \frac{1}{2\pi}\iint_{\{W\geq 1\}} m_{\hbar}(x,p)dxdp \\
			&= \frac{1}{2\pi}\iint m_{\hbar}(x,p) dxdp -\frac{1+C}{\pi} \iint_{\{W\leq 1\}} m_{\hbar}(x,p)dxdp \\
			&\geq \frac{1}{2\pi}\iint m_{\hbar}(x,p) dxdp - \frac{1+C}{\pi} \iint_{\{W\leq 1\}} dxdp,
		\end{align*}
		where we used $0\leq m_\hbar(x,p)\leq 1$ for all $x,p \in \R$. Therefore, we obtain
		\begin{align} \label{boun-m}
			\frac{1}{2\pi}\iint m_{\hbar}(x,p) dxdp \leq \frac{1+C}{\pi} \iint_{\{W\leq 1\}} dxdp = \text{constant}
		\end{align}
		for some $C>0$ since $\{W\leq 1\}$ is a non-empty compact set of $\R^2$. This together with \eqref{momen-densi-relation} proves \eqref{uppe-boun-trace} when $d\geq 3$.  
		
		$\boxed{d=2}$ In this case, we first extend the refined Hardy inequality \eqref{Hardy-ineq-2d} to the whole line, i.e., for $0<\theta<1$,
		\[
		\int |(-\partial^2_x)^{\theta/2} f(x)|^2 dx \leq C \int\Big|\Big(-\partial^2_x -\frac{1}{4|x|^2}\Big)^{\theta/2}f(x)\Big|^2 dx, \quad \forall f \in C^\infty_0(\R) \text{ odd functions}.
		\]
		As operators, we have from the operator monotonicity of $x\mapsto x^\theta$ with $0<\theta<1$ that
		\begin{align*}
			(-\hbar^2\partial^2_x)^\theta &\leq C \left[\hbar^2\left(-\partial^2_x -\frac{1}{4|x|^2}\right)\right]^\theta \\
			&\leq C \left[ \hbar^2\left(-\partial^2_x -\frac{1}{4|x|^2}\right) +1\right]^\theta\\
			&\leq C \left[\hbar^2\left(-\partial^2_x -\frac{1}{4|x|^2}\right) +1 \right]
		\end{align*}
		which gives
		\[
		\hbar^2 \left(-\partial^2_x -\frac{1}{4|x|^2}\right) \geq C(-\hbar^2\partial^2_x)^\theta -1.
		\]
		In particular, we have
		\[
		\bar{H}_{1,\hbar} \geq C(-\hbar^2\partial^2_x)^\theta +|x|^s -1.
		\]
		for some constant $C>0$ which may change from line to line. By taking $0<\theta<1/2$, we get
		\begin{align} \label{uppe-boun-2d-1}
			\hbar E^{\qua}_{\hbar} =\hbar \mathcal E^{\qua}_{\hbar}[\gamma_{\hbar}] = \hbar {\rm Tr}[(\bar{\L}_{1,\hbar}-1)\gamma_{\hbar}] \geq \hbar {\rm Tr}[ (C(-\hbar^2\partial^2_x)^\theta +|x|^s-2)\gamma_\hbar].
		\end{align}
		The same argument as in the case $d\geq 3$ applies here. The only difference is the term
		\[
		\frac{1}{2\pi}\iint m_{\hbar}(x,p) |p|^{2\theta} dxdp
		\]
		which can be treated as follows. We have
		\begin{align*}
			\frac{1}{2\pi}\iint m_{\hbar}(x,p) |p|^{2\theta} dxdp	
			&=\hbar \int t_\hbar(q) \left(\int |p|^{2\theta} |g^{\hbar}(p-q)|^2 dp \right) dq \\
			&=\hbar \int t_\hbar (q) \left(\int |q+\sqrt{\hbar} p|^{2\theta} |\hat{f}(p)|^2 dp \right) dq \\
			&= \hbar \int t_\hbar(q) |q|^{2\theta} dq + O(\hbar^{1+\theta}) \int t_{\hbar}(q)dq.
		\end{align*}
		We now can repeat the same reasoning as in the case $d\geq 3$ (i.e., \eqref{boun-d-geq3-4} with $|p|^{2\theta}$ in place of $|p|^2$ and $-2$ instead of $-1$) to get the upper bound \eqref{uppe-boun-trace} when $d=2$. 
		
		Finally, by the lower bound \eqref{lowe-boun-trace} and the upper bound \eqref{uppe-boun-trace}, we show \eqref{limi-trac-h}. This completes the proof of Theorem \ref{theo-Weyl-rad}.
	\end{proof}

	\begin{remark}[\bf Precise Weyl asymptotics when $d\geq 3$] \label{rem-rad-d-geq3} \mbox{} \\
		When $d\geq 3$, the above proof actually gives the asymptotic behavior
		\begin{align} \label{preci-limi-trace}
			\lim_{\hbar \to 0} \hbar {\rm Tr}[\gamma_\hbar] = \frac{1}{2\pi} \iint m_0(x,p) dxdp,
		\end{align}
		where 
		\[
		m_0(x,p) = \ind_{\{|p|^2+|x|^s-1\leq 0\}}.
		\]
		To see this, we define the classical energy
		$$
		E^{\cl} := \inf \left\{\mathcal E^{\cl}[m]:=\frac{1}{2\pi}\iint m(x,p)(|p|^2+|x|^s-1) dxdp : m \in L^1(\R\times \R), 0\leq m\leq 1\right\}.
		$$
		By the bathtub principle (see e.g. \cite[Theorem 1.14]{LL01}), the unique minimizer for $E^{\cl}$ is given by the function $m_0$ defined above. By \eqref{uppe-boun-energy} and letting $K \to 0$, we have
		\begin{align*}
			\limsup_{\hbar \to 0} \hbar E^{\qua}_{\hbar} &\leq \frac{1}{2\pi} \iint (|p|^2 + |x|^s -1) m_K(x,p) dxdp \\
			&\leq \frac{1}{2\pi} \iint (|p|^2 + |x|^s -1) m_0(x,p) dxdp = E^{\cl},
		\end{align*}
		where $m_0$ is as in \eqref{m0-test}. On the other hand, from \eqref{momen-densi-relation}, \eqref{boun-d-geq3-3}, and \eqref{boun-m}, we have
		\[
		\hbar E^{\qua}_{\hbar} =\hbar \mathcal E^{\qua}_{\hbar}[\gamma_\hbar] =\frac{1}{2\pi} \iint m_\hbar(x,p)(|p|^2+|x|^s-1) dxdp + O(\hbar^{1/2}) \geq E^{\cl} - C\hbar^{1/2}.
		\]
		Taking the liminf, we obtain
		$$
		\liminf_{\hbar \to 0} \hbar E^{\qua}_\hbar \geq E^{\cl}.
		$$
		This shows that
		$$
		\lim_{\hbar} \hbar E^{\qua}_\hbar = E^{\cl}
		$$
		and
		$$
		\iint m_\hbar(x,p)(|p|^2+|x|^s) dxdp \leq C
		$$
		for all $\hbar \to 0$. In particular, $m_\hbar$ is a tight minimizing sequence for $E^{\cl}$. From this, we deduce that $m_\hbar \to m_0$ strongly in $L^1(\R \times \R)$ and \eqref{preci-limi-trace} follows.		
	\end{remark}
	
	As a corollary, we deduce an estimate on the number of eigenvalues in some spectral windows.

	\begin{corollary} 
		\label{LEM:CLR}
		There exist $C_0, c_0>0$ and $k_0>1$ such that for all $\Ld \ge \ld_0^2$,
		\begin{align}
			\label{CLR2}
			c_0 \Ld^{\frac12 +\frac1s} \le \# \{ \ld_n^2 : \Ld < \ld_n^2 \leq k_0 \Ld\} \le C_0 \Ld^{\frac12 +\frac1s}.
		\end{align}
	\end{corollary}
	
	\begin{proof}
		%
		From~\eqref{Weyl-rad} or its' equivalent in 1D, there exist $C,c>0$ and $\Lambda_0>0$ large such that
		\begin{align}\label{Weyl-rad-appl}
			c \Lambda^{\frac12+\frac1s} \leq N(\mathcal L, \Lambda) \leq C \Lambda^{\frac12+\frac1s}
		\end{align}
		for all $\Lambda \geq \Lambda_0$. The same estimate still holds for $\ld^2_0 \leq \Ld \leq \Ld_0$ by adjusting the constants $C,c$ accordingly. In fact, we have
		\[
		N(\L,\Ld) \geq 1 \geq \Ld_0^{-\frac12 -\frac1s} \Lambda^{\frac12 + \frac1s}
		\]
		and
		\begin{align*}
			N(\L,\Ld) \leq N(\L,\Ld_0) \leq C\Ld_0^{\frac12 +\frac1s} = C\left(\frac{\Ld_0}{\Ld}\right)^{\frac12 +\frac1s} \Ld^{\frac12+\frac1s} \leq C \left(\frac{\Ld_0}{\ld^2_0}\right)^{\frac12 +\frac1s} \Ld^{\frac12 + \frac1s}.
		\end{align*}
		To see \eqref{CLR2}, we use~\eqref{Weyl-rad-appl} to have 
		$$
		\# \{\ld_n^2 : \Ld <\ld_n^2 \leq k_0 \Ld\} \leq N(\L,k_0 \Ld) \leq C k_0^{\frac12+\frac1s} \Ld^{\frac12+\frac1s}
		$$
		and 
		$$
		\begin{aligned}
			\# \{\ld_n^2 : \Ld <\ld_n^2 \leq k_0 \Ld\} &= N(\L, k_0\Ld)-N(\L,\Ld)\\
			&\geq c (k_0\Ld)^{\frac12+\frac1s}- C\Ld^{\frac12-\frac1s}\\
			&= \left(ck_0^{\frac12+\frac1s}-C\right)\Ld^{\frac12-\frac1s} \\
			&= c_0\Ld^{\frac12-\frac1s}
		\end{aligned}
		$$
		provided that $k_0>1$ is taken sufficiently large depending on $C,c$. 
	\end{proof}

	Let $\L_{N} = \P_{N} \L$, where $\P_N$ is given in \eqref{projN}. Then, we have
	\[
	{\rm Tr} [\L_{N}^{-p}] = \sum_{n =0}^N \ld_n^{-2p}. 
	\]
	Similarly, we define $\L_{N}^\perp = \L - \P_{N} \L$ and thus 
	\[
	{\rm Tr} [(\L_{N}^\perp)^{-p}] = \sum_{n =N+1}^\infty \ld_n^{-2p}. 
	\]
	
	The following corollaries are crucial in proving the non-normalizability.
	
	\begin{corollary}[\textbf{Behavior of truncated Schatten norms}]\label{COR:CLR}\mbox{}\\
		Let $d\ge 1$, $s > 0$ and $ p > 0$. We have
		\begin{align}
			\label{traceN}
			{\rm{Tr}} [\L^{-p}_{N}] = \sum_{n = 0}^N \ld_n^{-2p} \sim 
			\begin{cases}
				1 &\text{if } p >\frac12 + \frac1s,\\ 
				\left(\log \ld_N\right)^2 &\text{if }  p = \frac12 + \frac1s,  \\
				\ld_N^{-2p+ 1 + \frac2s} &\text{if }   p<\frac12 + \frac1s.
			\end{cases} 
		\end{align}
	\end{corollary}
	
	\begin{proof}
		When $p > \frac12 + \frac1s$, \eqref{traceN} follows from Lemma \ref{LEM:main2} and \eqref{traceEst}. In the following, we assume that $p \le \frac12 + \frac1s$.
		We first decompose the interval $[\ld_0,\ld_N]$ dyadically and then apply the Weyl bound~\eqref{CLR2} to get
		\begin{align*}
			\begin{split}
				{\rm{Tr}} [\L^{-p}_{N}] & = \sum_{n = 0}^N \ld_n^{-2p}   = \sum_{l=1}^{c\log \ld_N} \sum_{k_0^{-l} \ld_N^2  < \ld_n^2 \le k_0^{-l+1} \ld_N^2}  \ld_n^{-2p} \\
				& \sim \sum_{l=1}^{c\log \ld_N} (k_0^{-l} \ld_N^2)^{-p} \cdot \#  \{\ld_n^2: k_0^{-l} \ld_N^2 < \ld_n^2 \le k_0^{-l+1} \ld_N^2\} \\
				& \sim \sum_{l=1}^{c\log \ld_N} (k_0^{-l} \ld_N^2)^{-p} (k_0^{-l} \ld_N^2)^{ \frac12+ \frac1s} \\
				& \sim  
				\begin{cases} \left(\log \ld_N\right)^2  &\text{if } p = \frac12 + \frac1s,\\  
					\Big( \sum_{l=1}^{c\log \ld_N} k_0^{-l(\frac12 + \frac1s -p)} \Big) \cdot \ld_N^{-2p + 1 + \frac2s} &\text{if } p < \frac12 + \frac1s,
				\end{cases}
			\end{split}
		\end{align*}
		which gives the desired estimate.
	\end{proof}
	
	We also need the following:
	
	\begin{corollary}[\textbf{Tail estimate}]\label{COR:CLR1}\mbox{}\\
		Let $d\geq 1$, $s > 0$ and $ p > \frac12 + \frac1s$.
		Then, we have
		\begin{align}
			\label{traceN1}
			{\rm{Tr}} [(\L^\perp_N)^{-p}] = \sum_{n = N+1}^\infty \ld_n^{-2p} \sim 
			\ld_N^{-2p+ 1 + \frac2s}.
		\end{align}
	\end{corollary}
	
	\begin{proof}
		We have from \eqref{CLR2} that
		\begin{align}
			\begin{split}
				{\rm{Tr}} [(\L^\perp_N)^{-p}] & = \sum_{n = N+1}^\infty \ld_n^{-2p}  \sim \sum_{l=0}^{\infty} \sum_{k_0^{l} \ld_N^2  < \ld_n^2 \le k_0^{l+1} \ld_N^2}  \ld_n^{-2p} \\
				& \sim \sum_{l=0}^{\infty} (k_0^{l} \ld_N^2)^{-p} \cdot \#  \{\ld_n^2: k_0^{l} \ld_N^2 < \ld_n^2 \le k_0^{l+1} \ld_N^2\} \\
				& \sim \sum_{l=0}^{\infty} (k_0^{l} \ld_N^2)^{-p} (k_0^{l} \ld_N^2)^{ \frac12+ \frac1s} \\
				& \sim \ld_N^{-2p + 1 + \frac2s},
			\end{split}
		\end{align}
		provided $p > \frac12 + \frac1s$.
	\end{proof}
	
	We finally state the following observation.
	
	\begin{lemma}[\textbf{Polynomial growth of eigenvalues}]\label{LEM:asym}\mbox{}\\
		Let $d\geq 1$, $s > 0$ and $k\in \N$. Then, there exists $c(s) > 0$ such that
		\[
		\lim_{N\to \infty} \frac{\ld_{kN}}{\ld_{N}^{c(s)}} = 0.
		\]
	\end{lemma}
	
	\begin{proof}
		It comes from the fact that
		$$
		N = N(\L,\ld_N^2) \sim \ld_N^{1+\frac2s} \text{ or } \ld_N \sim N^{\frac{s}{2+s}} \text{ as } N \to \infty.
		$$
	\end{proof}

	\section{Variational formulation}\label{sec:variational}
	Most of our arguments proceed from a variational formula for the partition function that we recall from~\cite{BD98}. This relies on the Bou\'e-Dupuis variational formula~\cite{BD98,Ustunel14,Zhang09,Borell-00,Lehec-13}, which has already been used extensively to define and characterize nonlinear Gibbs measures in related contexts, see e.g. ~\cite{BG20,OOT, OOT1}.
	
	We wish to determine whether the partition function
	\[
	\cZ_K := \int \ind_{\{|M(u)|\leq K\}} e^{\al R_p(u)} d\mu (u) 
	\]
	is finite or not, where $\mu$ is the Gaussian measure with covariance $\L^{-1}$, restricted to radial functions when $d\geq 2$. Since (see e.g. \cite[Proposition 2.5]{BD98})
	\begin{multline}\label{eq:freeener} 
		- \log \cZ_K = \\
		\inf\left\{ \int \left(-\al R_p(u) +\log (f(u)) \right) f(u) d\mu_K (u) \,|\, f\geq 0, f \in L^1(d\mu_K), \int f (u) d\mu_K (u)= 1 \right\}
	\end{multline}
	with $d\mu_K (u) = \ind_{\{|M(u)|\leq K\}} d\mu (u)$, the finiteness of $\cZ_K$ is related to some free energy being bounded from below. Namely, the question is whether the positive entropy term (relative to the Gaussian measure) in the above is sufficient to compensate the possibly very negative potential energy term. The Bou\'e-Dupuis formula will help us decide that question by providing a more wieldy formulation of the entropy term in the variational principle above. This requires introducing an extra time variable and seeing the Gaussian measure $\mu$ as the law of a Brownian motion at time $1$. Then, Girsanov's theorem provides a description of random processes whose laws are absolutely continuous with respect to that of the Brownian motion, and the Bou\'e-Dupuis formula an expression of the latter's entropies relative to the Wiener measure. This turns~\eqref{eq:freeener} into a control problem. We set this up briefly in Section~\ref{sec:BouDup1} below, sketching how this is of use for our problem. In Section~\ref{sec:BouDup2}, we collect some consequences of the bounds of Section~\ref{sec:Schro} which legitimate the application of the formula to our context. 
	
	\subsection{The Bou\'e-Dupuis formula and its' use}\label{sec:BouDup1}
	
	Let $W(t)$ denote a cylindrical Brownian motion in $L^2(\R^d)$ defined by
	\begin{align}
		\label{Bro}
		W(t) = \sum_{n\ge 0} B_n (t) e_n,
	\end{align}
	
	\noi
	where $\{e_n\}_{n\ge 0}$ is the sequence of normalized eigenfunctions of the operator $\mathcal L$ given in Lemma \ref{LEM:main1}, 
	and $\{B_n\}_{n \ge 0}$ is a sequence of mutually independent complex-valued Brownian\footnote{Essentially, we have $B_n(t) \sim \mathcal N_{\mathbb C}(0,2t)$, its' density is given by $\frac{1}{2\pi t} e^{-\frac{|z|^2}{2t}}$.} motions.
	We define a centered Gaussian process $Y(t)$ by
	\begin{align}
		\label{Yt}
		Y (t) = \L^{-\frac12} W(t) = \sum_{n\ge 0} \frac{B_n (t)}{\ld_n} e_n.
	\end{align}
	
	\noi
	Then, 
	$Y (t)$ is well-defined in $\H^{-\s} (\R^d)$ for any $\s > \frac1s - \frac12 $, see Corollary \ref{COR:intp} below. 
	But $Y(t)$ is not in $L^2(\R^d)$ when $1 < s \le 2$, since we have
	\[
	\E \big[\|Y(t)\|_{L^2(\R)}^2\big] \sim \sum_{n\ge 0} \frac{\E [|B_n (t)|^2] }{\ld_n^2}   = 2t \sum_{n\ge 0} \ld_n^{-2}  \begin{cases}
		< \infty &\text{if } s>2, \\
		= \infty &\text{if } s\le 2 \text{ (unless } t=0).
	\end{cases}
	\]
	
	\noi
	From \eqref{Yt}, we see that
	\begin{align}
		\label{law}
		\textup{Law} (Y(1)) = \mu,
	\end{align} 
	where $\mu$ is the Gaussian free field given in \eqref{mu}.
	

	Let $\mathbb{H}_a$ be the space of drifts,
	which consists of progressively measurable processes belonging to 
	$L^2 \left([0,1]; L^2 (\R^d)\right)$,
	$\mathbb P$-almost surely.
	One of the key tools in this paper is the following Bou\'e-Dupuis variational formula \cite[Theorem 5.1]{BD98} (see also \cite{Ustunel14} and \cite{Zhang09}). 
	
	\begin{lemma} [\bf Bou\'e-Dupuis variational formula]  
		\label{LEM:var} \mbox{}\\
		Let $Y(t)$ be as in \eqref{Yt}. For $\s > \frac1s - \frac12$, let $F: \H^{-\s} (\R^d) \to \R$ be a Borel measurable function that is bounded from above. Then, we have
		\begin{align}
			\label{var}
			-\log \E \Big[ e^{-F(Y(1))} \Big] = \inf_{\theta \in \mathbb{H}_a} \E
			\bigg[F\big(Y(1) + I(\theta) (1) \big) + \frac12 \int_0^1 \| \theta (t) \|_{L^2(\R^d)}^2 dt\bigg],
		\end{align}

		\noi
		where $I (\theta)$ is defined by
		\[
		I (\theta) (t) = \int_0^t \L^{-\frac12} \theta (\tau) d\tau
		\]
		
		\noi
		and
		the expectation $\E = \E_{\mathbb P}$ is with respect to the underlying probability measure $\mathbb P$.
	\end{lemma}
	\begin{proof}
		The fact that \eqref{var} holds when 
		\begin{equation}\label{varRHS}
			\inf_{\theta \in \mathbb{H}_a} \E
			\bigg[F\big(Y(1) + I(\theta) (1) \big) + \frac12 \int_0^1 \| \theta (t) \|_{L^2(\R^d)}^2 dt\bigg] > -\infty
		\end{equation}
		follows with same proof as \cite[Proposition A.2]{TW23}. Therefore, we only need to show equality when $\eqref{varRHS} = - \infty$ instead. Let $M > 0$. By monotone convergence, there exists $\ta_M \in \mathbb{H}_a$ and $L >0$ such that 
		\begin{equation*}
			\E
			\bigg[\max\Big(F\big(Y(1) + I(\theta_M) (1) \big), -L\Big)  + \frac12 \int_0^1 \| \theta_M (t) \|_{L^2(\R^d)}^2 dt\bigg] < -M.
		\end{equation*}
		Therefore, we have that 
		\begin{align*}
			-\log \E \Big[ e^{-F(Y(1))} \Big] &\le -\log \E \Big[ e^{-\max(F(Y(1)),L)} \Big] \\
			& = \inf_{\theta \in \mathbb{H}_a} \E
			\bigg[\max\Big(F\big(Y(1) + I(\theta) (1) \big), -L\Big) + \frac12 \int_0^1 \| \theta (t) \|_{L^2(\R^d)}^2 dt\bigg] \\
			& \le \E
			\bigg[\max\Big(F\big(Y(1) + I(\theta_M) (1) \big), -L\Big)  + \frac12 \int_0^1 \| \theta_M (t) \|_{L^2(\R^d)}^2 dt\bigg] \\
			& < - M.
		\end{align*}
		Since $M$ is arbitrary, we conclude that 
		\begin{equation*}
			-\log \E \Big[ e^{-F(Y(1))} \Big] = -\infty = \inf_{\theta \in \mathbb{H}_a} \E
			\bigg[F\big(Y(1) + I(\theta) (1) \big) + \frac12 \int_0^1 \| \theta (t) \|_{L^2(\R^d)}^2 dt\bigg].
		\end{equation*}
	\end{proof}
	We shall use the above with 
	\begin{equation}\label{eq:usevar}
		F(Y(1)) = -\al R_p(Y_N(1)) \ind_{\{|M(Y_N(1))|\leq K\}} 
	\end{equation}
	or a close variant, where we set (see \eqref{projN} for the definition of $\P_N$) 
	\begin{equation} \label{eq:YN}
		Y_N (1)= \P_{N} Y (1).
	\end{equation}
	This is a legitimate choice, as per the estimates provided in Section~\ref{sec:BouDup2} below. A lower bound to the infimum in~\eqref{var}, uniform in $N$, will imply finiteness of $\cZ_K$ and hence normalizability of the interacting Gibbs measure, while an upper bound diverging to $-\infty$ when $N\to \infty$ will imply $\cZ_K = +\infty$, hence non-normalizability.
	
	\bigskip
	
	\noindent \textbf{Proving normalizability.} To bound~\eqref{var} from below, we must show that the second term is, for any drift, large enough to compensate the first one. Clearly, from Minkowski's inequality, we have that $I(\theta)(1)$ enjoys 
	the following pathwise regularity bound:
	
	\begin{lemma}[\bf Pathwise regularity] 
		\label{LEM:bounds} \mbox{} \\
		For any $\theta \in \mathbb{H}_a$,
		we have
		\begin{align}
			\label{I}
			\| I (\theta) (1)\|^2_{\mathcal H^1 (\R^d)} \le \int_0^1 \| \theta (t) \|_{L^2(\R^d)}^2 dt.
		\end{align}
	\end{lemma}
	
	Hence the general scheme is essentially to bound 
	\[
	F\big(Y_N(1) + \P_{N}I(\theta) (1) \big) \geq - \frac{1}{2} \| I (\theta) (1)\|^2_{\mathcal H^1 (\R^d)} - G (Y_N(1))
	\]
	where $G (Y_N(1))$ is an expression involving different quantities (e.g. Lebesgue or Sobolev norms) only related to the Gaussian process $Y_N(1)$. This uses a (case-dependent) suitable mix of functional inequalities and triangle inequalities. Inserting this and Lemma~\ref{LEM:bounds} in~\eqref{var} will yield normalizability provided 
	\[
	\E \left[ G (Y_N(1)) \right] < \infty 
	\]
	uniformly in $N$, i.e. essentially 
	\[
	\int  G (u) d\mu (u) < \infty
	\]
	with $\mu$ the original Gaussian measure. We obtain such conclusions with a suitable selection of estimates from Section~\ref{sec:Schro}.

	\bigskip
	
	\noindent \textbf{Proving non-normalizability.} To bound~\eqref{var} from above, a suitable drift $\theta (t)$ is constructed, making the first term diverge to $-\infty$, while keeping the second one under control. In practice both terms will diverge, so that one needs a construction making the first one diverge faster. Under our assumptions, there exists blow-up profiles $f_N$ such that 
	\[ 
	F (f_N) \underset{N\to \infty}{\to} - \infty.
	\]
	We pick one and construct a drift $\theta$ such that 
	\[ 
	Y_N (1) + \P_{N}I(\theta) (1) \simeq f_N + \mbox{ lower order }.
	\]
	In particular, it is tempting to let 
	\[
	\theta (t) = \L^{1/2} \left( -\partial_t Y_N (t) + f_N \right)
	\]
	so that 
	\[ 
	I(\theta) (t) = - Y_N (t) + f_N.
	\]
	Of course, the time-derivative of the Brownian motion in the above is problematic, leading to the second term in~\eqref{var} being $+\infty$ for such a choice. The idea is thus to approximate $Y_N (t)$ by a smoother process $Z_N (t)$ and let    
	\[
	\theta (t) := \L^{1/2} \left( -\partial_t Z_N (t) + f_N \right)
	\]
	be our trial drift. With a suitable construction we ensure that, for this choice of $\theta$, the first term of~\eqref{var} diverges to $-\infty$ faster than the second term diverges to $+\infty$. Namely we introduce, following ideas from~\cite{TW23}.
	
	\begin{definition}[\textbf{Approximate Brownian motion}]\label{def:gaussapprox}\mbox{}\\
		Given $ N \ge M\gg 1$,  we define the process $Z_M$ by its coefficients in the eigenfunction expansion $(e_n)_{n\geq 0}$ of (the radial restriction of) $\L$.
		
		For $n \leq M$ let $\wt Z_M(n, t)$ be a solution of the following  differential equation:
		\begin{align}
			\begin{cases}
				d \wt Z_{M}(n, t) = c \ld_n^{-1} \ld_M   (\wt Y_N (n,t)- \wt Z_{M}(n, t)) dt \\
				\wt Z_{M}|_{t = 0} =0, 
			\end{cases}
			\label{ZZZ}
		\end{align}
		where $c>0$ is a constant to be chosen later on and\footnote{We use our liberty of choosing a real eigenbasis for $\L$.}
		$$ \wt Y_N (n,t) = \int_{\R^d} Y_N(t,x) e_n(x) dx. $$
		We set $\wt Z_{M}(n, t)  \equiv 0$ for $n > M$ and define
		$$Z_M(t,x) := \sum_{n \le M} \wt Z_M(n,t) e_n(x)$$ 
		which is a centered Gaussian process in $L^2({\R^d})$ and satisfies $\P_N Z_M = Z_M$.  
	\end{definition}
	
	Let us briefly explain why the above yields a suitable approximation to the Gaussian process $Y(t)$ defined in~\eqref{Yt}. For illustration, we only consider a one-dimensional Gaussian variable 
	$$Y (t) = \sigma W(t),$$
	where $W(t)$ is the standard Brownian motion. Thus $Y = Y(1) \sim \mathcal N(0, 2\sigma^2)$ is a Gaussian variable with variance $2\sigma^2$. Using a one-dimensional version of Lemma~\ref{LEM:var}, we have
	\[
	-\log \left( \mathbb E [e^{-F(Y)}] \right) = \sup_{Z \in \mathbb H^1_a} \mathbb E \bigg[ F(Y(1) - Z(1)) - \frac{\sigma^2}2 \int_0^1 |\dot Z (s)|^2  ds\bigg], 
	\]
	where $\mathbb H_a^1$ is a set of stochastic processes defined as
	\[
	\mathbb H_a^1 = \big\{ Z: Z(0) = 0,  \dot Z \in L^2 ([0,1]\times \Omega), \textup{ and } Z \textup{ is progressively measurable} \big\}.
	\]
	We want to construct a stochastic process $Z \in \mathbb H_a^1$ such that $Y(1) - Z(1)$ is small, while keeping $\int_0^1 |\dot Z (s)|^2  ds$ under control. For this purpose, we postulate an ansatz for the difference 
	$$X (t) = Y(t) - Z(t) = \sigma W (t) - Z (t)$$
	being the It\^o process 
	\begin{align} \label{app}
		dX = - A X dt + \sigma dW,
	\end{align}
	or
	\[
	X(t) = \int_0^t e^{-A (t-s)} \sigma dW(s)
	\]
	with $A \gg 1$. Then  $X$ is ``small" in the sense that 
	\begin{equation}\label{eq:good}
		\E \left[|X(t)|^2\right] \sim \sigma^2/A. 
	\end{equation}
	On the other hand
	\begin{equation}\label{eq:bad}
		\E \left[ \int_0^1 |\dot Z (s)|^2  ds\right] = A^2 \E \left[ \int_0^1 |\dot X (s)|^2  ds\right] \sim \sigma^2 A 
	\end{equation}
	Rewriting~\eqref{app} as 
	\[
	dZ = - A (Y - Z) dt.
	\]
	serves as inspiration for our choice in~\eqref{ZZZ}. The parameters $A$ and $\s$ (or their variants) should be determined by our later analysis, to balance the competing effects of~\eqref{eq:good} and~\eqref{eq:bad} when they are inserted in the variational principle.

	\subsection{Preliminary estimates}\label{sec:BouDup2}
	
	With the above notation, we have the following consequence of Lemmas \ref{LEM:main2} and \ref{LEM:main3}. They will help us vindicating that the choice~\eqref{eq:usevar} can indeed be inserted in Lemma~\ref{LEM:var}.
	
	\begin{corollary} [\bf Regularity and integrability of the Gaussian process]
		\label{COR:intp} \mbox{} \\
		Let $d\geq 1$, $s>0$ and assume the radial condition when $d\ge 2$. The following statements hold:
		\begin{itemize}
			\item
			[\textup{(i)}] Let $1\le q <\infty$ and 
			$$
			\max\left\{2,\frac4s\right\} < p < 
			\begin{cases}
				\infty &\text{if } d=1,2, \\
				\frac{2d}{d-2} &\text{if } d\ge 3.
			\end{cases}
			$$
			Then, we have
			\[
			\E \big[\| Y_N (1) \|_{L^p(\R^d)}^q \big] \les_{d,p,q}  1 ,
			\]
			
			\noi
			where the constant depends only on $d, p,q$. In particular, $Y_N(1)$ is a Cauchy sequence in $L^q(\Omega, L^p(\R^d))$ and
			$$
			\E \big[\|Y(1)-Y_N(1)\|^q_{L^p(\R^d)} \big] \to 0 \text{ as } N \to \infty.
			$$
			
			\item
			[\textup{(ii)}] Let $\dl > -\frac12 + \frac1s$ and $1 \le q < \infty$.
			Then, we have 
			\[
			\E \big[\| Y_N (1) \|_{\mathcal H^{-\dl}(\R^d)}^q \big] \les_{d,q}  1.
			\]
			
			\noi
			In addition, $Y_N(1)$ is a Cauchy sequence in $L^p(\Omega, \mathcal H^{-\dl}(\R^d))$ and
			$$
			\E \big[\|Y(1)-Y_N(1)\|^q_{\mathcal H^{-\dl}(\R^d)} \big] \les_{d,q} \ld_N^{-\left(\frac12 +\dl -\frac1s \right) q}.
			$$
			
		\end{itemize}
	\end{corollary}
	
	\begin{proof}
		(i) We only consider the case $q\geq p$ since the one where $1\leq q <p$ follows from H\"older's inequality. By the Minkowski with $q\ge p$ and the Khintchine inequality (see e.g., \cite[Lemma 4.2]{BT08a}) and Lemma \ref{LEM:main3}, we have
		\begin{align*}
			\E \big[\|Y_N(1)\|_{L^p(\R^d)}^q \big] &= \Big\| \Big\|\sum_{n=0}^N \frac{B_n(1)}{\ld_n} e_n \Big\|_{L^p(\R^d)}\Big\|_{L^q(\O)}^q \\
			&\le \Big\| \Big\|\sum_{n=0}^N \frac{B_n(1)}{\ld_n} e_n \Big\|_{L^q(\O)}\Big\|_{L^p(\R^d)}^q \\
			&\le \bigg(C(q) \Big\| \Big\|\sum_{n= 0}^N \frac{B_n(1)}{\ld_n} e_n\Big\|_{L^2(\O)} \Big\|_{L^p(\R^d)}\bigg)^q \\
			&= \bigg(C(q) \Big\| \Big(\sum_{n= 0}^N \frac{e_n^2}{\ld_n^2} \Big)^{1/2} \Big\|_{L^p(\R^d)}\bigg)^q \\
			&= \bigg( C(q) \bigg(\int_{\R^d} \Big( \sum_{n=0}^N \frac{e^2_n(x)}{\ld_n^2}\Big)^{p/2} dx\bigg)^{1/p}\bigg)^q \\
			&\le \bigg( C(q) \bigg(\int_{\R^d} \left( \L^{-1}(x,x)\right)^{p/2} dx\bigg)^{1/p}\bigg)^q \\
			&\le C(d,p,q)
		\end{align*}
		provided 
		$$
		\max \left\{1, \frac2s\right\} < p < \begin{cases}
			\infty &\text{if } d=1,2, \\
			\frac{d}{d-2} &\text{if } d\ge 3.
		\end{cases}
		$$
		For $M>N$, we have
		\begin{align*}
			\E \big[ \|Y_M(1)-Y_N(1)\|^q_{L^p(\R^d)}\big] &\leq \bigg( C(q) \bigg(\int_{\R^d} \Big( \sum_{n=N+1}^M \frac{e^2_n(x)}{\ld_n^2}\Big)^{p/2} dx\bigg)^{1/p}\bigg)^q \\
			&\leq \bigg( C(q) \bigg(\int_{\R^d} \Big( \sum_{n=N+1}^\infty \frac{e^2_n(x)}{\ld_n^2}\Big)^{p/2} dx\bigg)^{1/p}\bigg)^q
		\end{align*}
		which converges to zero when $N \to \infty$ by the dominated convergence theorem. 
		
		(ii) It suffices to consider $q\ge 2$ since the one where $1\le q <2$ follows from H\"older's inequality and the case $q=2$. By the Minkowski inequality and the Khintchine inequality, we have from Lemma \ref{LEM:main2} that
		\begin{align*}
			\E \big[\|Y_N(1)\|^q_{\mathcal H^{-\dl}(\R^d)} \big] &= \E\Big[\Big\|\sum_{n=0}^N \frac{B_n(1)}{\ld_n^{1+\dl}} e_n \Big\|^q_{L^2(\R^d)}\Big] \\
			&\le \Big\| \Big\| \sum_{n=0}^N \frac{B_n(1)}{\ld_n^{1+\dl}} e_n \Big\|_{L^q(\Omega)}\Big\|_{L^2(\R^d)}^q \\
			&\le C(q) \Big\| \Big( \sum_{n=0}^N \frac{e^2_n(x)}{\ld_n^{2+2\dl}}\Big)^{1/2} \Big\|_{L^2(\R^d)}^q \\
			&= C(q) \Big( \sum_{n=0}^N \ld_n^{-2(1+\dl)} \Big)^{q/2}\\
			&\le C(d,q)
		\end{align*}
		provided $1+\dl > \frac12 + \frac1s$ or $\dl > -\frac12 + \frac1s$. In a similar manner, for $M>N$, we have 
		\begin{align*}
			\E \big[\|Y_M(1)-Y_N(1)\|^q_{\mathcal H^{-\dl}(\R^d)} \big] &\leq C(q) \Big( \sum_{n=N+1}^M \ld_n^{-2(1+\dl)} \Big)^{q/2} \\
			&\leq C(q) \Big( \sum_{n=N+1}^\infty \ld_n^{-2(1+\dl)} \Big)^{q/2} \\
			&\leq C(q) \ld_N^{-\left(\frac12 +\dl -\frac1s \right) q},
		\end{align*}
		where we have used Corollary \ref{COR:CLR1} to get the last estimate. 
	\end{proof}

	\begin{corollary}[\bf Integrability of Wick renormalized mass]
		\label{COR:WCE} \mbox{} \\
		Let $d\geq 1$, $s>\frac23$, $p\ge 1$ and assume the radial condition when $d\ge 2$. 
		Then, we have
		\[
		\Big\|\int_{\R^d} \wick{|Y_N (1)|^2} dx \Big\|_{L^p(\O)} \les_{d,p} 1.
		\]
		
		\noi
		where the constant depends only on $d, p$. Moreover, the sequence $\int_{\R^d} \wick{|Y_N (1)|^2}dx$ is Cauchy in $L^p(\O)$ and  
		$$ 
		\Big\|\int_{\R^d} \wick{|Y (1)|^2} dx - \int_{\R^d} \wick{|Y_N (1)|^2} dx \Big\|_{L^p(\O)} \les_{d,p} \ld_N^{- \frac32 + \frac1s}.
		$$
	\end{corollary}
	
	\begin{proof}
		By the H\"older inequality, it suffices to consider the case $p\ge 2$. Using the Wiener chaos estimate for $p \ge 2$ (see e.g., \cite[Theorem I.22]{Simon15} or \cite[Proposition 2.4]{TT10}), we have
		\[
		\Big\|\int_{\R^d} \wick{|Y_N (1)|^2} dx \Big\|_{L^p(\O)}  \les_p  \Big\|\int_{\R^d} \wick{|Y_N (1)|^2} dx \Big\|_{L^2(\O)} .
		\]
		
		\noi
		Then from the definition of Wick order \eqref{Wick} and \eqref{variance},
		it follows that
		\begin{align}
			\begin{split}
				\int_{\R^d} \wick{|Y_N (1)|^2} dx & = \int_{\R^d} |Y_N (1)|^2 dx - \int_{\R^d} \s_N (x) dx\\
				& = \sum_{n=0}^N \frac{|B_n(1)|^2 - 2}{\ld_n^2}.
			\end{split}
			\label{wickY}
		\end{align}
		
		\noi
		Since the random variable $\{B_n(1)\}$ are normalized and independent, we have
		\[
		\E [(|B_{n}(1)|^2 - 2)^2] = \E [|B_{n}(1)|^4] - 4 [|B_{n}(1)|^2] + 4 = 4
		\]
		
		\noi
		and
		\[
		\E [(|B_{n_1}(1)|^2 - 2)(|B_{n_2}(1)|^2 - 2)] = 0
		\]
		
		\noi
		for all $n_1 \neq n_2$.
		Therefore, using Corollary \ref{COR:CLR}, we have
		\begin{align*}
			\E \bigg[ \Big|\int_{\R^d}  \wick{|Y_N (1)|^2} dx \Big|^2 \bigg] &= \sum_{n=0}^N \frac{\E [(|B_n(1)|^2 - 2)^2]}{\ld_n^4}\\
			& = \sum_{n=0}^N \frac{4}{\ld_n^4} < \infty
		\end{align*}
		provided $2>\frac12 +\frac1s$ or $s>\frac23$.
		
		We now move to the second part. Proceeding similarly, for $M > N$, we have from Corollary \ref{COR:CLR1} that 
		\begin{align*}
			\Big\|\int_{\R^d} \wick{|Y_M (1)|^2} dx - \int_{\R^d} \wick{|Y_N (1)|^2} dx \Big\|_{L^p(\O)}^p & \les_p 
			\Big(\sum_{n=N+1}^M \ld_n^{-4}\Big)^\frac p2 \les \Big(\ld_N^{- \frac32 + \frac1s}\Big)^p,
		\end{align*}
		which shows that the sequence $\int_{\R^d} \wick{|Y_N (1)|^2}dx$ is Cauchy in $L^p(\O)$ for $s > \frac23$. We conclude the proof by taking a limit as $M \to \infty$.
	\end{proof}
	
	\section{Subharmonic potential}
	\label{SEC:subharmonic}
	This section concerns the normalizability and non-normalizability 
	of the focusing Gibbs measure with subharmonic potential $s<2$. 
	
	\subsection{Normalizability}
	\label{SEC:nor}
	In this subsection, we show the integrability part of Theorem \ref{THM:main}. The conclusion will be obtained in Section~\ref{SUB:nor} below from the main estimate we now state, in the form of a bound on $\mathcal Z_{K,N}$:

	\begin{lemma}[\textbf{Exponential Integrability}]\label{lem:exp int}\mbox{}\\
		Let $s<2$ and assume 
		\begin{align}
			\label{condition1}
			\begin{split}
				\textup{either (i) }& \frac{4}{s} < p < 2 + \frac{4s}{(d-1)s+2}; \\
				\textup{or   (ii) } & p = 2+ \frac{4s}{(d-1)s+2} \textup{ in the weakly nonlinear regime}.
			\end{split}
		\end{align}
		Then 
		\begin{align}
			\label{var1d}
			\sup_N \E_{\mu}  \Big[ \exp (\al R_p(u_N)) \cdot  \ind_{\{ | \int_{\R^d} \wick{|u_N (x)|^2} dx| \le K\}}\Big] < \infty,
		\end{align}
		where $u_N = \P_{N} u$ and $R_p(u_N)$ is the potential energy denoted by
		\begin{align}
			\label{RpN}
			R_p (u_N) : = \frac{1}{p} \int_{\R^d} |u_N (x)|^p dx.
		\end{align}
	\end{lemma}
	
	\noi 
	Observing that
	\begin{align*}
		\E_{\mu}  \Big[ \exp (\al R_p(u_N)) \cdot  \ind_{\{ | \int_{\R^d} \wick{|u_N (x)|^2} dx| \le K\}} \Big]
		\le \E_{\mu}  \bigg[ \exp \Big(\al R_p(u_N) \cdot  \ind_{\{ | \int_{\R^d} \wick{|u_N (x)|^2} dx| \le K\}} \Big)\bigg],
	\end{align*}
	
	\noi
	the bound \eqref{var1d} follows once we have
	\begin{align}
		\label{var3}
		\begin{split}
			\sup_N   \E_{\mu}  \bigg[ \exp \Big(\al R_p(u_N) \cdot  \ind_{\{ | \int_{\R^d} \wick{|u_N (x)|^2} dx| \le K\}} \Big)\bigg]  < \infty.
		\end{split}
	\end{align}
	
	\noi
	To prove \eqref{var3}, we apply the Bou\'e-Dupuis variational formula Lemma \ref{LEM:var} with $F(Y(1))$ as in \eqref{eq:usevar} and use the fact that 
	$$\textup{Law}(Y_N(1)) = (\P_N)_*\mu$$ 
	to get
	\begin{align}
		\label{var4}
		\begin{split}
			& -  \log  \E_{\mu}  \Big[ \exp \Big(\al R_p(u_N) \cdot  \ind_{\{ | \int_{\R^d} \wick{|u_N |^2} dx| \le K\}}  \Big)\Big]\\
			&\hphantom{X} = \inf_{\theta \in \mathbb{H}_a} \E \Big[ - \al R_p \big( Y_N(1)+ \P_{N} I(\theta) (1) \big) 
			\cdot  \ind_{\{ | \int_{\R^d} \wick{| Y_N(1)  + \P_{N} I(\theta) (1) |^2} dx|  \le K \}} \\
			& \hphantom{XXXXXXX} + \frac12 \int_0^1 
			\| \theta (t) \|_{L^2(\R^d)}^2 dt  \Big].
		\end{split}
	\end{align}
	
	\noi 
	Here, $\E_{\mu}$ and $\E$ denote expectations with respect to the Gaussian field $\mu$ and the underlying probability measure $\mathbb P$ respectively.
	In what follows, we will denote
	\begin{align}  
		Y_N = Y_N(1)=\P_N Y(1), \quad \Theta_N = \P_{N} I(\theta)(1)
		\label{nots}
	\end{align}
	
	\noi 
	for simplicity.
	In the rest of this subsection, 
	we show that the right hand side of \eqref{var4} 
	has a finite lower bound under \eqref{condition1}, separating cases (i) and (ii).
	
	\medskip
	
	\subsubsection{Subcritical cases}
	We first consider the easier subcritical case 
	$$\frac{4}{s} < p < 2 + \frac{4s}{(d-1)s+2},$$ 
	where the Wick-ordered renormalized mass is of any finite size $K>0$, and the nonlinearity is of any magnitude $\al>0$.
	Note that the condition on $p$ coupled with $s<2$ implies 
	$$\frac{d-2+\sqrt{d^2+8}}{d+1}<s<2.$$

	\begin{proof}[Proof of \eqref{var1d}]
		By duality and Young's inequality, we have
		\begin{align}
			\begin{aligned}
				\Big| \int_{\R^d} & \wick{| Y_N + \Theta_N|^2} dx \Big|  \\
				& = \Big| \int_{\R^d} \wick{|Y_N|^2}  dx  + 2 \int_{\R^d} \Re( Y_N \overline{\Theta_N})  dx + \int_{\R^d} |\Theta_N|^2 dx \Big| \\
				& \ge -\Big| \int_{\R^d} \wick{|Y_N|^2}  dx \Big|  - 2 \|Y_N \|_{\mathcal H^{-\delta}(\R^d)} \| \Theta_N \|_{\mathcal H^\delta(\R^d)} + \int_{\R^d}  |\Theta_N|^2 dx\\
				& \geq  -\Big| \int_{\R^d} \wick{|Y_N|^2}  dx \Big|  - C_\eps \| Y_N \|_{\mathcal H^{-\delta}(\R^d)}^{p_1} -  \eps \| \Theta_N \|_{\mathcal H^\delta(\R^d)}^{q_1} + \int_{\R^d} |\Theta_N|^2  dx.
			\end{aligned}
			\label{cutoff0-d}
		\end{align}
		where $\frac{1}{s}-\frac{1}{2}<\delta<1$ and $p_1, q_1 > 1$ are such that
		\[
		\frac{1}{p_1} + \frac{1}{q_1} = 1.
		\]
		Furthermore, by interpolation, we have 
		\[
		\|\Theta_N\|_{\mathcal H^\delta(\R^d)}
		\leq  C  \| \Theta_N \|_{L^2(\R^d)}^{(1-\delta)p_2} +  C \| \Theta_N \|_{\mathcal H^1(\R^d)}^{\delta q_2},
		\]
		where $p_2, q_2 > 1$ are such that
		\[
		\frac{1}{p_2} + \frac{1}{q_2} = 1.
		\]
		We may then choose $ \delta =\frac{2-s}{2s}+$, $q_1 = 1+$ and $p_2 = \frac{4s}{3s-2}+$ so that
		\begin{align} \label{choiexpo-d}
			q_1(1-\delta)p_2 = 2.
		\end{align}	
		It follows that
		\begin{align}
			\label{Lqbound-d}
			\| \Theta_N \|_{\mathcal H^\delta(\R^d)}^{q_1} 
			\leq   C \| \Theta_N \|_{L^2(\R^d)}^{2} +  C \| \Theta_N \|_{\mathcal H^1(\R^d)}^{\delta q_1 q_2}.
		\end{align}
		Here the constant $C$ is independent of $N$ and may vary from line to line. By choosing $\eps C <\frac{1}{2}$, from \eqref{cutoff0-d} and \eqref{Lqbound-d}, we conclude that
		\begin{align}
			&\bigg\{  \Big| \int_{\R^d}  \wick{| Y_N + \Theta_N|^2} dx \Big|   \leq K \bigg\}  \notag\\
			&\quad \subset \bigg\{  \|\Theta_N\|_{L^2(\R^d)}^2 \leq   K + \Big| \int_{\R^d} \wick{ |Y_N|^2 }  dx   \Big| + C_\eps \| Y_N \|_{\mathcal H^{-\delta}(\R^d)}^{p_1}  + \frac{1}{2} \| \Theta_N \|_{L^2(\R^d)}^{2} + \eps C  \| \Theta_N \|_{\mathcal H^1(\R^d)}^{\delta q_1 q_2} \bigg\} \notag\\
			&\quad = \bigg\{  \|\Theta_N \|_{L^2(\R^d)}^2 \leq  2 K + 2 \Big| \int_{\R^d} \wick{ |Y_N|^2 }  dx   \Big| + 2 C_\eps \| Y_N \|_{\mathcal H^{-\delta}(\R^d)}^{p_1}   + 2 \eps C \| \Theta_N \|_{\mathcal H^1(\R^d)}^{\delta q_1 q_2} \bigg\} \notag \\
			&\quad  = : \O_K.
			\label{cutoff1-d}
		\end{align}
		
		\noi
		We then recall an elementary inequality, 
		which is a direct consequence of the mean value theorem and 
		the Young's inequality.
		Given $p >2$ and $\eps>0$, there exists $C_\eps$ such that
		\begin{align}
			\label{Young-d}
			|z_1 + z_2|^p \le (1+ \eps)|z_1|^p + C_\eps |z_2|^p  
		\end{align}
		
		\noi
		holds uniformly in $z_1,z_2 \in \mathbb C$.
		Here $C_\eps$, 
		which may differ from line to line,
		denotes a constant depending only on $\eps$.
		We conclude from \eqref{Rp}, \eqref{Young-d} with $\eps=1$, \eqref{cutoff1-d}, the sharp Gagliardo-Nirenberg-Sobolev inequality \eqref{GNS} and
		$$
		\|\nabla u\|^2_{L^2(\R^d)} \leq \|\nabla u\|^2_{L^2(\R^d)} + \|\jb{x}^{s/2} u\|^2_{L^2(\R^d)} = \|u\|^2_{\mathcal H^1(\R^d)}
		$$ 
		that
		\begin{align}
			\al R_p \big(Y_N + \Theta_N\big) 
			&\cdot  \ind_{\{ | \int_{\R^d} : |Y_N + \Theta_N|^2 : dx | \le K  \}} \notag\\
			&  \le 2\al R_p \big(\Theta_N\big) 
			\cdot  \ind_{\{ | \int_{\R^d} : |Y_N + \Theta_N|^2 : dx | \le K  \}} + C \al R_p(Y_N)  \label{var5-1-d}\\
			& \le 2 \al R_p \big( \Theta_N \big) 
			\cdot  \ind_{\O_K}  + C \al R_p(Y_N) \notag\\
			& \le \frac{2\al }p C_{\textup{GNS}} \| \Theta_N\|^{\frac{d(p-2)}{2}}_{\H^1 ({\R^d})}\|\Theta_N\|^{\frac{4-(d-2)(p-2)}{2}}_{L^2({\R^d})} \cdot \ind_{\O_K}   + C \al R_p(Y_N) \notag\\
			\intertext{where $C_{\textup{GNS}}$ is the implicit constant (depending on the dimension and $p$) and the set $\O_K$ is given in \eqref{cutoff1-d}. 
				Noting that $\frac{d(p-2)}{2} < 2$ when $p<  2+ \frac{4s}{(d-1)s+2} $, 
				we apply Young's inequality to continue with}
			& \le  C  \| \Theta_N  \|_{L^2({\R^d})}^{\frac{2(4-(d-2)(p-2))}{4-d(p-2)}} \cdot \ind_{\O_K}
			+ \frac14 \| \Theta_N  \|_{\mathcal H^1 ({\R^d})}^2 
			+ C R_p(Y_N),
			\label{var5-d}
		\end{align}
		
		\noi
		where the constant $C$ depends only on $d,p,\al$. Then from \eqref{cutoff1-d}, interpolation and Young's inequality, 
		we have
		\begin{align}
			&  \| \Theta_N \|_{L^2({\R^d})}^{\frac{2(4-(d-2)(p-2))}{4-d(p-2)}} \cdot \ind_{\O_K}  \notag \\
			& \les  K^{\frac{4-(d-2)(p-2)}{4-d(p-2)}} +  \Big| \int_{\R^d} \wick{|Y_N|^2}  dx   \Big|^{\frac{4-(d-2)(p-2)}{4-d(p-2)}}  +   
			C_\eps  \| Y_N \|_{\mathcal H^{-\dl} ({\R^d})}^{p_1 \frac{4-(d-2)(p-2)}{4-d(p-2)} }   + \Big( \eps    \| \Theta_N \|_{\mathcal H^1 ({\R^d})}^{\dl q_1 q_2}  \Big)^{\frac{4-(d-2)(p-2)}{4-d(p-2)}}   \notag \\
			& \le C_K +  C   \Big| \int_{\R^d} \wick{|Y_N|^2}  dx   \Big|^{\frac{4-(d-2)(p-2)}{4-d(p-2)}}   + C_\eps   \| Y_N \|_{\mathcal H^{-\dl} (\R^d)}^{p_1 \frac{4-(d-2)(p-2)}{4-d(p-2)}} 
			+   \eps C  \| \Theta_N \|_{\mathcal H^1 ({\R^d})}^{2}  ,
			\label{cutoff2-d}
		\end{align}
		
		\noi
		provided 
		\[
		\delta q_1 q_2 \frac{4-(d-2)(p-2)}{4-d(p-2)} <2.
		\]
		Since $\delta =\frac{2-s}{2s}+$, $q_1=1+$ and $q_2 = \frac{4s}{s+2}-$ (as $p_2=\frac{4s}{3s-2}+$, see \eqref{choiexpo-d}), the above condition requires 
		\begin{align} \label{cond-p}
			p<2+\frac{4s}{(d-1)s+2}.
		\end{align} 
		Here we remark that the constants $C,C_K$ in \eqref{cutoff2-d}, which may differ from line to line, are independent of $\eps$. By collecting \eqref{var4}, \eqref{I}, \eqref{var5-d} and \eqref{cutoff2-d}, 
		we arrive at
		\begin{align}
			\label{var70-d}
			\begin{split}
				- & \log  \E_{\mu}  \Big[ \exp \Big(\al R_p(u_N) \cdot  \ind_{\{ | \int_{\R^d} \wick{|u_N |^2} dx| \le K\}}  \Big)\Big]\\
				& \ge \inf_{\theta \in \mathbb{H}_a} \E \bigg[ - C \| \Theta_N \|_{L^2({\R^d})}^{\frac{2(4-(d-2)(p-2))}{4-d(p-2)}} \cdot \ind_{\O_K}
				- C  R_p(Y_N) \\
				& \hphantom{XXXXXXXX} 
				- \frac14 \| \Theta_N \|_{\mathcal H^1 ({\R^d})}^2 + \frac12 \int_0^1 
				\| \theta (t) \|_{L^2(\R^d)}^2 dt \bigg]\\
				& \ge \inf_{\theta \in \mathbb{H}_a}  \E \bigg[  - C  \| \Theta_N \|_{L^2({\R^d})}^{\frac{2(4-(d-2)(p-2))}{4-d(p-2)}} \cdot \ind_{\O_K} + \frac14  \int_0^1 
				\| \theta (t) \|_{L^2(\R^d)}^2 dt 
				- C R_p(Y_N) \bigg]\\
				& \ge \inf_{\theta \in \mathbb{H}_a}  \E \bigg[ - C C_K  - C   \Big| \int_{{\R^d}} \wick{|Y_N|^2}  dx   \Big|^{\frac{4-(d-2)(p-2)}{4-d(p-2)}} - C C_\eps    \| Y_N \|_{\mathcal H^{-\dl} ({\R^d})}^{p_1 \frac{4-(d-2)(p-2)}{4-d(p-2)}}  \\
				& \hphantom{XXXXXXXXX} 	+ \Big( \frac14 - C \eps \Big) \int_0^1 
				\| \theta (t) \|_{L^2(\R^d)}^2 dt 
				- C R_p(Y_N)  \bigg].
			\end{split}
		\end{align}
		
		\noi
		Then, by choosing $\eps >0$ sufficiently small such that  $C \eps  < \frac14$, we obtain 
		\begin{align}
			\begin{split}
				- & \log  \E_{\mu}  \Big[ \exp \Big(\al R_p(u_N) \cdot  \ind_{\{ | \int_{\R^d} \wick{|u_N |^2} dx| \le K\}}  \Big)\Big]\\
				& \ge  - C C_K   - C  \E\big[ R_p(Y_N) \big]  
				- C  \E \bigg[ \Big| \int_{{\R^d}} \wick{|Y_N|^2}  dx   \Big|^{\frac{4-(d-2)(p-2)}{4-d(p-2)}} \bigg]  \\
				& \hphantom{XXXXX} - C C_\eps  \E \bigg[   \| Y_N \|_{\mathcal H^{-\dl} ({\R^d})}^{p_1 \frac{4-(d-2)(p-2)}{4-d(p-2)}}  \bigg].
			\end{split}
			\label{var71-d}
		\end{align}
		
		\noi
		Since $\frac{4}{s}<p<2+\frac{4s}{(d-1)s+2}$, Corollaries \ref{COR:intp} and \ref{COR:WCE} give
		$$
		\sup_N \E\big[ R_p(Y_N) \big] + \sup_N \E \bigg[ \Big| \int_{{\R^d}} \wick{|Y_N|^2}  dx   \Big|^{\frac{4-(d-2)(p-2)}{4-d(p-2)}} \bigg] + \sup_N \E \bigg[   \| Y_N \|_{\mathcal H^{-\dl} ({\R^d})}^{p_1 \frac{4-(d-2)(p-2)}{4-d(p-2)}}  \bigg] <\infty.
		$$
		This proves \eqref{var3}, hence \eqref{var1d} in the subcritical case.
	\end{proof}

	\subsubsection{Critical case}
	\label{SUB:cri}
	~~~

	This section focuses on the critical case 
	$$p = 2+ \frac{4s}{(d-1)s+2}.$$ 
	We shall prove Theorem \ref{THM:main} (ii) - (a) as well as the case when $\al \ll 1$.
	The argument here is inspired by \cite{OOT,OOT1}.  
	We first show the uniform exponential integrability 
	\begin{align}
		\label{uniint_pc-d}
		\sup_{N \in \mathbb N} \Big\|  \ind_{ \{ | \int_{\R^d} \wick{ | u_N (x)|^2 } dx | \le K\}}  e^{\frac{\al} p{\| u_N \|_{L^p (\R^d)}^p}} \Big\|_{L^1 (\mu)}   < \infty ,
	\end{align}
	
	\noi 
	for all $0 < \al \ll 1$ uniformly in $K > 0$. Then, given $K >0$, it follows that the set
	\[
	A (K) = \bigg\{ \al > 0 : \sup_{N \in \mathbb N} \Big\|  \ind_{ \{ | \int_{\R^d} \wick{ | u_N (x)|^2 } dx | \le K\}}  e^{\frac{\al} p{\| u_N \|_{L^p (\R^d)}^p}} \Big\|_{L^1 (\mu)}   < \infty \bigg\},
	\]
	is non-empty.
	We define
	\begin{align}
		\label{threshold-d}
		\al_0 (K) = \sup A (K).
	\end{align}
	
	\noi 
	Then, given $K > 0$, it is easy to see that \eqref{uniint_pc-d} holds for all $0 < \al < \al_0 (K)$.
	We will establish the convergence of the truncated Gibbs measure in the next subsection.
	
	Proceeding as in \eqref{var5-d}, we have
	\begin{align}
		\begin{split}
			R_p \big(Y_N + \Theta_N\big) 
			&\cdot  \ind_{\{ | \int_{\R^d} : |Y_N + \Theta_N|^2 : dx | \le K  \}} \\
			&  \le 2 R_p \big(\Theta_N\big) 
			\cdot  \ind_{\{ | \int_{\R^d} \wick{ |Y_N + \Theta_N|^2 } dx | \le K  \}} + C R_p(Y_N).
		\end{split}
		\label{CN0-d}
	\end{align}
	
	\noi 
	By Gagliardo-Nirenberg-Sobolev inequality and Young's inequality, we have
	\begin{align}
		\begin{split}
			R_p(\Theta_N) & = \frac1 p \int_{\R^d} |\Theta_N|^p dx  \\
			& \les  \|\Theta_N \|_{L^2(\R^d)}^{\frac{4-(d-2)(p-2)}{2}}   \|\Theta_N \|_{\H^{1}(\R^d)}^{\frac{d(p-2)}2} \\
			& \les  \| \Theta_N\|_{L^2(\R^d)}^{\frac{2(4-(d-2)(p-2))}{4-d(p-2)}} +  \|\Theta_N \|_{\H^1(\R^d)}^2.
		\end{split}
		\label{CN1_1-d}
	\end{align}
	
	\noi 
	Also, we notice that
	\begin{align}
		\bigg| \int_{\R^d} \wick{|Y_N + \Theta_N|^2 } dx \bigg| \ge \bigg| \int_{\R^d} 2 \Re(Y_N \overline{\Theta_N}) + |\Theta_N|^2   dx \bigg|  - \bigg| \int_{\R^d} \wick{ |Y_N |^2 } dx \bigg| 
		\label{CN1-d}
	\end{align}
	
	Now, we consider two cases.
	
	{\bf Case 1.} We first assume 
	\begin{align} 
		\| \Theta_N\|_{L^2(\R^d)}^2 \gg \bigg| \int_{\R^d} \Re(Y_N \overline{\Theta_N}) dx \bigg|.
	\end{align}
	
	\noi 
	Then, together with \eqref{CN1-d} we have
	\begin{align}
		\begin{split}
			\| \Theta_N & \|_{L^2(\R^d)}^{\frac{2(4-(d-2)(p-2))}{4-d(p-2)}} \ind_{\{ | \int_{\R^d} \wick{ |Y_N + \Theta_N|^2 } dx | \le K  \}} \\
			& \sim \bigg| \int_{\R^d} 2 \Re(Y_N \overline{\Theta_N}) + |\Theta_N|^2 dx\bigg|^{\frac{4-(d-2)(p-2)}{4-d(p-2)}} \ind_{\{ | \int_{\R^d} \wick{ |Y_N + \Theta_N|^2 } dx | \le K  \}} \\
			& \les K^{\frac{4-(d-2)(p-2)}{4-d(p-2)}} + \bigg| \int_{\R^d} \wick{ |Y_N |^2 } dx \bigg|^{\frac{4-(d-2)(p-2)}{4-d(p-2)}},
		\end{split}
		\label{CN1_2-d}
	\end{align}
	
	\noi 
	which will be sufficient for our purpose.
	
	{\bf Case 2.} Let us assume that
	\begin{align}
		\| \Theta_N\|_{L^2(\R^d)}^2 \les \bigg| \int_{\R^d} \Re(Y_N \overline{\Theta_N}) dx \bigg|.
		\label{CN2-d}
	\end{align}
	
	\noi We write
	\begin{align*}
		Y_N = \sum_{n=0}^N Y_{N,n} e_n, \quad \Theta_N = \sum_{n=0}^N \Theta_{N,n} e_n.
	\end{align*}
	For each $n$, we decompose 
	\begin{align}
		\label{Tdecom-d}
		\Theta_{N,n} = a_n Y_{N,n} + w_n ,
	\end{align}
	
	\noi 
	where 
	\begin{align*}
		a_n : = 
		\left\{
		\begin{array}{cl}
			\frac{\Re(\Theta_{N,n} \overline{Y_{N,n}})}{|Y_{N,n}|^2} & \textup{ if } n \le N\\
			0 & \textup{ otherwise}
		\end{array}
		\right.
		\quad \text{ and } w_n = \Theta_{N,n} - a_nY_{N,n}.
	\end{align*}
	
	\noi 
	From the above definition, we have
	\begin{align}
		\|\Theta_N\|_{L^2(\R^d)}^2 & = \sum_{n=0}^N \big(a_n^2 |Y_{N,n}|^2 + |w_n|^2\big) , \label{CN3-d}\\
		\int_{\R^d} \Re(Y_N \overline{\Theta_N}) dx & = \sum_{n=0}^N a_n |Y_{N,n}|^2.\label{CN4-d}
	\end{align}
	
	\noi 
	By collecting \eqref{CN2-d}, \eqref{CN3-d} and \eqref{CN4-d}, we have
	\begin{align}
		\label{CN5-d}
		\sum_{n=0}^N a_n^2 |Y_{N,n}|^2 \les \bigg| \sum_{n=0}^N a_n |Y_{N,n}|^2 \bigg| .
	\end{align}
	
	\noi 
	We fix $n_0 < N$ to be chosen later.
	Then, by the orthogonal decomposition \eqref{Tdecom-d}, we see that $a_n^2 |Y_{N,n}|^2 \le |\Theta_{N,n}|^2$, and thus
	\begin{align}
		\begin{split}
			\bigg| \sum_{n=n_0+1}^N a_n |Y_{N,n}|^2 \bigg| & \le \bigg( \sum_{n=0}^N a_n^2 \ld_n^2 |Y_{N,n}|^2\bigg)^{\frac12} \bigg( \sum_{n=n_0+1}^N \ld_n^{-2} |Y_{N,n}|^2\bigg)^{\frac12} \\
			& \le \bigg( \sum_{n=0}^N  \ld_n^2 |\Theta_{N,n}|^2\bigg)^{\frac12} \bigg( \sum_{n=n_0+1}^N \ld_n^{-2} |Y_{N,n}|^2\bigg)^{\frac12} \\
			& = \|\Theta_N \|_{\mathcal H^1(\R^d)} \| \P^{\perp}_{n_0} Y_N \|_{\mathcal H^{-1}(\R^d)}.
		\end{split}
		\label{CN6-d}
	\end{align}
	
	\noi 
	For $n \le n_0$, from \eqref{CN5-d} and then \eqref{CN4-d},
	we have
	
	\begin{align}
		\begin{split}
			\bigg| \sum_{n=0}^{n_0} a_n |Y_{N,n}|^2 \bigg| 
			& \le \bigg( \sum_{n=0}^N a_n^2 |Y_{N,n}|^2\bigg)^{\frac12} \bigg( \sum_{n=0}^{n_0} |Y_{N,n}|^2\bigg)^{\frac12} \\
			& \le C \bigg| \sum_{n=0}^N  a_n |Y_{N,n}|^2\bigg|^{\frac12} \bigg( \sum_{n=0}^{n_0} |Y_{N,n}|^2\bigg)^{\frac12} \\
			& \le \frac12 \bigg| \sum_{n=0}^N  a_n |Y_{N,n}|^2\bigg| + C' \| \P_{n_0} Y \|_{L^2(\R^d)}^2\\
			& = \frac12 \bigg| \int_{\R^d} \Re(Y_N \overline{\Theta_N}) dx \bigg| + C' \| \P_{n_0} Y \|_{L^2(\R^d)}^2.
		\end{split}
		\label{CN7-d}
	\end{align}
	
	\noi 
	By collecting \eqref{CN4-d}, \eqref{CN6-d} and \eqref{CN7-d}, we arrive at
	\begin{align*}  
		\bigg| \int_{\R^d} \Re(Y_N \overline{\Theta_N}) dx \bigg|  & \les \| \Theta_N \|_{\mathcal H^1(\R^d)} \| \P^{\perp}_{n_0} Y_N \|_{\mathcal H^{-1}(\R^d)} +  \| \P_{ n_0} Y \|_{L^2(\R^d)}^2 , 
	\end{align*}
	
	\noi 
	which together with \eqref{CN2-d} yields
	\begin{align}
		\label{CN8-d}
		\| \Theta_N\|_{L^2(\R^d)}^2  & \les \| \Theta_N \|_{\mathcal H^1(\R^d)} \| \P^\perp_{n_0} Y_N \|_{\mathcal H^{-1}(\R^d)} +  \| \P_{n_0} Y \|_{L^2(\R^d)}^2 .
	\end{align}
	
	By denoting
	\[
	Z_{N,n_0} = \L^{-1/2} \P^\perp_{n_0} Y_N, \qquad \wt \s_{N,n_0} =  \E \big[\| Z_{N,n_0} \|_{L^2(\R^d)}^2 \big] ,
	\]
	
	\noi 
	we have
	\begin{align}
		\label{B0e-d}
		\begin{split}
			\|\P^\perp_{n_0} Y_N \|_{\mathcal H^{-1}(\R^d)}^2 
			& = \|\L^{-1/2} \P^{\perp}_{n_0} Y_N\|_{L^2(\R^d)}^2 \\
			& =   \int_{\R^d} |\L^{-1/2} \P^\perp_{n_0} Y_N|^2 - \E\big[|\L^{-1/2} \P^\perp_{n_0} Y_N|^2 \big] + \E\big[|\L^{-1/2} \P^\perp_{n_0} Y_N|^2 \big] dx  \\
			& = \int_{\R^d} |Z_{N,n_0}|^2- \E[|Z_{N,n_0}|^2] dx  + \wt \s_{N,n_0}.
		\end{split}
	\end{align}
	
	\noi 
	Then by using Corollary \ref{COR:CLR1} we obtain
	\begin{align} 
		\label{B0-d}
		\wt \s_{N,n_0}  = \sum_{n= n_0+1}^N \frac2{\ld_n^4} \les \ld_{n_0}^{-3 + \frac2s},
	\end{align}
	
	\noi 
	and
	\begin{align}
		\label{B01-d}
		\begin{split}
			\E \bigg[ \Big( \int_{\R^d} |Z_{N,n_0}|^2- \E[|Z_{N,n_0}|^2] dx \Big)^2 \bigg] 
			& = \sum_{n = n_0+1}^N \frac4{\ld_n^8} \les \ld_{n_0}^{-7 + \frac2s}.    
		\end{split}
	\end{align}
	
	\noi  
	Now we define a non-negative random variable $B_1(\o)$ by
	\begin{align}
		\label{B1-d}
		B_1(\o) = \ld_{n_0}^{\frac72-\frac1s} \bigg| \int_{\R^d} |Z_{N,n_0}|^2- \E[|Z_{N,n_0}|^2] dx\bigg|. 
	\end{align}
	
	\noi 
	By the Wiener chaos estimate (see e.g., \cite[Theorem I.22]{Simon15} or \cite[Proposition 2.4]{TT10}), we have
	\begin{align}
		\label{B1e-d}
		\E [B_1^q] \le C(q)\bigg( \ld_{n_0}^{7-\frac2s} \E\bigg[\bigg( \int_{\R^d} |Z_{N,n_0}|^2 -\E[|Z_{N,n_0}|^2] dx\bigg)^2\bigg]\bigg)^{\frac{q}2} <\infty
	\end{align}
	
	\noi 
	for any finite $q\ge 2$. The same bound holds for $1\leq q<2$ by H\"older's inequality. From \eqref{B0e-d}, \eqref{B0-d}, \eqref{B01-d}, \eqref{B1-d} and \eqref{B1e-d},
	we have 
	\begin{align}
		\label{B1e2-d}
		\| \P^{\perp}_{n_0} Y_N \|_{\mathcal H^{-1}(\R^d)}^2 \les \ld_{n_0}^{- \frac72 + \frac1s}  B_1(\o) + \ld_{n_0}^{-3+\frac2s}. 
	\end{align}

	Now we set 
	\begin{align}
		\label{B2-d}
		B_2(\o) =\bigg| \int_{\R^d} |Y_N|^2 - \E[|Y_N|^2] dx \bigg|.
	\end{align}
	
	\noi 
	Then it follows that
	\begin{align}
		\label{B2e-d}
		\begin{split}
			\|\P_{n_0} Y_N \|_{L^2(\R^d)}^2 &=\int_{\R^d} |\P_{n_0} Y_N|^2 - \E[|\P_{n_0} Y_N|^2] dx +  \E \big[\|\P_{n_0}Y_N \|_{L^2(\R^d)}^2\big] \\
			& \le B_2 (\o) + \sum_{n=0}^{n_0} \frac2{\ld_n^2} \\
			& \les B_2(\o) + \ld_{n_0}^{-1 + \frac2s},
		\end{split}
	\end{align}
	
	\noi 
	where we used Corollary \ref{COR:CLR}.
	Furthermore, from computations similar to those in the proof of Corollary \ref{COR:WCE}, and an application of Corollary~\ref{COR:CLR1}, we obtain
	\begin{align}
		\begin{split}
			\E [B_2^q] & \le C(q)\bigg( \E\bigg[\bigg(\int_{\R^d} |Y_N|^2 - \E[|Y_N|^2] dx\bigg)^2\bigg] \bigg)^{q/2} \\
			& \le C(q)\bigg( \sum_{n=0}^N \frac4{\ld_n^4} \bigg)^{q/2} \\
			& \le C(q) \big(\ld_N^{-3+\frac2s}\big)^{q/2} < \infty, 
		\end{split}
		\label{B2ea-d}
	\end{align}
	uniformly, since $-3+\frac2s <0$ and $\ld_N \ge \ld_0$.
	
	From \eqref{CN8-d}, \eqref{B1e2-d}, and \eqref{B2e-d}, we have
	\begin{align}
		\label{CN9-d}
		\begin{split}
			&\| \Theta_N\|_{L^2(\R^d)}^{\frac{2(4-(d-2)(p-2))}{4-d(p-2)}}  \\
			&\quad \les \Big(\| \Theta_N \|_{\mathcal H^1(\R^d)} \| \P^{\perp}_{n_0} Y_N \|_{\mathcal H^{-1}(\R^d)} +  \| \P_{n_0} Y \|_{L^2(\R^d)}^2 \Big)^{\frac{4-(d-2)(p-2)}{4-d(p-2)}} \\
			&\quad \les \Big( \ld_{n_0}^{-\frac72 + \frac1s}  B_1(\o) + \ld_{n_0}^{-3+\frac2s} \Big)^{\frac{4-(d-2)(p-2)}{4-d(p-2)}}  \| \Theta_N \|_{\mathcal H^1(\R^d)}^{\frac{4-(d-2)(p-2)}{4-d(p-2)}}   + \Big(B_2(\o)  + \ld_{n_0}^{\frac{2-s}s } \Big)^{\frac{4-(d-2)(p-2)}{4-d(p-2)}} \\
			&\quad \les \Big( \ld_{n_0}^{-\frac72 + \frac1s}  B_1(\o)   \| \Theta_N \|_{\mathcal H^1(\R^d)}\Big)^{\frac{4-(d-2)(p-2)}{4-d(p-2)}} + \Big( \ld_{n_0}^{-3+\frac2s}    \| \Theta_N \|_{\mathcal H^1(\R^d)}\Big)^{\frac{4-(d-2)(p-2)}{4-d(p-2)}}  \\
			&\quad \hphantom{XXXX} + B_2(\o)^{\frac{4-(d-2)(p-2)}{4-d(p-2)}}  + \ld_{n_0}^{\frac{2-s}s \cdot {\frac{4-(d-2)(p-2)}{4-d(p-2)}} },
		\end{split}
	\end{align}
	
	\noi 
	We next choose $n_0$ such that
	\begin{align}
		\label{sizeM-d}
		\ld_{n_0}^{\frac{2(2-s)}s \cdot {\frac{4-(d-2)(p-2)}{4-d(p-2)}} } \sim \| \Theta_N \|_{\mathcal H^1(\R^d)}^2.
	\end{align}

	\noi 
	Recall that in the critical case $p = 2+\frac{4s}{(d-1)s+2}$ and thus 
	\[
	\frac{4-(d-2)(p-2)}{4-d(p-2)} = \frac{2+s}{2-s},
	\]
	
	\noi 
	which together with \eqref{sizeM-d} gives
	\begin{align}
		\label{CN10-d}
		\Big( \ld_{n_0}^{-3+\frac2s}    \| \Theta_N \|_{\mathcal H^1(\R^d)}\Big)^{\frac{4-(d-2)(p-2)}{4-d(p-2)}} =  \| \Theta_N \|_{\mathcal H^1(\R^d)}^2.
	\end{align}
	
	\noi 
	Then it follows that
	\begin{align}
		\label{CN11-d}
		\begin{split}
			&\Big( \ld_{n_0}^{-\frac72 + \frac1s}  B_1(\o)   \| \Theta_N \|_{\mathcal H^1(\R^d)}\Big)^{\frac{4-(d-2)(p-2)}{4-d(p-2)}} \\ 
			&\hphantom{XXXX} = \ld_{n_0}^{\left(-\frac72 + \frac1s\right)\frac{4-(d-2)(p-2)}{4-d(p-2)}} B_1(\o)^{\frac{4-(d-2)(p-2)}{4-d(p-2)}} \|\Theta_N\|_{\mathcal H^1(\R^d)} \\
			&\hphantom{XXXX} \les \ld_{n_0}^{\left(-\frac72 + \frac1s\right)\frac{2(4-(d-2)(p-2))}{4-d(p-2)}} B_1(\o)^{\frac{2(4-(d-2)(p-2))}{4-d(p-2)}}  + \|\Theta_N\|_{\mathcal H^1(\R^d)}^2\\
			&\hphantom{XXXX} \les B_1(\o)^{\frac{2(4-(d-2)(p-2))}{4-d(p-2)}}  + \|\Theta_N\|_{\mathcal H^1(\R^d)}^2,
		\end{split}
	\end{align}
	
	\noi 
	where we used the fact that $-\frac72 + \frac1s < 0$.
	By collecting \eqref{CN9-d}, \eqref{CN10-d} and \eqref{CN11-d},
	it follows that 
	\begin{align}
		\label{CN12-d}
		\begin{split}
			\| \Theta_N\|_{L^2(\R^d)}^{\frac{2(4-(d-2)(p-2))}{4-d(p-2)}}  & \les \|\Theta_N\|_{\mathcal H^1(\R^d)}^2 + B_1(\o)^{\frac{2(4-(d-2)(p-2))}{4-d(p-2)}} + B_2(\o)^{\frac{4-(d-2)(p-2)}{4-d(p-2)}} + C
		\end{split}
	\end{align}
	for some constant $C>0$ independent of $N$.

	Finally,
	we are ready to prove \eqref{uniint_pc-d}.
	By collecting \eqref{var4}, \eqref{CN0-d}, 
	\eqref{B1e-d}
	\eqref{B2ea-d},
	\eqref{CN11-d} and \eqref{CN12-d},
	we arrive at
	\begin{align}
		\label{CN13-d}
		\begin{split}
			- & \log  \E_{\mu}  \bigg[ \exp \Big(\al R_p(u_N) \cdot  \ind_{\{ | \int_{\R^d} \wick{|u_N (x)|^2} dx| \le K\}} \Big) \bigg] \\
			&= \inf_{\theta \in \mathbb{H}_a} \E \Big[ - \al R_p \big( Y_N+ \Theta_N \big) 
			\cdot  \ind_{\{ | \int_{\R^d} \wick{| Y_N  + \Theta_N|^2} dx|  \le K \}} + \frac12 \int_0^1 
			\| \theta (t) \|_{L^2(\R^d)}^2 dt  \Big] \\
			& \ge \inf_{\theta \in \mathbb{H}_a} \E \bigg[ - C \al  \|\Theta_N\|_{L^2(\R^d)}^{\frac{2(4-(d-2)(p-2))}{4-d(p-2)}} \cdot \ind_{\{ | \int_{\R^d} \wick{|Y_N + \Theta_N|^2} dx | \le K  \}} \\ 
			& \hphantom{XXXX} 
			+ \bigg( \frac12 - C \al \bigg) \int_0^1 \|\theta (t) \|_{L^2(\R^d)}^2 dt   - C  \al R_p(Y_N) \bigg]\\
			& \ge \inf_{\theta \in \mathbb{H}_a}  \E \bigg[  - C \al  \|\Theta_N\|_{\mathcal H^1(\R^d)}^2 - C   B_1(\o)^{\frac{2(4-(d-2)(p-2))}{4-d(p-2)}} - C B_2(\o)^{\frac{4-(d-2)(p-2)}{4-d(p-2)}} - C_K \\ 
			& \hphantom{XXXX}
			- \bigg| \int_{\R^d} \wick{ |Y_N |^2 } dx \bigg|^{\frac{4-(d-2)(p-2)}{4-d(p-2)}} 
			+  \bigg( \frac12 - C \al \bigg)   \int_0^1 
			\| \theta (t) \|_{L^2(\R^d)}^2 dt 
			- C\al  R_p(Y_N) \bigg]\\
			& \ge \inf_{\theta \in \mathbb{H}_a}  \E \bigg[  - C  +    \Big( \frac12 - 2 C \al \Big) \int_0^1 
			\E \big[ \| \theta (t) \|_{L^2(\R^d)}^2 \big] dt 
			\bigg] > -C,
		\end{split}
	\end{align}
	
	\noi 
	provided $\al \ll 1$.
	Here the constant $C$ may vary from line to line.
	Thus, we finish the proof of the uniform exponential integrability \eqref{uniint_pc-d} uniformly in $K > 0$.
	
	\subsubsection{Normalizability}
	\label{SUB:nor}
	In this subsection, we complete the proof of Theorem \ref{THM:main} (i) and (ii)-(a), i.e. the convergence \eqref{cov-lp} provided \eqref{uniint_pc-d}.
	Before proceeding with the proof, we recall a useful lemma.
	
	A set of functions $\{f_n\}_{n\ge 1} \subset L^1(\mu)$ is called uniformly integrable if 
	\[
	\lim_{M\to \infty} \sup_n \int_{\{|f_n|>M\}} |f_n| d\mu = 0.
	\]
	
	\noi 
	Then we have the following lemma
	\begin{lemma}[\textbf{Vitali}]\label{LEM:vitali}\mbox{}\\
		The sequence $(f_n)$ converges in $L^1(\mu)$ if and only if $(f_n)$ converges in $\mu$-measure and $(|f_n|)$ is uniformly integrable.
	\end{lemma}
	
	Now we are ready to prove Theorem \ref{THM:main}.
	
	\begin{proof}[Proof of Theorem \ref{THM:main} (i) and (ii)-(a)]
		According to Lemma \ref{LEM:vitali}, we need to show that the density 
		\begin{align}
			\label{density}
			f_N = \ind_{ \{ | \int_{\R^d} \wick{ | u_N (x)|^2 } dx | \le K\}}  e^{\frac{\al} p{\| u_N \|_{L^p (\R^d)}^p}} 
		\end{align}
		
		\noi 
		is convergent in $\mu$-measure and is uniformly integrable.
		
		We first note that from Corollary \ref{COR:intp} we have 
		\begin{align*} 
			\lim_{N} e^{\frac{\al} p{\| u_N \|_{L^p (\R^d)}^p}}  = e^{\frac{\al} p{\| u \|_{L^p (\R^d)}^p}}  
		\end{align*}
		
		\noi 
		in $\mu$-measure.
		The $\mu$-measure convergence of $\ind_{ \{ | \int_{\R^d} \wick{ | u_N (x)|^2 } dx | \le K\}}$ follows from a similar argument as in \cite[Lemma 2.7]{RSTW22}.
		
		It remains to show that $f_N$ given in \eqref{density} is uniformly integrable. To see this, we distinguish two cases.
		For the subcritical case, i.e. $\frac4s < p < 2 + \frac{4s}{(d-1)s+2}$, we have
		\begin{align}
			\label{uni_sub}
			\sup_N \int_{|f_N| > M} |f_N| d\mu \le M^{-1}   \sup_N \int |f_N|^2 d\mu \les M^{-1},
		\end{align}
		
		\noi 
		where in the last step we used \eqref{uniint_p} with $r = 2$.
		This then implies that $(f_N)$ is uniformly integrable.
		Now we turn to the critical case, i.e. $p = 2+ \frac{4s}{(d-1)s+2}$. 
		Let $\al_0$ be as in Subsection \ref{SUB:cri} such that \eqref{uniint_pc-d} holds for all $0 < \al < \al_0$. Given $\al \in (0, \al_0)$, there exists $\eps > 0$ such that $\al (1+\eps) < \al_0$. In particular, we have
		\begin{align}
			\label{uniint_pc1}
			\begin{split}
				\sup_{N \in \mathbb N} &  \|  f_N \|_{L^{1+\eps} (\mu)}  =  \sup_{N \in \mathbb N} \Big\|  \ind_{ \{ | \int_{\R^d} \wick{ | u_N (x)|^2 } dx | \le K\}}  e^{\frac{\al (1+\eps)} p{\| u_N \|_{L^p (\R^d)}^p}} \Big\|_{L^{1} (\mu)}    < \infty .
			\end{split}
		\end{align}
		
		\noi 
		Then we have 
		\begin{align}
			\label{uni_cri}
			\sup_N \int_{|f_N| > M} |f_N| d\mu \le M^{-\eps}   \sup_N \int |f_N|^{1+\eps} d\mu \les M^{-\eps},
		\end{align}
		
		\noi 
		which implies that $(f_N)$ is uniformly integrable.
	\end{proof}

	\subsection{Non-normalizability}
	\label{SEC:non}
	
	In this section, we prove the second part of Theorem \ref{THM:main}, i.e. (ii)-(b) and (iii). We follow the strategy of \cite{RSTW22} but with some necessary modifications. The main estimate is as follows: 
	\begin{lemma}[\textbf{Divergence of the partition funcion}]\label{lem:div}\mbox{}\\
		Let $s<2$ and assume 
		\begin{align}
			\label{condition2}
			\begin{split}
				\textup{either (i) }& p \ge 2+\frac{4s}{(d-1)s+2} \textup{ and } \alpha \gg 1 \textup{ when } p=2+\frac{4s}{(d-1)s+2}; \\
				\textup{or   (ii) } & K > 0, \al > \al_0 (K) \textup{ when } p=2+\frac{4s}{(d-1)s+2},
			\end{split}
		\end{align}
		where $\al_0(K)$ is given in \eqref{threshold-d}. Then
		\begin{align}
			\limsup_{N\to \infty}  \E_{\mu}  \Big[ \exp (\al R_p(u_N)) \cdot  \ind_{\{ | \int_{\R^d} \wick{|u_N (x)|^2} dx| \le K\}}\Big] =  \infty.
			\label{pax}
		\end{align}
	\end{lemma}
	
	First, we notice that 
	\begin{align}
		\begin{split}
			\E_\mu &\Big[  \exp\big( \al R_p(u_N)\big) 
			\cdot\ind_{\{ |\int_{{\R^d}}  \wick{ |u_N|^2 } dx | \le K\}} \Big]\\
			&\ge \E_\mu\Big[\exp\Big(  \al R_p(u_N)
			\cdot \ind_{\{ |\int_{{\R^d}} \wick{ |u_N|^2 } dx | \le K\}}\Big)   \Big]
			- 1.
		\end{split}
		\label{pax1}
	\end{align}
	
	\noi
	Therefore, the  divergence \eqref{pax} follows once we prove
	\begin{align}
		\limsup_{N\to \infty}  \E_\mu\Big[\exp\Big( \al R_p(u_N)
		\cdot \ind_{\{ |\int_{{\R^d}} \wick{ |u_N|^2 } dx | \le K\}}\Big)   \Big] =  \infty.
		\label{pa0}
	\end{align}
	
	\noi
	By  the  Bou\'e-Dupuis variational formula Lemma \ref{LEM:var}, we have
	\begin{align}
		\begin{split}
			- \log & {\E_\mu\Big[\exp\Big( \al  R_p(u_N)
				\cdot \ind_{\{ |\int_{{\R^d}} \wick{ |u_N|^2 } dx | \le K\}}\Big)   \Big]} \\
			&= \inf_{\dr \in \mathbb H_a} \E\bigg[ - \al R_p (Y_N + \Theta_N) \cdot \ind_{\{ |\int_{{\R^d}} \wick{ | Y_N|^2} 
				+ 2\Re (Y_N  \cj{\Theta_N} ) + |\Theta_N|^2 dx | \le K\}} \\
			&\hphantom{XXXXX}
			+ \frac 12 \int_0^1 \| \dr(t)\|_{L^2(\R^d)} ^2 dt \bigg],
			\label{DPf}
		\end{split}
	\end{align}
	
	\noi
	where $Y_N$ and $\Theta_N$ are as in \eqref{nots}.
	Here,  $\E_\mu$ and $\E$ denote expectations
	with respect to the Gaussian field~$\mu$
	and the underlying probability measure $\PP$, respectively. We prove that~\eqref{DPf} diverges to $-\infty$ for large $N$ by exhibiting a suitable trial state for the variational problem.

	\subsubsection{Construction of the drift}
	We follow the plan outlined in Section~\ref{sec:BouDup1}. We first construct a profile which stays bounded in $L^2(\R^d)$ but grows in $L^p(\R^d)$ with $p > 2$.
	Fix a large parameter $M \gg 1$.
	Let $f: {\R^d} \to \R$ be a real-valued radial Schwartz function
	with $\|f\|_{L^2 (\R^d)} = 1$
	such that 
	the Fourier transform $\ft f$ is a smooth function
	supported  on $\big\{\frac 12 <  |\xi| \le 1 \big\}$.
	Define a function $f_{M}$  on ${\R^d}$ by 
	\begin{align}
		f_{M}(x) =  M^{-\frac{1}2} \int_{\R^d} e^{2\pi ix\cdot \xi} \ft f(\tfrac{\xi}M) d\xi = M^{\frac{1}2} f(Mx), 
		\label{fMdef} 
	\end{align}

	\noi
	where $\ft f$ denotes the Fourier transform on ${\R^d}$ defined by 
	\[ \ft f(\xi) =  \int_{{\R^d}} f(x) e^{-2\pi  i\xi \cdot x} dx.\]
	
	\noi
	Then, a direct computation yields  the following lemma (see e.g., \cite[Lemma 4.1]{RSTW22}).
	
	\begin{lemma}[\textbf{Blow-up profile}]\mbox{} \label{LEM:soliton}\\
		Let $s \in \R$. Then, we have
		\begin{align}
			\int_{{\R^d}} f_{M}^2 dx &= 1 , \label{fM0} \\
			\int_{{\R^d}} (\L^{\frac{s}2} f_{M})^2 dx &\les M^{2s}, \label{fm2} \\
			\int_{{\R^d}}  |f_{M}|^p  dx &\sim M^{\frac{dp}2-d} \label{fM1}
		\end{align}
		
		\noi
		for any $p>0$ and $M \gg 1$.
	\end{lemma}

	Next, for some $1 \ll M \leq N$ we construct an approximation $Z_M (t)$ to $Y_N (t)$ as in Definition~\ref{def:gaussapprox}. We also set
	\begin{align} 
		\al_{M, N}= \frac {\E \bigg[ 2  \int_{\R^d} \Re(Y_N \cj{Z_M}) dx-\int_{\R^d}|Z_M|^2  dx \bigg]}{\int_{\R^d} |\P_{N}f_{\ld_M}|^2 dx}.
		\label{fmb1}
	\end{align}
	for $N\ge M \gg 1$. Our trial drift is now:
	
	\begin{definition}[\textbf{Trial drift in the subharmonic case}]\label{def:trial 1}\mbox{}\\
		For $M \gg 1$, we set $f_{\ld_M}$, $Z_M$, and $ \al_{M, N}$ as above and define a drift $\dr= \dr^0$ by 
		\begin{align}
			\begin{split}
				\dr^0 (t) 
				& = \L^{\frac12} \bigg( -\frac{d}{dt} Z_M(t) + \sqrt{ \al_{M, N}} f_{\ld_M} \bigg)
			\end{split}
			\label{drift}
		\end{align}
		and
		\begin{align}
			\Dr^0 = I(\theta^0)(1) 
			= \int_0^1 \L^{-\frac12} \dr^0(t) \, dt = - Z_M + \sqrt{ \al_{M, N}} f_{\ld_M}.
			\label{paa0}
		\end{align} 
	\end{definition}
	
	We remark that $\sqrt{\al_{M,N}} 
	(\P_{N} f_{\ld_M})$ acts as a blow-up profile in our analysis, 
	and $\theta^0 \in \mathbb H_a$ is the stochastic drift such that $Y_N + \Theta^0_N$ approximates $  \sqrt{\al_{M,N}} 
	(\P_{N} f_{\ld_M})$,
	which drives the potential energy $R_p(Y_N + \Theta^0_N)$ to blow up.
	Remarkably,  due to the construction, 
	the Wick-ordered $L^2$ norm of this approximation $Y_N + \Theta^0_N$ can be made as small as possible, i.e. the cutoff in the Wick-ordered $L^2$ norm does not exclude the blow-up profiles.

	Before inserting this choice in the variational formula~\eqref{DPf} to obtain an upper bound, we vindicate that Definition~\ref{def:gaussapprox} indeed efficiently approximates the Brownian motion $Y_N (t)$ in the case at hand:
	
	\begin{lemma}[\textbf{Approximating the Brownian motion, subharmonic case}] 
		\label{LEM:approx}\mbox{}\\
		Let $d\ge 1$, $1<s<2$ and restrict to radial functions when $d\ge 2$. Given $ N \ge M\gg 1$,  we define $Z_M$ by its coefficients in the eigenfunction expansion of $\L$
		as in Definition~\ref{def:gaussapprox}.
		
		If $c$ is chosen large enough in~\eqref{ZZZ}, we have the following estimates:
		\begin{align}
			&\E \big[ \|Z_M \|_{L^2({\R^d})}^2 \big]  \sim  \ld_M^{\frac2s-1},
			\label{NRZ0}\\
			&\E\bigg[  2 \Re \int_{{\R^d}} Y_N \cj{Z_M} dx - \int_{{\R^d}} |Z_M|^2 dx   \bigg] \sim \ld_M^{\frac2s-1}, \label{NRZ1}\\
			&\E \bigg[  \Big|  \wick{\| Y_N-Z_M\|_{L^2({\R^d})}^2}  \Big|^2      \bigg] \les \ld_M^{- 3+ \frac{2}s},    \label{NRZ3}\\
			&\E\bigg[\Big| \int_{{\R^d}} Y_N  f_{\ld_M} dx \Big|^2\bigg] 
			+ \E\bigg[\Big| \int_{{\R^d}} Z_M  f_{\ld_M} dx \Big|^2\bigg] \les \ld_M^{- 2},   \label{NRZ5}\\
			&\E\bigg[\int_0^1 \Big\| \frac d {d\tau} Z_M(\tau) \Big\|^2_{\H^1 (\R^d)}d\tau\bigg] \les   \ld_M^{\frac{2}s} \label{NRZ6}
		\end{align}
		
		\noi
		for any $N \ge M \gg 1$, 
		where  $Z_M =Z_M (1)$
		and
		\begin{align}
			\wick{ \| Y_N-Z_M\|_{L^2({\R^d})}^2} \stackrel{\textup{def}}{=} 
			\| Y_N-Z_M\|_{L^2({\R^d})}^2 - \E\big[ \| Y_N-Z_M\|_{L^2({\R^d})}^2 \big].
			\label{ZZZ2}
		\end{align}
	\end{lemma}
	
	\begin{proof}
		Let 
		\begin{align}
			X_n(t)=\wt Y_N(n, t)- \wt Z_{M}(n, t), 
			\quad 0\le n \le M.
			\label{ZZ1} 
		\end{align}
		Then,  from  \eqref{ZZZ}, 
		we see that $X_n(t)$ satisfies 
		the following stochastic differential equation:
		\begin{align*}
			\begin{cases}
				dX_n(t)=- c \ld_n^{-1} \ld_M  X_n(t) dt + \ld_n^{-1} dB_n(t)\\
				X_n(0)=0
			\end{cases}
		\end{align*}	
		
		\noi
		for $0\le n \le M$, where $c \gg 1$ is a constant.
		By solving this stochastic differential equation, we have
		\begin{align}
			X_n(t)= \ld_n^{-1} \int_0^t e^{- c  \ld_n^{-1} \ld_M  (t-\tau)}dB_n(\tau).
			\label{ZZ2}
		\end{align}
		
		\noi
		Then, from \eqref{ZZ1} and \eqref{ZZ2}, we have 
		\begin{align}
			\wt Z_{M}(n, t)= \wt Y_N(n, t)- \ld_n^{-1} \int_0^t  e^{-   c \ld_n^{-1} \ld_M  (t-\tau)} dB_n(\tau)
			\label{SDE1}
		\end{align}
		
		\noi
		for $n \le M$. Hence, from \eqref{SDE1}, the independence of $\{B_n \}_{n \in \N}$, 
		Ito's isometry (see \cite[Section 4.2]{Evans12}) and Corollary \ref{COR:CLR} with $p = 1$ and $p = \frac12$, we have
		\begin{align}
			\begin{split}
				\E \big[ \|Z_M\|_{L^2 ({\R^d})}^2 \big]&=\sum_{n \le M} \bigg( \E \big[  | \wt Y_N(n) |^2  \big]
				- 2 \ld_n^{- 2} \int_0^1 e^{-  c \ld_n^{-1} \ld_M    (1-\tau)}d\tau \\
				& \hphantom{XXXXX}
				+ \ld_n^{- 2} \int_0^1 e^{-2 c \ld_n^{-1} \ld_M   (1-\tau)}d\tau \bigg)\\
				& \sim \sum_{n=0}^M\ld_n^{-2} + O\Big( c^{-1}\sum_{n \le M } \ld_n^{-1} \ld_M^{-1}  \Big)\\
				& \sim \ld_M^{\frac2s-1} + O(c^{-1} \ld_M^{\frac{2}s-1}) \sim \ld_M^{\frac2s-1},
			\end{split}
			\label{ZZ3}
		\end{align}
		
		\noi
		for any $M\gg 1$, $c \gg 1$ and $s \in (1,2)$.
		This proves \eqref{NRZ0}.
		
		By the $L^2$ orthogonality of $\{e_n\}_{n\in\N}$,  \eqref{SDE1}, \eqref{NRZ0}, and  
		proceeding as in \eqref{ZZ3}, we have
		\begin{align*}
			\E\bigg[  2 \Re \int_{{\R^d}} Y_N & \cj{Z_M} dx - \int_{{\R^d}} |Z_M|^2 dx   \bigg]
			=\E \bigg[ 2 \Re \sum_{n \le M }\wt Y_N(n)  \cj{ \wt Z_M(n)} -\sum_{n \le M }|  \wt Z_M(n) |^2    \bigg]\\
			&=\E \bigg[ \sum_{n \le M } | \wt Z_M(n)  |^2+ \sum_{n \le M } \Re \bigg( 2 \ld_n^{-1}  \int_0^1 e^{- c \ld_n^{-1} \ld_M    (1-\tau) }dB_n(\tau) \bigg)  \cj{\wt Z_M(n)}   \bigg]\\
			&\sim \ld_M^{\frac2s-1} +O\Big( c^{-1} \sum_{n \le M } \ld_n^{-1} \ld_M^{-1}  \Big)\\
			&\sim \ld_M^{\frac2s-1} 
		\end{align*} 
		
		\noi 
		for any $N\ge M\gg 1$ and $c \gg 1$.
		Here we used Corollary \ref{COR:CLR} with $p=\frac12$ and $s\in (1,2)$.
		This proves  \eqref{NRZ1}.

		Note that  $\wt Y_N(n)-\wt Z_M(n)$ is a mean-zero Gaussian random variable.
		Then, from \eqref{SDE1}
		and Ito's isometry, 
		we have 
		\begin{align}
			\begin{split}
				\E \bigg[ \Big( & |\wt Y_N(n)-  \wt Z_M(n)|^2-\E\big[ |\wt Y_N(n)-\wt Z_M(n)|^2 \big] \Big)^2  \bigg]\\
				& = 7
				\Big(\E\big[ |\wt Y_N(n)-\wt Z_M(n)|^2 \big] \Big)^2  \\
				& =7 \ld_n^{-4} \bigg(\int_0^1 e^{-2 c \ld_n^{-1} \ld_M    (1-\tau)  } d\tau \bigg)^2 \\
				&\sim \ld^{-2}_n \ld_M^{-2} ,
			\end{split}
			\label{ZZ4}
		\end{align}
		
		\noi
		for $0 \le n \le M$,
		where in the second step we used {that $\E [|X|^4] = 8 \s^4$ for complex random variable $X \sim \mathcal N (0,2\s^2)$}.
		Hence, 
		from Plancherel's theorem, \eqref{ZZZ2}, 
		the independence of $\{B_n \}_{n \in \N}$, 
		the independence
		of 
		$\big\{ |\wt Y_N(n) |^2- \E \big[ | \wt Y_N(n)|^2 \big]\big\}_{M < n \le N}$
		and 
		\[\big\{ |\wt Y_N(n)-\wt Z_M(n)|^2-\E\big[ |\wt Y_N(n)-\wt Z_M(n)|^2\big]\big\}_{n \le M},\]
		
		\noi 
		and \eqref{ZZ4}, we have
		\begin{align*}
			\E \Big[  & \big|  \wick{\| Y_N-Z_M\|_{L^2({\R^d})}^2}  \big|^2      \Big] \notag \\
			&= \sum_{M< n \le N } \E \bigg[ \Big( |\wt Y_N(n) |^2- \E \big[ | \wt Y_N(n)|^2 \big] \Big)^2 \bigg]\notag \\
			&\hphantom{XX}+ \sum_{n \le M} \E \bigg[ \Big( |\wt Y_N(n)-\wt Z_M(n)|^2-\E\big[ |\wt Y_N(n)-\wt Z_M(n)|^2 \big] \Big)^2  \bigg]\notag  \\
			&\les \sum_{M< n \le N} {\ld_n^{-4}}
			+\sum_{ n \le M }
			\ld_n^{-2} \ld_M^{-2} 
			\les \ld_M^{- 3+ \frac{2}s},
		\end{align*}
		
		\noi
		where we used Corollary \ref{COR:CLR} with $p =2$ or $p=1$ and $s\in (1,2)$.
		This proves \eqref{NRZ3}.

		From \eqref{fm2} and the definition of $Y_N$, we have
		\begin{align}
			\begin{split}
				\E\bigg[ \Big| \int_{{\R^d}} Y_N  f_{\ld_M} dx  \Big|^2\bigg]
				&= \E \bigg[ \Big| \sum_{n \le N}  \wt Y_N(n) { \jb{ f_{\ld_M}, e_n}} \Big|^2 \bigg]
				= \sum_{n \le N} \ld_n^{-2} |\jb{ f_{\ld_M}, e_n}|^2 \\
				&\le \int_{{\R^d}} \big|\L^{-\frac12} f_{\ld_M} (x)\big|^2 dx
				\les \ld_M^{-2},
			\end{split}
			\label{app4}
		\end{align}
		
		\noi
		where in the last step we used Lemma \ref{LEM:soliton}.
		From \eqref{ZZ2}, Ito's isometry, and \eqref{fm2}, we have 
		\begin{align}
			\begin{split}
				\E \bigg[ \Big| \sum_{n \le M} X_n(1) \cj{ \jb{ f_{\ld_M}, e_n} } \Big|^2 \bigg]
				&=\E \Bigg[ \bigg|\sum_{n \le M} \bigg(\ld_n^{-1}  \int_0^1 e^{- c \ld_n^{-1} \ld_M (1-\tau) } dB_n(\tau) \bigg)  {\jb{ f_{\ld_M}, e_n}}     \bigg|^2     \Bigg]\\
				&\les \ld_M^{-1} \sum_{n \le M}  \ld_n^{-1} | \jb{ f_{\ld_M}, e_n}|^2\\
				&\les \ld_M^{-1} \|\L^{-\frac14} f_{\ld_M}\|_{L^2(\R^d)}^2\\
				& \les \ld_M^{-2} .
			\end{split}
			\label{app5}
		\end{align}

		\noi
		Hence, \eqref{NRZ5} follows
		from \eqref{app4} and \eqref{app5}
		with \eqref{SDE1}.
		
		Lastly, from \eqref{ZZZ}, \eqref{ZZ1}, \eqref{ZZ2}, and Ito's isometry, we have 
		\begin{align*}
			\E\bigg[\int_0^1 \Big\| \frac d {d\tau} Z_M(\tau) \Big\|^2_{\mathcal H^1(\R^d)}d\tau\bigg] &= \ld_M \E\bigg[\int_0^1 \big\| \P_{M}(Y_N(\tau)) - Z_M(\tau) \big\|^2_{L^2(\R^d)}d\tau\bigg] \\
			&=  \ld_M \E\bigg[ \int_0^1  \sum_{n \le M} |X_n(\tau)|^2   d\tau\bigg]\\
			&=  \ld_M \sum_{n \le M} \int_0^1\E\big[|X_n(\tau)|^2\big] d\tau \\
			&= \ld_M \sum_{n \le M} \ld_n^{-2} \int_0^1 \int_0^\tau e^{-2 c \ld_n^{-1} \ld_M  (\tau-\tau')} d \tau' d\tau \\
			& \les \ld_M \sum_{n \le M } \ld_n^{-1} \ld_M^{-1} \\
			&\les  \ld_M^{\frac{2}s }, 
		\end{align*}

		\noi
		yielding  \eqref{NRZ6}.
		This completes the proof of Lemma~\ref{LEM:approx}.
	\end{proof}	
	
	As a consequence of Lemma \ref{LEM:approx}, we have a control on the second term of~\eqref{DPf}.
	\begin{lemma}[\textbf{Entropy of a the test drift}]
		\label{LEM:bddrift}\mbox{}\\
		Let $\theta^0$ be as in \eqref{drift},
		then we have
		\[
		\int_0^1 \E \big[\| \theta^0 (t)\|_{L^2 ({\R^d})}^2 \big] dt \les \ld_M^{\frac2s + 1},
		\]
		
		\noi
		uniformly in $N \ge N_0(M) \gg M \ge 1$.
	\end{lemma}
	
	\begin{proof}
		From \eqref{NRZ6} and \eqref{drift}, it suffices to show that
		\[
		\al_{M, N} \|  f_{\ld_M} \|_{\H^1({\R^d})}^2 \les   \ld_M^{\frac2s+1}.
		\]
		However, from \eqref{fM0} and \eqref{NRZ1}, we have
		\begin{align}
			\al_{M, N} \sim \ld_M^{\frac2{s} - 1}
			\label{logM}
		\end{align}
		provided $N \gg M$ and $N$ is sufficiently large. The conclusion follows from \eqref{fm2} and \eqref{logM} provided $N \gg M$.
	\end{proof}
	
	In what follows, we abuse notation by denoting
	\begin{align}
		\Dr^0_N = \P_{N} \Dr^0
		= -Z_M +  \sqrt{\al_{M,N}} 
		(\P_{N} f_{\ld_M})
		\label{YY0a}
	\end{align}
	for $N \ge M \ge 1$. Now we vindicate our claim about the Wick-ordered $L^2$ mass of our trial state. 
	
	\begin{lemma}[\textbf{Mass of the test drift}]
		\label{LEM:key}\mbox{}\\
		For any $K>0$, there exists $M_0=M_0(K) \geq 1$ such that 
		\begin{align}
			\PP\bigg( \Big| \int_{\R^d} \wick{|Y_N (x)|^2} dx + \int_{\R^d} 2 \Re(Y_N \cj{\Dr^0_N}) + |\Dr^0_N|^2  dx \Big| \le K \bigg) 
			\ge \frac 12, 
			\label{pa5}
		\end{align}
		
		\noi
		uniformly in $N \ge N_0(M) \gg M \ge M_0$.
	\end{lemma}
	
	\begin{proof} 
		First, from \eqref{fM0}, we note that the condition $N\ge N_0(M) \gg M$ guarantees that 
		\begin{align} 
			\label{N_0(M)} 
			\int_{{\R^d}} |\P_{N}f_{\ld_M}|^2 dx \ges 1,
		\end{align}
		which further implies that $\al_{M,N} \sim \ld_M^{\frac2s-1}$.
		From \eqref{paa0}, we have 
		\begin{align}
			\begin{split}
				&\E \bigg[ \Big| \int_{\R^d} \wick{|Y_N (x)|^2} dx  + \int_{\R^d} 2 \Re(Y_N \cj{\Dr^0_N}) + |\Dr^0_N|^2  dx \Big|^2 \bigg]\\
				&= \E \bigg[ \Big|\int_{\R^d} \wick{|Y_N (x)|^2} dx  - 2 \int_{{\R^d}} \Re(Y_N \cj{Z_M}) dx+\int_{{\R^d}} |Z_M|^2 dx\\
				&\hphantom{XX}
				+ \al_{M, N} \int_{{\R^d}} |\P_{N}f_{\ld_M}|^2 dx  + 2 \sqrt{ \al_{M, N}}  \int_{{\R^d}} \Re( (Y_N-Z_M)f_{\ld_M}) dx \Big|^2 \bigg].
			\end{split}
			\label{Pr1}
		\end{align}
		
		\noi
		From \eqref{logM} and \eqref{NRZ5} in Lemma \ref{LEM:approx}, we have
		\begin{align}
			\E \bigg[ \Big| \sqrt{ \al_{M, N}}  \int_{{\R^d}} (Y_N-Z_M)f_{\ld_M} dx   \Big|^2   \bigg] \les   \ld_M^{\frac{2}s-3}.
			\label{Pr4}
		\end{align}	
		
		\noi
		On the other hand, 
		from \eqref{fmb1} and \eqref{ZZZ2}, we have	
		\begin{align}
			\begin{split}
				\int_{\R^d} & \wick{|Y_N (x)|^2} dx - 2  \int_{{\R^d}} \Re(Y_N \cj{Z_M}) dx+\int_{{\R^d}} |Z_M|^2 dx+ \al_{M, N} \int_{{\R^d}} |\P_{N}f_{\ld_M}|^2 dx \\
				& =\int_{{\R^d}} |Y_N-Z_M|^2 -\E \big[|Y_N-Z_M |^2 \big] dx \\
				& =\wick{\| Y_N-Z_M\|_{L^2({\R^d})}^2  }.
			\end{split}
			\label{Pr2} 
		\end{align}	
		
		\noi
		Hence, from \eqref{Pr1}, \eqref{Pr4}, and \eqref{Pr2} with \eqref{NRZ3}
		in Lemma \ref{LEM:approx}, we obtain
		\begin{align*}
			\E \bigg[ \Big|   \int_{\R^d} \wick{|Y_N (x)|^2} dx + \int_{\R^d} 2 \Re(Y_N \cj{\Dr^0_N}) + |\Dr^0_N|^2  dx \Big|^2 \bigg]\les \ld_M^{-  3 +\frac2s}.
		\end{align*}
		
		\noi
		Therefore, by Chebyshev's inequality, 
		given any $K > 0$, there exists $M_0 = M_0(K) \geq 1$ such that 
		\begin{align}
			\PP\bigg( \Big|  \int_{\R^d} \wick{|Y_N (x)|^2} dx + \int_{\R^d} 2 \Re(Y_N \cj{\Dr^0_N}) + |\Dr^0_N|^2  dx \Big| > K \bigg)
			&\le C\frac{\ld_M^{-3+\frac2s}}{K^2}
			< \frac 12
			\label{M_0(K)}
		\end{align}

		\noi
		for any $M \ge M_0 (K)$ and $s\in (1,2)$.
		This proves  \eqref{pa5}.
	\end{proof}

	\subsubsection{Divergence of the partition function}
	We are now ready to prove:
	\begin{lemma}[\textbf{The test drift leads to divergences}]\label{lem:test subharm}\mbox{}\\
		Given $K > 0$, let $M$ and $N$ as in Lemma \ref{LEM:key}. Recall the choice $\Theta^0 = -Z_M + \sqrt{\al_{M,N}} f_{\ld_M}$ from Definition~\ref{def:trial 1}. Then 
		\[
		\E\bigg[ - \al R_p (Y_N + \Theta^0) \cdot \ind_{\{ |\int_{{\R^d}} \wick{ | Y_N|^2} 
			+ 2\Re (Y_N  \cj{\Theta_N} ) + |\Theta_N|^2 dx | \le K\}} + \frac 12 \int_0^1 \| \theta ^0(t)\|_{L^2(\R^d)} ^2 dt \bigg] \underset{ M \to \infty}{\longrightarrow} -\infty.
		\]
	\end{lemma}
	
	\begin{remark}
		\rm 
		Here $M \le N$ are chosen such that Lemma \ref{LEM:bddrift} and Lemma \ref{LEM:key} hold. In particular, given $K > 0$, there exists $M_0(K) \gg 1$ such that \eqref{M_0(K)} holds for all $M \ge M_0(K)$. With $M \ge M_0(K)$ chosen, there exists $N_0(M) \ge M$ such that \eqref{N_0(M)} (or equivalently \eqref{logM}) holds for all $N \ge N_0(M)$.
	\end{remark}

	\begin{proof} Using the mean value theorem and Young's inequality, we have for any $\eps > 0$,
		\begin{align}
			\begin{split} 
				\big| & R_p(Y_N + \Dr^0) - R_p( \sqrt{\alpha_{M,N}} f_{\ld_M})\big|\\
				&\le C \int_{\R^d} |Y_N - Z_M| \big(|Y_N-Z_M| + | \sqrt{\alpha_{M,N}} f_{\ld_M}|\big)^{p-1}dx\\
				& \le \eps R_p( \sqrt{\alpha_{M,N}} f_{\ld_M}) +C_\eps R_p(Y_N-Z_M).
			\end{split}
			\label{paa}
		\end{align}
		Moreover, 
		we have
		\begin{align}
			\begin{split}
				\E\big[R_p(Y_N -Z_M) \big] & = \frac 1p \int_{\R^d} \E \bigg[ \Big|\sum_{M < n\le N}\frac{B_n(1) e_n}{\lambda_n}+\sum_{n\le M}X_n(1)e_n\Big|^p \bigg] dx\\
				&\les  \int_{\R^d} \bigg( \E \bigg[ \Big|\sum_{M < n\le N}\frac{B_n(1) e_n}{\lambda_n}+\sum_{n\le M}X_n(1)e_n\Big|^2 \bigg]\bigg)^{\frac{p}{2}}dx\\
				&\les  \int_{\R^d}\Big(\sum_{M < n\le N}\frac{e_n^2 (x)}{\lambda_n^2}+\sum_{n\le M}\frac{e_n^2 (x)}{\lambda_n \ld_M } \Big)^{\frac{p}{2}}dx\\ 
				&\les  \int_{\R^d}\Big(\sum_{ n\le N}\frac{e_n^2 (x)}{\lambda_n^2 } \Big)^{\frac{p}{2}}dx \\ 
				&\les  1,
			\end{split}
			\label{paa2}
		\end{align}
		
		\noi 
		provided $p > \frac4s$, uniformly in $M$ and $N$.
		Here we use the fact that $X_n (1)$ is a Gaussian random variable with variance
		$\sim (\ld_n \ld_M )^{-1}$.
		
		We are now ready to put everything together.
		It follows from 
		\eqref{DPf},  \eqref{paa}, \eqref{paa2}, 
		and \eqref{YY0a}
		that there exists $C>0$ such that 
		\begin{align*}
			-\log &  \E_\mu \Big[\exp\Big({  \al R_p (u_N)} \cdot \ind_{\{|\int_{{\R^d}} : |u_N|^2 : dx|  \le K\}}\Big)   \Big]\notag \\
			&\le \E\bigg[ - \al R_p (Y_N + \Dr^0) \cdot 
			\ind_{\{ |  \int_{{\R^d}} \wick{|Y_N|^2} dx + \int_{\R^d} 2 \Re(Y_N \cj{\Dr^0_N}) + |\Dr^0_N|^2  dx  | \le K\}} + \frac 12 \int_0^1 \| \dr^0(t)\|_{L^2(\R^d)} ^2 dt \bigg] \notag \\
			&\le  \E\bigg[ - \frac12 R_p (    \sqrt{\alpha_{M,N}}f_{\ld_M}) \cdot \ind_{\{ |  \int_{{\R^d}} : |Y_N|^2 : dx  + \int_{\R^d} 2 \Re(Y_N \cj{\Dr^0_N}) + |\Dr^0_N|^2  dx | \le K\}} \notag 
			\\
			& \hphantom{XXXX} + C  R_p (Y_N - Z_M)
			+ \frac 12 \int_0^1 \| \dr^0(t)\|_{L^2(\R^d)} ^2 dt \bigg] \notag \\
			&\le    - \frac12 R_p (    \sqrt{\alpha_{M,N}} f_{\ld_M}) \cdot \mathbb \PP\bigg( \Big| \int_{\R^d} \wick{|Y_N|^2} dx + \int_{\R^d} 2 \Re(Y_N \cj{\Dr^0_N}) + |\Dr^0_N|^2  dx \Big| \le K \bigg)  \notag 
			\\
			& \hphantom{XXXX} + C \E \big[ R_p (Y_N - Z_M) \big]
			+ \frac 12 \int_0^1 \E \big[ \| \dr^0(t)\|_{L^2(\R^d)} ^2 \big] dt \notag \\
			\intertext{then from Lemma \ref{LEM:key}, \eqref{paa2}, and Lemma \ref{LEM:bddrift}, we may continue with} 
			& \le -  C_1^p \al \ld_M^{\frac{dp}2-d}  (\al_{M,N})^\frac{p}2+ C_2   ^2 \ld_M^2 \al_{M,N}  + C_3
		\end{align*}
		for some constants $C_1,C_2,C_3>0$, provided  $N \ge N_0(M) \gg M \ge M_0(K)$.
		Therefore, when  $p > 2+ \frac{4s}{(d-1)s+2}$, 
		it follows that
		\begin{align*}
			\begin{split}
				\liminf_{N \to \infty} \E_\mu\Big[& \exp\Big({ R_p (u_N)} \cdot \ind_{\{ | \int_{{\R^d}} : |u_N|^2 : dx | \le K\}}\Big)   \Big]\\
				\ge &~  \exp\Big(  C_1^p \al   \ld_M^{\frac{dp}2-d} (\ld_M^{\frac2s-1})^\frac{p}2 - C_2^2  \ld_M^2 \ld_M^{\frac2s-1}   -C_3 \Big), 
			\end{split}
		\end{align*}
		
		\noi
		which diverges to infinity
		as $M \to \infty$ provided $p > 2+\frac{4s}{(d-1)s+2}$ and $K, \al > 0$.
		
		It remains to consider the critical case when $p = 2+\frac{4s}{(d-1)s+2}$. 
		From the above computation, we see that 
		\begin{align} 
			\begin{split}
				\sup_{N \in \N} & \E_\mu \Big[\exp\big({  \al R_p (u_N)} \big) \cdot \ind_{\{|\int_{{\R^d}} : |u_N|^2 : dx|  \le K\}}   \Big] \\
				& = \sup_{N \in \mathbb N} \Big\|  \ind_{ \{ | \int_{\R^d} : | u_N (x)|^2 : dx | \le K\}}  e^{\frac{\al} p{\| u_N \|_{L^p (\R^d)}^p}} \Big\|_{L^1 (\mu)}  \\
				& = \infty 
			\end{split}
			\label{div_cri}
		\end{align} 
		
		\noi 
		provided $\al \gg 1$ such that $C_1^p \al > C_2^2 $.
		In particular, this shows when $ p = 2+\frac{4s}{(d-1)s+2}$ the number $\al_0 (K)$ defined in \eqref{threshold-d} is bounded for given $K >0$.
		From the definition \eqref{threshold-d}, given $K > 0$,
		we see that \eqref{div_cri} holds for all $\al > \al_0 (K)$.
		Thus, we finish the proof of \eqref{pax} for $p \ge 2+\frac{4s}{(d-1)s+2}$, and $\alpha>\alpha_0 (K)$ for any given $K >0$ when $p=2+\frac{4s}{(d-1)s+2}$.
	\end{proof}
	
	\subsubsection{Intermediate cases}
	This subsection considers the proof of Theorem \ref{THM:main} (ii) - (b).
	From the previous subsection, we have \eqref{pax}. 
	Recall the decomposition \eqref{mu} 
	$$ d \mu(u) = d \mu_N(u_N) \otimes d \mu_N^\perp(u_N^\perp), $$
	where $u_N = \P_{N} u$ and $u_N^\perp = u - \P_{N} u$. 
	Moreover, by \eqref{law}, we have that 
	$$\Law(Y_N(1), Y(1) - Y_N(1)) = \mu_N \otimes \mu_N^\perp.$$
	Define the set 
	\begin{equation}
		\O_K^\perp = \Big\{ u_N^\perp: \Big| \int_{\R^d} \wick{|u_N^\perp|^2}dx\Big| \le \frac K2, \frac 1p \int_{\R^d} |u_N^\perp|^p dx \le 1 \Big\}, \label{muperpdef}
	\end{equation}
	where we defined
	\[ 
	\int_{\R^d} \wick{|u_N^\perp|^2}dx = \sum_{n=N+1}^\infty \frac{|g_n|^2 - 2}{\ld_n^2},
	\]
	
	\noi
	which is chosen so that 
	$$ \int_{\R^d} \wick{|u_N^\perp|^2}dx = \int_{\R^d} \wick{|u|^2}dx - \int_{\R^d} \wick{|u_N|^2}dx. $$
	The proof of Theorem \ref{THM:main} (ii) - (b) requires a delicate analysis of the cut-off function when it is slightly perturbed. To overcome the challenges that arise, we introduce a crucial lemma that allows us to preserve the cut-off size $K$ in the approximation process. 
	\begin{lemma} \label{LEM:sameK}
		Consider the sets 
		\begin{gather*} 
			\O_{+}^\perp := \bigg\{ u_N^\perp \in \O_K^\perp : \int_{\R^d} \wick{|u_N^\perp|^2}dx \ge 0 \bigg\}, \\ 
			\Omega_{-}^\perp := \bigg\{ u_N^\perp \in \O_K^\perp : \int_{\R^d} \wick{|u_N^\perp|^2}dx \le 0 \bigg\} .
		\end{gather*}
		Then there exists $\eps_0 > 0$ such that 
		\begin{align} 
			\label{sameK_1}
			\min(\mu_N^\perp(\O_+^\perp), \mu_N^\perp(\O_-^\perp)) \ge \eps_0 
		\end{align}
		for every $N \gg 1$ large enough.
	\end{lemma}
	
	Before we prove Lemma \ref{LEM:sameK}, we prepare a technical lemma.
	\begin{lemma}
		\label{LEM:tech}
		Let $Y \ge 0$ be a random variable such that $0 <  c (\E[Y^2])^{\frac12} \le \E[Y]  < \infty$ for some positive constant $c$. 
		Then, we have
		$$ \mathbb P(Y > 0) \ge \frac{c^2}{4}. $$
	\end{lemma}
	
	\begin{proof} 
		We first note that up to multiplying $Y$ by a constant, we can assume that $\E[Y^2] = 1$. For $M > 0$, we have 
		$$\E[Y\ind_{Y > M}] \le \E\bigg[\frac{Y^2}{M}\ind_{Y > M} \bigg] \le \frac{1}{M}. $$
		Therefore, by choosing $M= \frac{2}{c}$, we obtain that 
		$$ \frac c 2 \le \E [Y] - \E [Y \ind_{Y > M}] \le \E[Y\ind_{Y \le M}] \le M \mathbb P(Y>0).  $$
		Therefore, 
		$$ \mathbb P(Y>0) \ge \frac{1}{M} \cdot \frac {c}{2} = \frac{c^2}{4}. $$
		
		\noi 
		We thus finish the proof.
	\end{proof} 
	
	We are ready to prove Lemma \ref{LEM:sameK}.
	
	\begin{proof}[Proof of Lemma \ref{LEM:sameK}]
		Define the random variables 
		\begin{gather*}
			Y = \int_{\R^d} \wick{|u_N^\perp|^2}dx \, \textup{ and }\,
			Y_\pm = \Big(\int_{\R^d} \wick{|u_N^\perp|^2}dx\Big)_{\pm}, 
		\end{gather*}
		where $u_N^\perp$ is distributed according to $\mu_N^\perp$, $a_+ = \max (0,a)$, and $a_- = |\min (0,a)|$.
		By Corollary \ref{COR:intp} and Corollary \ref{COR:WCE}, we have that $\mu_N^\perp (\O_K^\perp) \to 1$ as $N\to \infty$. As 
		\[
		\O^\perp_\pm = \O_K \cap \{\o: Y_+ > 0\},
		\]
		so for \eqref{sameK_1}, it is enough to show that 
		\begin{align} 
			\label{same_K2} \min \big(\mathbb P(Y_+ > 0), \mathbb P(Y_- > 0) \big) > 2 \eps_0, \end{align}
		for some $\eps_0$ independent of $N$. 
		From Corollary \ref{COR:WCE} and the fact
		$$\E\Big[\int_{\R^d} \wick{|u_N^\perp|^2}dx\Big] = 0, $$
		we have that 
		\begin{align} 
			\label{Ypm}
			\E[Y_+] = \E[Y_-] = \frac12 \E \big[|Y| \big] < \infty.
		\end{align}
		Moreover, by a direct computation (as in the proof of Corollary \ref{COR:WCE} and \eqref{wickY}), we have
		\begin{align}
			\begin{split}
				\E [Y^4] & = \E \bigg[ \Big( \sum_{n=N+1}^\infty \frac{|B_n|^2 - 2}{\ld_n^2} \Big)^4 \bigg] \\
				& =  -2 \bigg(\sum_{n=N+1}^\infty \frac{\E \big[ (|B_n|^2 - 2)^4 \big] }{\ld_n^8}\bigg)  + 3  \bigg( \sum_{n =N+1}^\infty \frac{\E \big[(|B_n|^2 - 2)^2\big]  }{\ld_n^4}  \bigg)^2 \\
				& \le C \bigg( \sum_{n =N+1}^\infty \frac{\E \big[(|B_n|^2 - 2)^2\big]  }{\ld_n^4}  \bigg)^2  = C \big( \E [ Y^2]\big)^2
			\end{split}
			\label{Y4}
		\end{align}
		
		\noi 
		for some universal constant $C > 0$.
		Therefore, by H\"older inequality and \eqref{Y4}, we have
		$$ \E[Y^2] \le \E[|Y|]^\frac23 \E[Y^4]^\frac13 \le C^\frac13 \E[|Y|]^\frac23 \E[Y^2]^\frac23. $$
		The above also reads
		$$ \E[|Y|] \ge C^{-\frac12} \E[Y^2]^\frac12, $$
		which together with \eqref{Ypm} implies
		\begin{align}
			\label{Ypm2}
			\E [Y_\pm]  \ge \frac12 C^{-\frac12} \E[Y^2]^\frac12 \ge \frac12 C^{-\frac12} \E[Y^2_\pm]^\frac12.
		\end{align}
		
		\noi 
		Then \eqref{sameK_1} follows from \eqref{Ypm2} and Lemma \ref{LEM:tech} by taking $\eps_0 = \frac{1}{32C}$.
	\end{proof}

	From the elementary inequality 
	\[ 
	|a + b|^p \ge (1-\eps) |a|^p - C_\eps |b|^p 
	\]
	
	\noi 
	for some constant $C_\eps > 0$, we obtain that 
	\begin{equation}\label{elementary}
		\exp\Big(\frac \al p \int_{\R^d} |u|^p dx \Big) \ge \exp\Big(-\frac{C_\eps}p \int_{\R^d} |u_N^\perp|^p dx \Big) \exp\Big(\frac {\al-\eps}{p} \int_{\R^d} |u_N|^p dx \Big).
	\end{equation}
	Therefore, by \eqref{muperpdef}, \eqref{elementary}, Lemma \ref{LEM:sameK}, and \eqref{pax} with $r = 1$, we obtain 
	\begin{align*}
		\mathcal Z_K & = \int \exp\Big(\frac \al p \int_{\R^d} |u|^p dx \Big) \ind_{ \{ | \int_{\R^d} \wick{| u (x)|^2} dx | \le K\}}   d \mu(u) \\
		&\ge \int \exp\Big(-\frac {C_\eps}p \int_{\R^d} |u_N^\perp|^p dx \Big)\exp\Big(\frac {\al-\eps}{p} \int_{\R^d} |u_N|^p dx \Big) \\
		&\phantom{\int\exp\Big()} \times \ind_{ \{ | \int_{\R^d} \wick{| u_N (x)|^2} dx  + \int_{\R^d} \wick{| u_N^\perp (x)|^2} dx | \le K\}}\ind_{\O_{\mathrm{sgn}(\int_{\R^d} \wick{| u_N (x)|^2} dx)}^\perp}(u_N^\perp) d \mu_N(u_N)d \mu_N^\perp(u_N^\perp)\\
		&\ge \int e^{-C_\eps} \exp\Big(\frac {\al -\eps}{p} \int_{\R^d} |u_N|^p dx \Big) \ind_{ \{ | \int_{\R^d} \wick{| u_N (x)|^2} dx| \le  K\}}\ind_{\O_{\mathrm{sgn}(\int_{\R^d} \wick{| u_N (x)|^2} dx)}^\perp}(u_N^\perp) d \mu_N(u_N)d \mu_N^\perp(u_N^\perp) \\
		&\ge e^{-C_\eps}\eps_0 \int \exp\Big(\frac {\al -\eps}{p} \int_{\R^d} |u_N|^p dx \Big) \ind_{ \{ | \int_{\R^d} \wick{| u_N (x)|^2} dx| \le K \}}d \mu_N(u_N)\\
		& \ge e^{-C_\eps} \eps_0  \int \exp\Big(\frac {\al -\eps}{p} \int_{\R^d} |u_N|^p dx \Big) \ind_{ \{ | \int_{\R^d} \wick{| u_N (x)|^2} dx| \le K\}}d \mu_N(u_N),
	\end{align*}
	provided $N\gg 1$, which together with \eqref{threshold-d} (or \eqref{pax}) implies that 
	\[
	\mathcal Z_K \ge \limsup_{N\to \infty} 
	e^{-C_\eps} \eps_0 \int \exp\Big(\frac {\al-\eps}{p} \int_{\R^d} |u_N|^p dx \Big) \ind_{ \{ | \int_{\R^d} \wick{| u_N (x)|^2} dx| \le  K\}}d \mu_N(u_N) = \infty
	\]
	
	\noi 
	provided $\eps \ll 1$ such that $\al - \eps > \al_0 (K)$ for given $K > 0$.
	This concludes the proof of \eqref{non_int2}, and hence Theorem \ref{THM:main} (ii) - (b).

	\section{Superharmonic potential}
	\label{SEC:superharmonic}
	
	In this section,
	we see how to extend the previous results for cases of $s \in (1,2)$ to cases of $s > 2$.
	The argument of this section is inspired by \cite{LW22}.
	One of the key advantages of the superharmonic case is that $\| u\|_{L^2(\R^d)} < \infty$ almost surely with respect to $\mu$, i.e.
	\begin{align}
		\label{sup_L2}
		\begin{split}
			\E_{\mu} [\|u\|_{L^2(\R^d)}^2] 
			& = \E [\|Y(1)\|_{L^2(\R^d)}^2]  = \sum_{n=0}^\infty \frac2{\ld_n^2} < \infty,      
		\end{split}
	\end{align}
	
	\noi 
	where $Y(1)$ is defined in \eqref{Yt},  in view of Lemma \ref{LEM:main2}. Furthermore, for $q \ge p $ we have
	\begin{align*}
		\begin{split}
			\big(\E_{\mu} [\|u\|_{L^p(\R^d)}^q] \big)^{\frac1q}
			& \leq \big\|  \| Y(1) \|_{L^q (\Omega)} \big\|_{L^{p} ({\R^d})} \\
			& \leq C(q) \big\|  \| Y(1) \|_{L^2 (\Omega)} \big\|_{L^{p} ({\R^d})} \\
			& = C(q)\bigg\|  \Big( \sum_{n\ge 0} \frac{e_n^2}{\ld_n^2} \Big)^\frac12 \bigg\|_{L^{p} ({\R^d})} \\
			& = C(q)\| \L^{-1} (x,x)\|_{L^{\frac{p}2} (\R^d)}^{\frac12} < \infty,     
		\end{split}
	\end{align*}
	
	\noi 
	provided $p>2$ and $p<\frac{2d}{d-2}$ when $d\ge 3$, by using Lemma \ref{LEM:main3}. The H\"older inequality then implies 
	\begin{align} \label{sup_Lp}
		\E_{\mu} [\|u\|_{L^p(\R^d)}^q] = \E[\|Y(1)\|_{L^p(\R^d)}^q]<\infty
	\end{align}
	provided $1\leq q <\infty$, $p>2$ and $p<\frac{2d}{d-2}$ when $d\ge 3$. We also need the following consequence of Fernique's theorem \cite{Fernique75}.
	See also Theorem~2.7 in \cite{DZ14}
	and Lemma 3.3 in \cite{RSTW22}.

	\begin{lemma}[\textbf{Fernique-type bounds}]\label{LEM:Fer}\mbox{}\\
		There exists a  constant $c>0$ such that if $X$ is a mean-zero Gaussian process  with values 
		in a separable Banach space $B$ with $\E\big[\|X\|_{B}\big]<\infty$, then
		\begin{align*}
			\int e^{ c \frac{\|X\|_B^2}{\left(\E[\|X\|_B]\right)^2}}\,d \PP <\infty.
		\end{align*}
		
		\noi
		In particular, we have
		\begin{equation*}
			\PP\big(\|X\|_B \ge t \big)\les \exp \bigg[- \frac{c t^2}{ \big( \E\big[\|X\|_B\big] \big)^2}  \bigg]
		\end{equation*}

		\noi
		for any $t>1$.
	\end{lemma}

	\subsection{Normalizability}
	\label{SEC:nor1}
	In this subsection, 
	we provide the proof of 
	the integrability part of Theorem \ref{THM:main1}.
	Namely,
	we prove:
	\begin{lemma}[\textbf{Integrability for superharmonic potentials}]\mbox{}\\
		Let $s>2$. Assume either one of the following conditions:
		\begin{align}
			\begin{split}
				&\textup{(i) subcritical nonlinearity: } 2< p < 2+\frac4d \text{ and } K > 0;\\
				&\textup{(ii) critical nonlinearity: } p = 2+\frac4d \text{ and } K < \| Q\|^2_{L^2(\R^d)}.
			\end{split}
			\label{cond_int}
		\end{align}
		Then
		\begin{align}
			\label{var1}
			\mathcal Z_{K} = \E_{\mu}  \Big[ \exp (\al R_p(u)) \cdot  \ind_{\{\|u\|^2_{L^2(\R^d)}\le K\}}\Big] < \infty, 
		\end{align}
		where $R_p(u)$ is given in \eqref{Rp}.
	\end{lemma}

	\begin{proof}
		
		\noindent\textbf{Step 1. Preliminaries.} Observing that
		\begin{align*}
			\E_{\mu}  \Big[ \exp (\al R_p(u)) \cdot  \ind_{\{\|u\|^2_{L^2(\R^d)}\le K\}}\Big]
			\le \E_{\mu}  \bigg[ \exp \Big(\al R_p(u) \cdot  \ind_{\{\|u\|^2_{L^2(\R^d)}\le K\}} \Big)\bigg],
		\end{align*}
		
		\noi
		the bound \eqref{var1} follows once we have
		\begin{align}\label{var3_1}
			\E_{\mu}  \bigg[ \exp \Big(\al R_p(u) \cdot  \ind_{\{\|u\|^2_{L^2(\R^d)}\le K\}} \Big)\bigg]  < \infty.
		\end{align}
		
		\medskip 
		One can observe that the equation \eqref{var3_1} does not require any frequency truncation $\P_N$ unlike the subharmonic case \eqref{var1d}. The main reason is that the $L^2$ mass does not involve the Wick renormalization. On the other hand, the equation \eqref{var1d} needs the frequency truncation $\P_N$ to deal with the Wick power $:|u_N|^2:$ defined by \eqref{Wick} and \eqref{Wick bis}.
		
		Using $\Law(Y(1))= \mu$, we apply the Bou\'e-Dupuis variational formula\footnote{See \cite[Theorem 3.2]{Zhang09} for a version of the non-singular case, where the frequency cut-off is not needed. Also, see \cite[Proposition A.1]{TW23} or \cite[Lemma 3.1]{LW22} for similar results.}, Lemma \ref{LEM:var} to 
		$$
		F(Y(1)) = - \al R_p(Y(1)) \cdot  \ind_{\{\|Y(1)\|^2_{L^2(\R^d)}\le K\}}
		$$ 
		and get
		\begin{align}
			\label{var4_1}
			\begin{split}
				&- \log \E_\mu \bigg[ \exp \Big(\al R_p(u) \cdot  \ind_{\{\|u\|^2_{L^2(\R^d)}\le K\}} \Big)\bigg] \\
				&\hphantom{X} = - \log \E \bigg[ e^{-F(Y(1))}\bigg]\\
				&\hphantom{X} = \inf_{\theta \in \mathbb{H}_a} \E \bigg[ -\al R_p\big(Y(1) + I(\theta)(1)\big)\cdot \ind_{\{\|Y(1)+I(\theta)(1)\|^2_{L^2(\R^d)}\le K\}} \\
				& \hphantom{XXXXXXX} + \frac12 \int_0^1 
				\| \theta (t) \|_{L^2(\R^d)}^2 dt  \Big],
			\end{split}
		\end{align}
		
		\noi
		where $Y(1)$ is given in \eqref{Yt}.
		Here, $\E_{\mu}$ and $\E$ denote expectations with respect to the Gaussian field $\mu$ and the underlying probability measure $\PP$ respectively.
		In the following, 
		we show that the right hand side of \eqref{var4_1} 
		has a finite lower bound.
		
		\medskip
		
		\noi
		{\bf Step 2. Subcritical case.}
		In the case
		$$2 < p <  2+\frac4d,$$
		we prove \eqref{var3_1} 
		with a mass cut-off of any finite size $K$.
		Then, by using \eqref{Young-d} with $\eps =1$ and the sharp Gagliardo-Nirenberg-Sobolev inequality, we obtain
		\begin{align}
			\label{var5_1}
			\begin{split}
				& \al R_p \big(Y(1) +I(\theta)(1)\big) 
				\cdot  \ind_{\{\|Y(1) +I(\theta)(1)\|^2_{L^2(\R^d)}\le K\}} \\
				& \le 2 \al R_p \big( I(\theta)(1)\big) 
				\cdot  \ind_{\{\|I(\theta)(1)\|_{L^2(\R^d)}\le \sqrt{K}+\| Y(1) \|_{L^2 (\R^d)} \}}   + C \al R_p(Y(1)) \\
				& \le \frac{2\al}p C_{\textup{GNS}} (\sqrt{K}+\| Y(1) \|_{L^2 (\R^d)})^{\frac{4-(d-2)(p-2)}{2}}  \| I(\theta)(1)\|^{\frac{d(p-2)}{2}}_{\mathcal H^1 (\R^d)}  + C \al R_p(Y(1)). \\
				& \le  C  + C \| Y(1)\|_{L^2 (\R^d)}^{\frac{2(4-(d-2)(p-2))}{4-d(p-2)}} 
				+ \frac14 \| I(\theta)(1)\|_{\mathcal H^1 (\R^d)}^2 
				+ C R_p(Y(1)),
			\end{split}
		\end{align}
		
		\noi
		where $C_{\textup{GNS}}$ is the sharp Gagliardo-Nirenberg-Sobolev constant and the constant $C$ depends only on $d,p,\al$. 
		By collecting \eqref{var4_1}, \eqref{var5_1} and Lemma \ref{LEM:bounds},
		we arrive at
		\begin{align*}
			\begin{split}
				- & \log \E_\mu \bigg[ \exp \Big(\al R_p(u) \cdot  \ind_{\{\|u\|^2_{L^2(\R^d)}\le K\}} \Big)\bigg] \\
				& \ge \inf_{\theta \in \mathbb{H}_a} \E \bigg[ - C - C \| Y(1)\|_{L^2 (\R^d)}^{\frac{2(4-(d-2)(p-2))}{4-d(p-2)}}  
				- C R_p(Y(1)) \\
				& \hphantom{XXXXXXX} 
				- \frac14 \| I(\theta)(1)\|_{\mathcal H^1 (\R^d)}^2 + \frac12 \int_0^1 
				\| \theta (t) \|_{L^2(\R^d)}^2 dt \bigg]\\
				& \ge \inf_{\theta \in \mathbb{H}_a} \E \bigg[ - C - C \| Y(1)\|_{L^2 (\R^d)}^{\frac{2(4-(d-2)(p-2))}{4-d(p-2)}}   
				- C \|Y(1)\|_{L^p(\R^d)}^p \\
				& \hphantom{XXXXXXX} 
				+ \frac14  \int_0^1 
				\| \theta (t) \|_{L^2(\R^d)}^2 dt \bigg]\\
				& \ge  \E \Big[  - C - C \|Y(1)\|_{L^p(\R^d)}^p  - C \| Y(1)\|_{L^2 (\R^d)}^{\frac{2(4-(d-2)(p-2))}{4-d(p-2)}}
				\Big] \\
				& > -\infty,
			\end{split}
		\end{align*}
		
		\noi
		where we used \eqref{sup_L2} and \eqref{sup_Lp} in the second to last step, i.e.
		\[
		\E \Big[    \|Y(1)\|_{L^p(\R^d)}^p  + \| Y(1)\|_{L^2 (\R^d)}^{\frac{2(4-(d-2)(p-2))}{4-d(p-2)}} \Big] < \infty.
		\]
		
		\noi 
		Here $C$ is a constant that may vary from line to line. 
		Thus we finish the proof of \eqref{var3_1} in the subcritical case.
		
		\medskip
		
		\noi
		{\bf Step 3. Critical case.} 
		Let now 
		$$
		p = 2+\frac4d.
		$$
		We shall prove \eqref{var3_1} under the assumption $\al^{\frac{d}2}K < \|Q\|^2_{L^2 (\R^d)}$.
		
		Since $s > 2$, from \eqref{sup_L2},
		it follows that
		\[
		\lim_{N \to \infty} \| \P^\perp_{N} Y(1) \|_{L^2 (\R^d)} = 0,
		\]
		
		\noi
		almost surely.  
		Therefore, given small $\eps >0$, 
		for $\o \in \O$ almost sure, 
		there exists a unique $N_\eps : = N_\eps(\o)$ such that  $N_\eps = 1$ for $\o \in \{\o: \| \P^\perp_{1} Y(1) \|_{L^2 (\R^d)} \le \eps\}$; otherwise 
		\begin{align}
			\label{Neps}
			N_\eps = \inf \big\{N \textup{ is dyadic}: N \ge 2 \textup{ such that }\| \P^\perp_{\frac{N}2} Y(1) \|_{L^2 (\R^d)} > \eps  \textup{ and }  
			\| \P^\perp_{N} Y(1) \|_{L^2 (\R^d)} \le \eps \big\}.
		\end{align}
		
		\noi
		Similar argument as before combined with \eqref{Young-d} and \eqref{Neps} yield that
		\begin{align}
			\al &R_p \big(Y(1) +I(\theta)(1)\big) 
			\cdot  \ind_{\{\|Y(1) +I(\theta)(1)\|^2_{L^2(\R)}\le K\}} \label{var5a} \\
			& \le \al (1+ \eps) R_p \big( \P_{ N_\eps} Y(1) + I(\theta)(1)\big) 
			\cdot  \ind_{\{\|\P_{N_\eps} Y(1) + I(\theta)(1)\|_{L^2(\R)}\le \sqrt{K}+\eps \}}   + C_\eps \al R_p(P_{N_\eps}^\perp Y(1)) \nonumber\\
			& \le \al \frac{1+\eps}p C_{\textup{GNS}} (\sqrt{K}+\eps)^{p-2} \left( \| \P_{N_\eps} Y(1)\|_{\H^1 (\R^d)}  + \| I(\theta)(1)\|_{\H^1 (\R^d)} \right)^2 + C_\eps 
			\al R_p(P_{N_\eps}^\perp Y(1)) \nonumber\\
			& \le \al \frac{(1+\eps)^2}p C_{\textup{GNS}}  (\sqrt{K}+\eps)^{p-2}  \| I(\theta)(1)\|^{2}_{\H^1 (\R^d)} + C_\eps \| \P_{N_\eps} Y(1)\|^{2}_{\H^1 (\R^d)} + C_\eps \al R_p(P_{N_\eps}^\perp Y(1)). \nonumber
		\end{align}
		
		\noi
		Since $p =2+\frac4d$, $C_{\textup{GNS}} = \frac{p}2 \|Q\|_{L^2(\R^d)}^{2-p}$ and $\al^{\frac{d}2} K<\|Q\|^2_{L^2(\R^d)}$, 
		there exist $\eta , \eps >0$ such that
		\begin{align}
			\label{eta}
			\al \frac{(1+\eps)^2}p C_{\textup{GNS}}  (\sqrt{K}+\eps)^{p-2}  < \frac{1-\eta}2.
		\end{align}
		
		\noi
		By collecting \eqref{var4_1}, \eqref{var5a}, \eqref{eta}
		and Lemma \ref{LEM:bounds},
		we get
		\begin{align*}
			\begin{split} 
				- & \log \E_\mu \bigg[ \exp \Big(\al R_p(u) \cdot  \ind_{\{\|u\|^2_{L^2(\R^d)}\le K\}} \Big)\bigg] \\
				& \ge \inf_{\theta \in \mathbb{H}_a} \E \bigg[  - \frac{1-\eta}{2} \| I(\theta)(1)\|_{\H^1 (\R^d)}^2 - C_\eps \| \P_{ N_\eps} Y(1)\|_{\H^1 (\R^d)}^2  -  C_\eps R_p(Y(1)) + \frac12 \int_0^1 
				\| \theta (t) \|_{L^2(\R^d)}^2 dt  \bigg]\\
				& \ge  \inf_{\theta \in \mathbb{H}_a}  \E \bigg[ -  C_\eps R_p(P_{N_\eps}^\perp Y(1)) - C_\eps \| \P_{N_\eps} Y(1)\|_{\H^1 (\R^d)}^2  + \frac{\eta}{2} \int_0^1 
				\| \theta (t) \|_{L^2(\R^d)}^2 dt \bigg] \\
				& \ge  \E \Big[ -  C_\eps R_p(P_{N_\eps}^\perp Y(1)) - C_\eps \| \P_{N_\eps} Y(1)\|_{\H^1 (\R^d)}^2  \Big] \\
				& \ge   -  C_\eps  - 
				C_\eps \E \big[  \| \P_{N_\eps} Y(1)\|_{\H^1 (\R^d)}^2  \big].
			\end{split}
		\end{align*}
		
		\noi
		We remark that $Y(1) \notin \H^1 (\R)$ almost surely.
		Therefore, to prove \eqref{var3_1},
		there remains to show that
		\begin{align}
			\label{Hsbound}
			\E \big[ \| \P_{N_\eps} Y(1)\|_{\H^1 (\R^d)}^2 \big] 
			< \infty,
		\end{align}
		
		\noi
		where $N_\eps$ is a random variable given by \eqref{Neps}.
		
		Noting that $Y(1)$ is a mean-zero random variable,
		we may decompose $\O$ (by ignoring a zero-measure set) as
		\begin{align}
			\label{decom}
			\O = \bigcup_{N \ge 1} \O_N,
		\end{align}
		
		\noi
		where
		\begin{align}
			\label{ON}
			\Omega_N = \big\{ \o\in \O : N_\eps (\o) = N \big\}.
		\end{align}
		
		\noi
		By \eqref{decom} and H\"older's  inequality, 
		we have
		\begin{align}
			\label{Hsbound1}
			\begin{split}
				\E \big[ \| \P_{N_\eps} Y(1)\|_{\H^1 (\R^d)}^2 \big]  
				& \le \sum_{N\ge 1} \E  \big[ \| \P_{N} Y(1)\|_{\H^1 (\R^d)}^2 \cdot \ind_{\O_N} \big] \\
				& \le \sum_{N\ge 1} \ld_N^{2} \E \big[ \| \P_{ N} Y(1)\|_{L^2 (\R^d)}^2 \cdot \ind_{\O_N} \big]\\
				& \le \sum_{N\ge 1} \ld_N^{2} \Big( \E \big[ \| Y(1)\|_{L^2 (\R^d)}^4 \big] \Big)^{\frac12} \cdot \PP({\O_N})^{\frac12}
				\\
				& \le C \sum_{N\ge 1} \ld_N^{2}  \PP({\O_N})^{\frac12}.
			\end{split}
		\end{align}
		
		\noi
		By using Corollary \ref{COR:CLR1}, 
		we also  have 
		\begin{align}
			\label{L2bound}
			\E \big[ \| \P^\perp_{\frac{N}4} Y(1) \|_{L^2 (\R^d)}^2  \big] = \sum_{n = \left[\frac{N}4\right]+1}^\infty \ld_n^{-2} \les \ld_{\left[\frac{N}4\right]}^{-1 + \frac2s}.
		\end{align}
		
		\noi
		It then follows from \eqref{Neps}, \eqref{ON},
		H\"older's  inequality, Lemma \ref{LEM:Fer}
		and \eqref{L2bound}, that
		\begin{align}
			\label{Hsbound2}
			\begin{split}
				\PP(\O_N) & \le \PP\big( \big\{ \| \P^\perp_{\frac{N}4} Y(1) \|_{L^2(\R^d)} > \eps \big\} \big) \\
				& \les \exp \bigg\{-c \bigg( \frac{\eps}{\E \big( \| \P^\perp_{\frac{N}4} Y(1) \|_{L^{2} (\R^d)}  \big) }\bigg)^2 \bigg\} \\
				& \les \exp \bigg\{- \frac{c \eps^2}{\E \big[ \| \P^\perp_{\frac{N}4} Y(1) \|_{L^{2} (\R^d)}^2  \big] } \bigg\} \\
				& \les e^{-\tilde c \eps^2 \ld_{\left[\frac{N}4\right]}^{-\frac2s + 1}},
			\end{split}
		\end{align}
		
		\noi
		where $c$ and $\tilde c$ are constants.
		By collecting \eqref{Hsbound1}, \eqref{Hsbound2} and Lemma \ref{LEM:asym},
		we conclude that
		\begin{align*}
			\begin{split}
				\E \big[ \| \P_{N_\eps} Y(1)\|_{\H^1 (\R^d)}^2 \big]  
				\le  \sum_{N\ge 1} \ld_N^{2}  e^{- \frac{\tilde c}2 \eps^2 \ld_{\left[\frac{N}4\right]}^{-\frac2s +1}} < \infty,
			\end{split}
		\end{align*}
		
		\noi
		where we used $s > 2$,
		which finishes the proof of \eqref{Hsbound},
		and thus \eqref{var3_1} in the critical case.
		%
		%
	\end{proof}
	
	\medskip
	
	\subsection{Non-normalizability}
	\label{SEC:non2}
	In this subsection,
	we prove the rest of Theorem \ref{THM:main1},
	i.e. the non-integrability part of (ii) and (iii): 
	\begin{lemma}[\textbf{Divergence of the partition function}]\label{lem:blow super}\mbox{}\\
		Let $s>2$ and assume either of the following conditions
		\begin{align}
			\begin{split}
				&\textup{(i) critical nonlinearity: } p = 2+\frac4d \text{ and } \al^{\frac{d}2} K > \| Q\|^2_{L^2(\R^d)};\\
				&\textup{(ii) supercritical nonlinearity: } p > 2+\frac4d \text{ and any } \al, K > 0.
			\end{split}
			\label{conditions}
		\end{align}
		where $Q$ is an optimizer of the GNS inequality. Then
		\begin{align}
			\label{part}
			\mathcal Z_{K} =\E_{\mu}
			\Big[\exp(\al {R_p(u)})
			\ind_{\{\|u\|^2_{L^2(\R^d)}\le K\}}\Big] = \infty.
		\end{align}
	\end{lemma}
	
	We construct a test drift giving a $-\infty$ upper bound in the Bou\'e-Dupuis variational principle, as sketched in Section~\ref{sec:BouDup1}. We use the following blow-up profiles:
	
	\begin{lemma}[\textbf{Blow-up profiles}]
		\label{LEM:soliton2}\mbox{}\\
		Assume \eqref{conditions} holds. Let 
		\begin{align}
			\label{W}
			W_\rho = \be \rho^{-\frac{d}2} Q(\rho^{-1} x).
		\end{align}
		Then, in the limit $\rho \to 0$ we have
		\begin{align}
			\label{soliton}
			\begin{split}
				\textup{ (i) } & H (W_\rho)  \le - A_1 \rho^{-\frac{dp}2 +d},\\
				\textup{ (ii) } &\| W_\rho\|_{L^p (\R^d)}^p \le A_2 \rho^{-\frac{dp}2 +d},\\
				\textup{ (iii) } &\| W_\rho \|^2_{L^2(\R^d)} \le K - \eta,
			\end{split}
		\end{align} 
		\noindent where $H$ is the Hamiltonian functional given in \eqref{Hamil}, and $A_1, A_2, \eta > 0$ are constants independent of $\rho > 0$.
	\end{lemma}

	\begin{proof}
		There exists $\be > 0$ such that 
		\[
		K > \| W_\rho \|^2_{L^2(\R^d)} = \be^2 \| Q\|^2_{L^2(\R^d)}.
		\]
		
		\noi 
		In fact, since $\al^{\frac{d}2}K > \| Q\|^2_{L^2(\R^d)}$ in the critical case, we take $\beta$ so that 
		\begin{align} \label{choi-be}
			\beta^2 \al^{\frac{d}2} >1;
		\end{align} 
		while in the supercritical case, we can take $\be$ small. By choosing $\eta \in \left(0, \al^{\frac{d}2}K - \| W_\rho \|^2_{L^2(\R^d)}\right]$, we have \eqref{soliton}-(iii).
		
		The rest follows from a similar computation as in \cite[Lemma 3.4]{LW22} by taking into account \eqref{choi-be} and
		\begin{align}
			\lim_{\rho \to 0} \int_{\R^d} |x|^s |W_\rho (x)|^2 dx = 0.   
			\label{poten_int}
		\end{align} 
		Note that \eqref{poten_int} comes from
		\[
		\int_{\R^d} |x|^s |W_\rho (x)|^2 dx  = \be^2 \rho^s \int_{\R^d} |x|^s |Q ( x)|^2 dx
		\]
		and the exponential decay at infinity of $Q$. See \cite[Proposition 3.1]{OST22}.
	\end{proof}

	\begin{remark}[Refined blow-up profiles]\rm \mbox{}\\
		We could use the blow-up profiles from Lemma \ref{LEM:soliton} to show \eqref{part} under \eqref{conditions} (ii), and (i) with $\al^{\frac{d}2} K \gg \| Q\|^2_{L^2(\R^d)}$. But this cannot explain the sharp phase transition \eqref{conditions} (i) when $\al^{\frac{d}2} K \sim \| Q\|^2_{L^2(\R^d)}$. For this, we build a sequence of scalings of ground state $Q$, namely $\{W_\rho\}$, which are new blow-up profiles that accurately capture the critical mass.
	\end{remark}

	\medskip
	
	We construct a series of drift terms as follows. Let $\rho >0$, $M=\rho^{-1}$, $W_\rho$ be as in Lemma \ref{LEM:soliton2} and $Z_M(t)$ as in Definition~\ref{def:gaussapprox}. We set 
	\begin{align}
		\label{theta}
		\theta_\rho (t) := -  \L^{\frac12} \frac{d}{dt} Z_M (t) +  \L^{\frac12}  W_\rho,
	\end{align}
	\noi 
	where $\rho \ll 1$ and $ \ld_M \sim \rho^{-1}$.
	
	From \eqref{theta}, we have
	\begin{align}
		\label{Itheta}
		\begin{split}
			I (\theta) (1) & =  \int_0^1   \L^{-\frac12} \theta (t) dt \\
			& = \int_0^1 ( W_\rho - \frac{d}{dt} Z_M (t)) dt\\
			& = W_\rho - Z_{M}(1).
		\end{split}
	\end{align}
	
	We need the following properties of the approximate Brownian motion.
	
	\begin{lemma}[\textbf{Approximating the Brownian motion, superharmonic case}]
		\label{LEM:appro2}\mbox{}\\
		Given $s > 2$ and a dyadic number $M \sim \rho^{-1} \gg 1$, let $Z_M(t)$ be as in Definition~\ref{def:gaussapprox}. The following holds:
		\begin{align}
			\E& \big[ \|Y(1) - Z_M\|_{L^2 ({\R^d})}^2 \big] \sim \ld_M^{\frac2s - 1}, \label{L2_1} \\
			\E & \big[\| Z_M  (1) - Y (1) \|_{L^p(\R^d)}^p \big] \les \left(\ld_M^{\frac2s - 1}\right)^{\frac{p}2}  \text{ for } p\ge 1, \label{Lp}\\
			\E & \bigg[\int_0^1 \Big\| \frac d {d\tau} Z_M(\tau) \Big\|^2_{\H^1(\R^d)}d\tau \bigg] \les  \ld_M^{\frac2s}, \label{dZL2}
		\end{align}
		
		\noi
		for any $M \gg 1$.
	\end{lemma}
	
	\begin{proof}
		Let 
		\begin{align}
			X_n(t)=\wt Y (n, t)- \wt Z_{M}(n, t), 
			\quad 0 < n \le M.
			\label{ZZ1_1} 
		\end{align}
		Then,  from  \eqref{ZZZ}, 
		we see that $X_n(t)$ satisfies 
		the following stochastic differential equation:
		\begin{align*}
			\begin{cases}
				dX_n(t)=- c \ld_n^{-1} \ld_M  X_n(t) dt + \ld_n^{-1} dB_n(t)\\
				X_n(0)=0
			\end{cases}
		\end{align*}	
		
		\noi
		for $0 < n \le M$, where $c \gg 1$ is a constant.
		By solving this stochastic differential equation, we have
		\begin{align}
			X_n(t)= \ld_n^{-1} \int_0^t e^{- c  \ld_n^{-1} \ld_M  (t-\tau)}dB_n(\tau).
			\label{ZZ2_1}
		\end{align}
		
		\noi
		Then, from \eqref{ZZ1_1} and \eqref{ZZ2_1}, we have 
		\begin{align}
			\wt Z_{M}(n, t)= \wt Y (n, t)- \ld_n^{-1} \int_0^t  e^{-   c \ld_n^{-1} \ld_M  (t-\tau)} dB_n(\tau)
			\label{SDE12}
		\end{align}
		
		\noi
		for $n \le M$. Hence, from \eqref{SDE12}, the independence of $\{B_n \}_{n \in \N}$, 
		Ito's isometry and Corollary \ref{COR:CLR} with $p = 1$ and $p = \frac12$, we have
		\begin{align}
			\begin{split}
				\E \big[ \|Z_M\|_{L^2 ({\R^d})}^2 \big]&=\sum_{n \le M} \bigg( \E \big[  | \wt Y (n) |^2  \big]
				- 2 \ld_n^{- 2} \int_0^1 e^{-  c \ld_n^{-1} \ld_M    (1-\tau)}d\tau \\
				& \hphantom{XXXXX}
				+ \ld_n^{- 2} \int_0^1 e^{-2 c \ld_n^{-1} \ld_M   (1-\tau)}d\tau \bigg)\\
				& \sim 1 + O\Big( c^{-1}\sum_{n \le M } \ld_n^{-1} \ld_M^{-1}  \Big)\\
				& \sim 1,
			\end{split}
			\label{ODE3}
		\end{align}
		
		\noi
		for any $M\gg 1$, $c \gg 1$ and $s >2$. 
		Similarly,  
		we have 
		\begin{align}
			\begin{split}
				\E \big[ \|Y(1) - Z_M\|_{L^2 ({\R^d})}^2 \big] & = \sum_{n \le M}  \E \big[  | X_n (t) |^2  \big] + \sum_{n > M}  \E \big[  | \wt Y (n) |^2  \big] \\
				& \sim  \sum_{n \le M}  \ld_n^{- 2} \int_0^1 e^{-2 c \ld_n^{-1} \ld_M   (1-\tau)}d\tau  + O\Big(  \sum_{n > M } \ld_n^{-2}    \Big)\\
				& \sim O\Big( c^{-1}\sum_{n \le M } \ld_n^{-1} \ld_M^{-1}  \Big) + \ld_M^{\frac2s - 1} \\
				& \sim \ld_M^{\frac2s - 1} ,
			\end{split}
			\label{ODE31}
		\end{align}
		
		\noi 
		where we used Corollaries \ref{COR:CLR} and \ref{COR:CLR1}. This proves \eqref{L2_1}.
		Then, \eqref{Lp} follows from \eqref{L2_1} and the Khintchine inequality (see e.g., \cite[Lemma 4.2]{BT08a}).
		
		By the $L^2$ orthogonality of $\{e_n\}_{n\in\N}$,  \eqref{SDE12}, \eqref{NRZ0} and  
		proceeding as in \eqref{ZZ3}, we have
		\begin{align*}
			\E\bigg[  2 \Re \int_{{\R^d}} Y_N & \cj{Z_M} dx - \int_{{\R^d}} |Z_M|^2 dx   \bigg]
			=\E \bigg[ 2 \Re \sum_{n \le M }\wt Y_N(n)  \cj{ \wt Z_M(n)} -\sum_{n \le M }|  \wt Z_M(n) |^2    \bigg]\\
			&=\E \bigg[ \sum_{n \le M } | \wt Z_M(n)  |^2+ \sum_{n \le M } \Re \bigg( 2 \ld_n^{-1}  \int_0^1 e^{- c \ld_n^{-1} \ld_M    (1-\tau) }dB_n(\tau) \bigg)  \cj{\wt Z_M(n)}   \bigg]\\
			&\sim \ld_M^{\frac2s-1} +O\Big( c^{-1} \sum_{n \le M } \ld_n^{-1} \ld_M^{-1}  \Big)\\
			&\sim \ld_M^{\frac2s-1} 
		\end{align*} 
		
		\noi 
		for any $N\ge M\gg 1$ and $c \gg 1$.
		
		Similarly, from \eqref{ZZZ}, \eqref{ZZ1_1}, \eqref{ZZ2_1}, and Ito's isometry, we have 
		\begin{align*}
			\E\bigg[\int_0^1 \Big\| \frac d {d\tau} Z_M(\tau) \Big\|^2_{\H^1(\R^d)}d\tau\bigg] &= \ld_M \E\bigg[\int_0^1 \big\| \P_{ M}(Y_N(\tau)) - Z_M(\tau) \big\|^2_{L^2(\R^d)}ds\bigg] \\
			&=  \ld_M \E\bigg[ \int_0^1  \sum_{n \le M} |X_n(\tau)|^2   d\tau\bigg]\\
			&=  \ld_M \sum_{n \le M} \int_0^1\E\big[|X_n(\tau)|^2\big] d\tau \\
			&= \ld_M \sum_{n \le M} \ld_n^{-2} \int_0^1 \int_0^\tau e^{-2 c \ld_n^{-1} \ld_M  (\tau-\tau')} d \tau' d\tau \\
			& \les \ld_M \sum_{n \le M } \ld_n^{-1} \ld_M^{-1} \\
			&\les  \ld_M^{\frac{2}s }, 
		\end{align*}
		
		\noi
		yielding  \eqref{dZL2}.
		This completes the proof of Lemma~\ref{LEM:appro2}.
	\end{proof}

	\medskip
	
	Now we are ready to prove the rest of Theorem \ref{THM:main}.
	
	\begin{proof}[Proof of Theorem \ref{THM:main1} - the second half of \textup{(ii) $and$ (iii)}]
		We shall prove \eqref{part} under conditions \eqref{conditions}.
		Observing that
		\begin{align*}
			\E_{\mu}  \Big[ \exp (\al R_p(u)) \cdot  \ind_{\{\|u\|^2_{L^2(\R^d)}\ge K\}}\Big]
			\ge \E_{\mu}  \Big[ \exp \Big(\al R_p(u) \cdot  \ind_{\{\|u\|^2_{L^2(\R^d)}\le K\}} \Big)\Big] - 1,
		\end{align*}
		
		\noi
		then \eqref{part} follows from
		\begin{align}
			\label{var10}
			\E_{\mu}  \bigg[ \exp \Big(\al R_p(u) \cdot  \ind_{\{\|u\|^2_{L^2(\R^d)}\le K\}} \Big)\bigg] = \infty.
		\end{align}
		
		We apply Lemma \ref{LEM:var}, together with \eqref{theta}  
		and \eqref{Itheta}, to get 
		\begin{align}
			\label{var13}
			\begin{split}
				- & \log  \E_{\mu}  \bigg[ \exp \Big(\al R_p (u) \cdot  \ind_{\{\|u\|^2_{L^2(\R^d)}\le K\}} \Big)\bigg]\\
				& = \inf_{\theta \in \mathbb H_a} \E \bigg[ \Big( - \al R_p (Y(1) + I (\theta) (1))  \cdot \ind_{\{\|Y(1) + I (\theta) (1)\|^2_{L^2 (\R^d)}\le K\}} + \frac12 \int_0^1 \| \theta (t)\|_{L^2 (\R^d)}^2\Big)    \bigg] \\
				&   \le \inf_{0 < \rho \ll 1} \E \bigg[ \Big( - \al R_p ( Y(1) - Z_M (1) + W_\rho )  \cdot \ind_{\{\| Y(1) - Z_M (1) + W_\rho \|^2_{L^2 (\R^d)}\le K\}}\\
				& \hphantom{XXXXXX}
				+ \frac12 \int_0^1 \Big\| -  \frac{d}{dt} Z_M (t) +    W_\rho \Big\|_{\H^1  (\R^d)}^2\Big) dt   \bigg]\\
				& \le \inf_{0 < \rho \ll 1} \E \bigg[ 
				\Big(- \al R_p (W_\rho) + \frac12 \| W_\rho\|_{\H^1  (\R^d)}^2 \Big)\\ 
				& \hphantom{XX}+  \Big( \al R_p ( W_\rho) - \al R_p  ( Y(1) - Z_M (1) + W_\rho ) \Big)  \cdot \ind_{\{\| Y(1) - Z_M (1) + W_\rho \|^2_{L^2 (\R^d)}\le K\}}\\
				& \hphantom{XX} + \al R_p(  W_\rho) \cdot \ind_{\{\| Y(1) - Z_M (1) + W_\rho \|^2_{L^2 (\R^d)} > K\}} \\
				& \hphantom{XX}
				+ \frac12 \int_0^1 \Big\| -  \frac{d}{dt} Z_M (t) \Big\|_{\H^1  (\R^d)}^2 - 2 \Big\langle   \frac{d}{dt} Z_M (t) ,   W_\rho \Big\rangle_{\H^1 (\R^d)}  dt   \bigg] \\
				& = \inf_{0 < \rho \ll 1} ( \textup{A + B  + C + D}),
			\end{split}
		\end{align}
		by inserting the test drift~\eqref{theta} in the Bou\'e-Dupuis variational principle. In what follows, we consider the terms in the right-hand side one by one.
		
		For the term (A), 
		from Lemma \ref{LEM:soliton2}, we have
		\begin{align}
			\label{term1}
			\textup{A}  = - \al R_p (W_\rho) + \frac12 \| W_\rho\|_{\H^1 (\R^d)}^2 = H (W_\rho) \les - \rho^{-\frac{dp}2 +d},
		\end{align}
		
		\noi
		where the Hamiltonian $H$ is given in \eqref{Hamil}.
		
		For the term (B),
		by using the mean value theorem we see that
		\[
		\begin{split}
			\int_{\R^d} \big( & |W_\rho|^p - |Y(1) - Z_M (1) + W_\rho|^p \big) dx \\
			& \les 
			\int_{\R^d} \big( |Y(1) - Z_M (1) |^p + |Y(1) - Z_M (1)| |  W_\rho|^{p-1} \big) dx ,
		\end{split}
		\]
		
		\noi
		which together with Lemma \ref{LEM:appro2}, Lemma \ref{LEM:soliton2}, and Young's inequality, gives
		\begin{align}
			\label{term3}
			\begin{split}
				\textup{B} = \E & \bigg[ \Big( \al R_p (W_\rho) - \al R_p ( Y(1) - Z_M (1) + W_\rho ) \Big)  \cdot \ind_{\{\| Y(1) - Z_M (1) + W_\rho \|^2_{L^2 (\R^d)}\le K\}}\bigg] \\
				& \les \int_{\R^d} \big( \E \big[ |Y(1) - Z_M (1) |^p \big] + \E \big[ |Y(1) - Z_M (1)|\big] | W_\rho|^{p-1} \big) dx\\
				& \le C_\eps \E \big[ \| Y(1) - Z_M (1) \|_{L^p(\R^d)}^p \big]  + \eps \| W_\rho \|_{L^p(\R^d)}^p\\
				& \le C_\eps \left(\ld_M^{\frac2s - 1}\right)^{\frac{p}2} + \eps \rho^{-\frac{dp}2 +d}\\ 
				& \le 2\eps \rho^{-\frac{dp}2 +d},
			\end{split}
		\end{align}
		
		\noi
		provided $M\gg 1$.
		Now we turn to the term (C),
		by using Chebyshev's inequality,  
		we have
		\begin{align}
			\begin{split}
				\textup{C} &= \E \big[ \al R_p(W_\rho) \cdot \ind_{\{\| Y(1) - Z_M (1) + W_\rho \|^2_{L^2 (\R^d)} > K\}} \big] \\
				& \le \al R_p(W_\rho) \cdot \E \bigg[ \ind_{\{\| Y(1) - Z_M (1)\|_{L^2 (\R^d)}  > \sqrt{K} - \| W_\rho \|_{L^2 (\R^d)}\}} \bigg]\\
				& \le \al R_p(W_\rho) \frac{\E [\| Y(1) - Z_M (1)\|_{L^2 (\R^d)}^2] }{\left(\sqrt{K} - \| W_\rho \|_{L^2 (\R^d)}\right)^2}\\ 
				& \les \eta^{-2} \rho^{-\frac{dp}2 +d} \ld_M^{\frac2s - 1} = o(1) \rho^{-\frac{dp}2 +d} .
			\end{split}
			\label{term4}
		\end{align}
		
		\noi
		where in the last step we use the fact $\ld_M \to \infty$ as $M\to \infty$ and $s > 2$.
		For term (D),
		from \eqref{dZL2},
		we have
		\begin{align}
			\label{term5}
			\begin{split}
				\textup{D}  = \frac12 \int_0^1 \E \bigg[  \Big\| -  \frac{d}{dt} Z_M (t) \Big\|_{\H^1  (\R^d)}^2 \bigg] dt  \le C \ld_M^{\frac2s} \ll \rho^{-\frac2s},
			\end{split}
		\end{align}
		
		\noi 
		where we used the fact that $\ld_M \ll \rho^{-1}$.
		By collecting estimates  \eqref{term1}, \eqref{term3}, \eqref{term4} and \eqref{term5},
		we conclude that 
		\begin{align}
			\label{abcd}
			\textup{A + B + C + D} \les - \rho^{-\frac{dp}2 +d} + c \rho^{-\frac2s}  \to - \infty,
		\end{align}
		
		\noi
		provided $c\ll 1$, $p \ge 6$ and $s > 2$.
		
		Finally,
		the desired estimate \eqref{var10} follows from \eqref{var13} and \eqref{abcd}.
		We thus finish the proof of Theorem \ref{THM:main1}.
	\end{proof}
	
	\newpage

	\appendix
	
	\section{Reminder on coherent states and Husimi functions}
	\label{SEC:Weyl}
	
	Let $f \in C^\infty_0(\R)$ be an odd function satisfying $\|f\|^2_{L^2(\R)}=1$. For $\hbar\in (0,1]$ and $x, p \in \R$, we define the coherent state
	\[
	f^{\hbar}_{x,p}(y) := \hbar^{-1/4} f\left(\frac{y-x}{\sqrt{\hbar}}\right) e^{i\frac{py}{\hbar}}.
	\]
	Denote
	\[
	f^\hbar(y) := \hbar^{-1/4} f\left(\frac{y}{\sqrt{\hbar}}\right), \quad g^\hbar(q) := \hbar^{-1/4} \hat{f}\left(\frac{q}{\sqrt{\hbar}}\right),
	\]
	where 
	\[
	\hat{f}(q):= \frac{1}{\sqrt{2\pi}} \int e^{-i q y} f(y) dy
	\]
	is the standard Fourier transform. We also define the semiclassical Fourier transform
	\[
	\mathcal F_{\hbar}[f](q) :=\frac{1}{\sqrt{2\pi \hbar}} \int e^{-i\frac{qy}{\hbar}} f(y) dy.
	\]
	
	We have the following observation (see e.g., \cite[Chapter 12]{LL01} or \cite[Section 2.1]{FLS18}).
	\begin{lemma}[\textbf{Coherent state formalism}]\label{lem:coherent}\mbox{}\\
		\noindent$\bullet$ (Plancherel identity)
		$$
		\jb{\varphi, \psi}_{L^2} = \jb{\mathcal F_{\hbar}\varphi, \mathcal F_{\hbar}\psi}_{L^2} \text{ for all } \varphi, \psi \in L^2(\R).
		$$

		\noindent$\bullet$ (Localization of coherent states)
		\begin{align*}
			|f^{\hbar}_{x,p}(y)|^2 &= \hbar^{-1/2}\left|f\left(\frac{y-x}{\sqrt{\hbar}}\right)\right|^2 \rightharpoonup \delta_{\{y=x\}} \text{ as } \hbar \to 0, \\
			|\mathcal F_{\hbar}[f^{\hbar}_{x,p}](q)|^2 &= h^{-1/2}\left|\hat{f}\left(\frac{q-p}{\sqrt{\hbar}}\right)\right|^2 \rightharpoonup \delta_{\{q=p\}} \text{ as } \hbar \to 0.			
		\end{align*}
		\noindent$\bullet$ (Resolution of identity)
		\begin{align} \label{reso-iden}
			\frac{1}{\sqrt{2\pi \hbar}} \iint |f^{\hbar}_{x,p}\rangle \langle f^{\hbar}_{x,p}| dx dp = \ind_{L^2(\R)}.
		\end{align}
	\end{lemma}
	
	We define the Husimi function associated to a non-negative trace-class operator $\gamma_{\hbar}$ as
	\[
	m_{\hbar}(x,p) := \jb{f^{\hbar}_{x,p}, \gamma_{\hbar}f^{\hbar}_{x,p}}
	\]
	Here are some of its' basic properties, see again (see e.g., \cite[Chapter 12]{LL01} or \cite[Section 2.1]{FLS18}).
	
	\begin{lemma}[\textbf{Properties of Husimi functions}]\label{lem:Husimi}\mbox{}\\
		Write the spectral decomposition of $\gamma_{\hbar}$ as 
		$$ \gamma_\hbar = \sum_{n\geq 1} \mu_n^{\hbar} |u^{\hbar}_n \rangle \langle u^{\hbar}_n|.$$
		We have 
		\begin{align*}
			0\leq m_{\hbar}(x,p) &\leq 1, \quad \forall x,p \in \R, \\
			\frac{1}{2\pi\hbar}\iint m_{\hbar}(x,p) dxdp &= {\rm Tr}[\gamma_\hbar], \\
			\frac{1}{2\pi\hbar} \int m_{\hbar}(x,p) dp &= \rho_{\hbar} \ast |f^{\hbar}|^2 (x), \\
			\frac{1}{2\pi\hbar} \int m_{\hbar}(x,p) dx &= t_{\hbar} \ast |g^{\hbar}|^2(p),
		\end{align*}
		where 
		\[
		\rho_{\hbar}(x) = \sum_{n\geq 1} \mu_n^{\hbar} |u^{\hbar}_n(x)|^2, \quad t_{\hbar}(p) = \sum_{n\geq 1} \mu_n^{\hbar} |\mathcal F_{\hbar}[u^{\hbar}_n](p)|^2
		\]
		are the density and momentum functions associated to $\gamma_{\hbar}$. In addition, we have
		\begin{align}\label{momen-densi-relation}
			\hbar \int \rho_\hbar(x) dx = \hbar \int t_{\hbar}(p) dp = \hbar {\rm Tr}[\gamma_\hbar] = \frac{1}{2\pi}\iint m_{\hbar}(x,p)dxdp.
		\end{align}
	\end{lemma}

	\section{Fractional Schr\"odinger operator}
	
	In this appendix,
	we prove some estimates for the fractional Schr\"odinger operator,
	which enable us to prove normalizability of Gibbs measure associated with fractional Schr\"odinger operators. 
	
	Let $H_\al$ be the fractional Schr\"odinger operator
	\[
	H_\al = (-\Delta)^{\al} + V(x)
	\]
	
	\noi 
	defined on $\R^d$, where $V: \R^d \to \R_+$ is a trapping potential, i.e., $V(x) \to +\infty$ as $|x| \to \infty$. In particular, we are interested in the anharmonic potential $V(x) = |x|^s$ with $s > 0$.
	The operator $H_\al$ has a sequence of eigenvalues $\ld_n^2$ with 
	\[
	0 < \ld_0 \le \ld_1 \le \cdots \le \ld_n \to \infty 
	\]
	
	\noi 
	and the corresponding eigenfunctions $(e_n)_{n \ge 0}$ form an orthonormal basis of $L^2 (\R^d)$. We emphasis that we do not assume the radial condition here. Let us start with the following result.
	
	\begin{proposition} [\bf Schatten-norm bounds for the resolvent] 
		\label{PROP:LS} \mbox{} \\
		Let $d\ge 1$, $\al>0$, $s>0$ and $\g > \frac{d}{2\al}$.
		Then we have
		\begin{align}
			\label{LT}
			{\rm Tr} [H_\al^{-\g}] = \sum_{n=0}^\infty \ld_n^{-2\g} \le C \int_{\R^d} (V(x))^{\frac{d}{2\al} - \g} dx.
		\end{align}
		
		\noi 
		In particular, if $V(x) = |x|^s$, then
		\[
		{\rm Tr} [H_\al^{-\g}] < \infty
		\]
		
		\noi 
		provided $\g > \frac{d}{2\al} + \frac{d}{s}$.
	\end{proposition}

	\begin{proof}
		Let $G(t,x)$ be the fundamental solution to the fractional heat equation
		\[
		\partial_t u +(-\Delta)^\al u = 0 
		\]
		
		\noi 
		such that the solution to the above equation with initial data $u(0) = f$ can be written as
		\[
		u(t,x) = [e^{-t(-\Delta)^\al} f](x) : = [G(t,\cdot) \ast f](x).
		\]
		
		\noi 
		In particular, we have
		\begin{align}
			\label{Green1}
			G(t,x) = t^{-\frac{d}{2\al}} G(1,t^{-\frac{1}{2\al}} x).
		\end{align}
		
		By Trotter's formula, we have
		\begin{align}
			\begin{split}
				{\rm Tr} [e^{-t ((-\Delta)^\al +V)}] & = \lim_n {\rm Tr} [(e^{-\frac{t}n (-\Delta)^\al}e^{-\frac{t}n V})^n] \\
				& = \lim_n  \int_{(\R^{d})^n} G\Big(\frac{t}n,x-x_1\Big) e^{-\frac{t}n V(x_1)} G\Big(\frac{t}n,x_1-x_2\Big) e^{-\frac{t}n V(x_2)} \\
				& \hphantom{XXXXXX}  \cdots G\Big(\frac{t}n,x_{n-1}-x\Big) e^{-\frac{t}n V(x )}dx dx_1 \cdots dx_{n-1} \\
				& = \lim_n  \int_{(\R^{d})^n} \prod_{j=0}^{n-1} G\Big(\frac{t}n,x_j-x_{j+1}\Big) e^{-\frac{t}n \sum_{k=0}^{n-1} V(x_k)} dx_0 dx_1 \cdots dx_{n-1} \\
				& \le \lim_n  \frac1n  \sum_{k=0}^{n-1} \int_{(\R^{d})^n} \prod_{j=0}^{n-1} G\Big(\frac{t}n,x_j-x_{j+1}\Big) e^{- t V(x_k)} dx_0 dx_1 \cdots dx_{n-1}\\
				& = \lim_n  \frac1n  \sum_{k=0}^{n-1} \int_{\R^d} G(t, x_k - x_k) e^{- t V(x_k)} dx_k \\
				& = G(t, 0) \int_{\R^d}  e^{- t V(x )} dx \\
				& = C_{\al} t^{-\frac{d}{2\al}} \int_{\R^d}  e^{- t V(x )} dx,
			\end{split}
			\label{Green2}
		\end{align}
		
		\noi 
		where in the last step we used \eqref{Green1} and 
		\[
		C_\al = c_d \int_{\R^d} e^{-|\xi|^{2\al}} d\xi.
		\]
		
		Then by using \eqref{Green2} and a similar argument as in \cite{DFLP06}, we get
		\begin{align}
			\begin{split}
				{\rm Tr } \big[ ((-\Delta)^\al + V)^{-\g} \big] & = \frac1{\Gamma (\g)} \int_0^\infty {\rm Tr} \big[ e^{-t ((-\Delta)^\al +V)} \big] t^{\g - 1} dt \\
				& \le C_\al \frac1{\Gamma (\g)} \int_0^\infty \int_{\R^d}  t^{-\frac{d}{2\al}} e^{tV(x)} t^{\g - 1} dt\\
				& = C_\al \frac{\G(\g - \frac{d}{2\al})}{\G(\g)}
				\int_{\R^d} ( V(x))^{\frac{d}{2\al} - \g} dx.
			\end{split}
		\end{align}
		
		Since $H_\al \ge \ld_0$, we have
		$$
		{\rm Tr}[H_\al^{-\g}] \leq 2^\g {\rm Tr}[(H_\al +\ld_0)^{-\g}],
		$$
		where $\ld_0>0$ is the first eigenvalue of $H_\al$. It follows that
		$$
		{\rm Tr}[H_\al^{-\g}] \leq C \int_{\R^d} (|x|^s + \ld_0)^{\frac{d}{2\al}-\g} dx <\infty
		$$
		provided $\g >\frac{d}{2\al} +\frac{d}{s}$. Thus we finish the proof.
	\end{proof}

	\begin{proposition}[\bf $L^p$ bounds for the resolvent's integral kernel] 
		\label{PROP:GreenLp} \mbox{} \\
		Let $d\ge 1$, $\al>\frac{d}{2}$, $s>0$ and $V(x) = |x|^s$.
		Then we have 
		\[
		\|H^{-1}_\al (x,x)\|_{L^p(\R^d)} = \bigg\|\sum_{n\in \N} \frac{e^2_n(x)}{\ld_n^2} \bigg\|_{L^p(\R^d)} < \infty
		\]
		
		\noi 
		provided 
		\[
		p > \max \bigg\{1, \frac{2\al d}{s(2\al - d)} \bigg\}.
		\]
	\end{proposition}

	\begin{proof}
		The following proof is similar to that of \cite{DR23}.  
		It suffices to show
		\begin{align}
			\label{G1}
			\int_{\R^d} H_\al^{-1}  (x,x) g^2(x) dx \les   \|g^2 \|_{L^q (\R^d)},
		\end{align}
		
		\noi 
		where $\frac1p + \frac1q = 1$, for all $g \ge 0$ and $g^{2} \in L^q({\R^d})$.
		With a slight abuse of notation, we still use $g$ to denote the multiplication operator of $g(x)$. 
		Then we can rewrite
		\begin{align}
			\label{G2}
			\int_{\R^d} H_\al^{-1}  (x,x) g^2(x) dx  = \textup{Tr} [g H_\al^{-1}  g]  = \big\|  H_\al^{-\frac12} P_N g\big\|_{\mathfrak S^2}^2,
		\end{align}
		
		\noi
		where $\mathfrak S^p$ is the $p$-Schatten space. 
		Then we apply the H\"older inequality in Schatten spaces to continue with
		\begin{align}
			\begin{split}
				\big\|  H_\al^{-\frac12}  g\big\|_{\mathfrak S^2}^2 & =  \big\|  H_\al^{\s-\frac12}  H_\al^{-\s} (1+ (-\Delta)^\al)^{ \s} (1+ (-\Delta)^\al)^{- \s}  g\big\|_{\mathfrak S^2}^2 \\
				& \le \big\|   H_\al^{\s-\frac12}P_N  \big\|_{\mathfrak S^{2p}}^2 \big\|  H_\al^{-\s} (1+ (-\Delta)^\al)^{ \s}\big\|_{\mathfrak S^\infty}^2 \big\|  (1+ (-\Delta)^\al)^{ -\s}  g\big\|_{\mathfrak S^{2q}}^2
			\end{split}
			\label{G3}
		\end{align}
		
		\noi 
		We shall estimate the three factors on the right-hand side of \eqref{G3} one by one.
		
		For the first factor in \eqref{G3},
		we have
		\begin{align}
			\begin{split}
				\big\|   H_\al^{\s-\frac12}  \big\|_{\mathfrak S^{2p}}^2 
				& = \Big( {\rm{Tr}} \big[ (  H_\al^{2\s-1} )^{p} \big] \Big)^{\frac1{p}}\\
				& = \bigg( \sum_{n\in \N}   \ld_n^{4\s p - 2p} \bigg)^{\frac1p} \\
				& < \infty
			\end{split}
			\label{G4}
		\end{align}
		
		\noi 
		provided $p - 2\s p > \frac{d}{2\al} + \frac{d}s $.
		
		We turn to the second factor in \eqref{G3}.
		Since $H_\al \ge (-\Delta)^\al + \ld_0 $
		with $\ld_0 > 0$ being the lowest eigenvalue of $H_\al$, we obtain that $\L \ge C (1 + (-\Delta)^\al)$ for some $C > 0$.
		We also note the operator monotonicity of $x \mapsto x^{2\s}$ for $\s < \frac12$ gives
		\[
		H_\al^{2\s} \ges (1+ (-\Delta)^\al)^{2\s}
		\]
		
		\noi 
		or 
		\[
		H_\al^{-\s} (1 + (-\Delta)^\al)^{2\s} H_\al^{-\s} \les 1.
		\]
		
		\noi 
		Therefore, we conclude that the operator $\L^{-\s} (1 + (-\Delta)^\al)^{2\s} \L^{-\s}$ is bounded for $0 < \s < \frac12$, i.e.
		\begin{align}
			\label{G5}
			\big\|  H_\al^{-\s} (1+ (-\Delta)^\al)^{ \s}\big\|_{\mathfrak S^\infty}^2 \les 1.
		\end{align}
		
		For the third factor in \eqref{G3},
		we apply the Kato-Seiler-Simon inequality to get
		\begin{align}
			\label{G6}
			\big\| (1+ (-\Delta)^\al)^{ - \s} g\big\|_{\mathfrak S^{2q}}^2 \le \| \jb{\xi}^{-2\al \s} \|_{L^{2q}({\R^d})}^2  \| g\|_{L^{2q} ({\R^d})}^2 \les \| g^2\|_{L^{q} ({\R^d})},
		\end{align}
		
		\noi 
		provided $4\al \s q > d$ and $1 \le q < \infty$.
		
		Finally, by collecting \eqref{G2}, \eqref{G3}, \eqref{G4}, \eqref{G5}, and \eqref{G6}, we arrive at
		\begin{align}
			\label{G7}
			\int_{\R^d} H_\al^{-1}  (x,x) g^2(x) dx < \infty
		\end{align} 
		
		\noi 
		provided 
		\begin{align}
			\label{G8}
			\left\{
			\begin{array}{rcl}
				p - 2\s p &>& \frac{d}{2\al} + \frac{d}s\\
				\frac1p + \frac1q &=& 1\\
				\s &<& \frac12\\
				4\al \s q &>& d\\
				1\le q &<& \infty.
			\end{array}
			\right.
		\end{align}
		
		\noi 
		Since $\al>\frac{d}{2}$, we choose $p > \frac{2\al d}{s(2\al - d)}$ and 
		\[
		\frac{d}{4\al q} < \s < \frac12 - \frac{1}{2p} \bigg( \frac{d}{2\al} + \frac{d}s \bigg)
		\]
		
		\noi 
		so that all the conditions in \eqref{G8} are satisfied.
		Thus we finish the proof.
	\end{proof}
	
	\begin{proposition}[\bf Weyl's law for fractional Schr\"odinger operators] \mbox{} \\
		Let $d\geq 1$, $\alpha>0$, and $s>0$. Then 
		\begin{align} \label{CLR-frac}
			N(H_\alpha, \Lambda) \sim \Lambda^{\frac{d}{2\alpha}+\frac{d}{s}} \quad \text{as } \Lambda \to \infty,
		\end{align}
		
		\noi
		where 
		$$
		N(H_\al, \Ld):= \#\{\ld_n^2 : \ld_n^2 \le \Ld\}.
		$$
	\end{proposition}
	
	The proof of this result follows the same argument as in Subsection \ref{subsec:Weyl}. The coherent state is now defined by
	\[
	f^{\hbar}_{x,p}(y):=\hbar^{-d/4} f\left(\frac{y-x}{\sqrt{\hbar}}\right) e^{i\frac{p\cdot y}{\hbar}}, \quad f^\hbar(y) = \hbar^{-d/4} f\left(\frac{y}{\sqrt{\hbar}}\right), \quad g^{\hbar}(q):=\hbar^{-d/4} \hat{f}\left(\frac{q}{\sqrt{\hbar}}\right)
	\] 
	for some function $f\in C^\infty_0(\R^d)$  satisfying $\|f\|_{L^2(\R^d)} =1$. The semiclassical Fourier transform is
	\[
	\mathcal F_{\hbar}[f](q):=\frac{1}{(2\pi\hbar)^{d/2}} \int e^{-i\frac{q\cdot y}{\hbar}} f(y) dy.
	\]
	The upper energy upper bound is proved by using the trial state
	\[
	\gamma^{\test} = \frac{1}{(2\pi\hbar)^d} \iint m_0(x,p) |f^{\hbar}_{x,p}\rangle \langle f^{\hbar}_{x,p}| dxdp
	\]
	with
	\begin{align} \label{m1-frac}
		m_0(x,p):= \ind_{\{|p|^{2\alpha}+|x|^s-1\leq 0\}}.
	\end{align}
	From the energy upper bound, we deduce the lower bound and the upper bound on the trace. As in Remark \ref{rem-rad-d-geq3}, we can prove that 
	\begin{align*} 
		\lim_{\hbar \to 0} \hbar^d {\rm Tr}[\gamma_\hbar] = \frac{1}{(2\pi)^d} \iint m_0(x,p) dxdp.
	\end{align*}
	Since most of the estimates are similar to the radial case, we omit the details.

	
	\begin{corollary}
		[\bf Behavior of truncated Schatten norms]
		\label{COR:LS} \mbox{} \\
		Let $d\ge 1$, $\al>\frac{d}{2}$, $s>0$ and $V(x) = |x|^s$. Then, we have
		\begin{align*} 
			{\rm Tr} [(\P_NH_\al)^{-\g}] = \sum_{n=0}^N \ld_n^{-2\g} \les \begin{cases}
				1 &\text{ if } \g > \frac{d}{2\al} + \frac{d}{s}, \\
				(\log \ld_N)^2 &\text{ if } \g = \frac{d}{2\al} + \frac{d}{s}, \\
				\ld_N^{-2\g + \frac{d}{\al} + \frac{d}{2s}} &\text{ if } \g <\frac{d}{2\al} +\frac{d}{s}.
			\end{cases}.
		\end{align*}
	\end{corollary}
	
	The proof is similar to that of Corollary \ref{COR:CLR}. We omit the details.
	
	\section*{Statements and Declarations}
	
	$\bullet$ {\bf Competing Interests:} The authors have no competing interests to declare that are relevant to the content of this article.
	
	$\bullet$ {\bf Data Availability Statement:} Not applicable. This study did not generate any datasets.
	\newpage
	
	\bibliographystyle{acm}
	\bibliography{biblio.bib}

\end{document}